\documentclass[a4paper,11pt]{article}
\usepackage{amsmath,amssymb,amsthm,graphicx,inputenc}
\usepackage{hyperref,url,enumitem,bm,esint,color,comment}
\usepackage[margin=30mm]{geometry}

\usepackage{cite}

\definecolor{rosso}{rgb}{0.85,0,0}
\definecolor{azzurro}{rgb}{0.13, 0.67, 0.8}

\def\anold #1{{#1}}
\def\an #1{{\color{rosso}#1}}
\def\an #1{#1}
\def\comm #1{{\color{blue}#1}}
\def\last #1{{\color{rosso}#1}}
\def\last #1{#1}

\def\mod #1{{\color{rosso}#1}}
\def\mod #1{#1}

\def\Accorpa #1#2 #3 {\gdef #1{\eqref{#2}--\eqref{#3}}%
  \wlog{}\wlog{\string #1 -> #2 - #3}\wlog{}}

\newcommand{\ov}[1]{\overline{#1}}

\renewcommand{\hat}[1]{\widehat{#1}}
\let\badeps\epsilon
\newcommand{\eps}{\varepsilon}

\newcommand{\pd}{\partial}

\newcommand{\xxx}{{x}}

\newcommand{\Lpanarh}{\mathcal L}
\newcommand{\Lpanarv}{\tilde{{\mathcal L}}}

\newcommand{\dt}{\partial_t}
\newcommand{\dn}{\partial_{\bnn}}
\def\bnn{{\boldsymbol n}}
\def\bnu{{\boldsymbol \nu}}

\def\norma #1{\mathopen \| #1\mathclose \|}
\def\<#1>{\mathopen\langle #1\mathclose\rangle}
\def\iO {\int_\Omega}
\def\intinf{\int_{-\infty}^{+ \infty}}

\renewcommand{\tilde}{\widetilde}

\newcommand{\0}{\bm{0}}

\newcommand{\ph}{{{\varphi}}}

\def\non{\notag}

\def\erre{{\mathbb{R}}}
\def\enne{{\mathbb{N}}}

\def\Vp{{(\Hx1)^*}}

\theoremstyle{plain}
\newtheorem{thm}{Theorem}[section]

\numberwithin{equation}{section}

\def\genspazio #1#2#3#4#5{#1^{#2}(#5,#4;#3)}
\def\spazio #1#2#3{\genspazio {#1}{#2}{#3}T0}

\def\L {\spazio L}
\def\H {\spazio H}

\def\Lx #1{L^{#1}(\Omega)}
\def\Hx #1{H^{#1}(\Omega)}

\def\lhs{left-hand side}
\def\rhs{right-hand side}

\def\jump#1{[#1]^+_-}
\def\Sp{{S_+}}
\def\Sm{{S_-}}
\def\Kp{{K_+}}
\def\Km{{K_-}}
\def\emb{{\hookrightarrow}}

\def\Lm{\anold{\lambda}_-}
\def\Lp{\anold{\lambda}_+}

\def\rad{Y}

\def\multibold #1{\def\arg{#1}%
  \ifx\arg\pto \let\next\relax
  \else
  \def\next{\expandafter
    \def\csname #1#1\endcsname{{\bf #1}}%
    \multibold}%
  \fi \next}

\def\pto{.}

\def\multimathbb #1{\def\arg{#1}%
  \ifx\arg\pto \let\next\relax
  \else
  \def\next{\expandafter
    \def\csname #1#1#1\endcsname{{\mathbb #1}}%
    \multimathbb}%
  \fi \next}

\def\multical #1{\def\arg{#1}%
  \ifx\arg\pto \let\next\relax
  \else
  \def\next{\expandafter
    \def\csname cal#1\endcsname{{\cal #1}}%
    \multical}%
  \fi \next}

\def\multimathop #1 {\def\arg{#1}%
  \ifx\arg\pto \let\next\relax
  \else
  \def\next{\expandafter
    \def\csname #1\endcsname{\mathop{\rm #1}\nolimits}%
    \multimathop}%
  \fi \next}

\multibold
qwertyuiopasdfghjklzxcvbnmQWERTYUIOPASDFGHJKLZXCVBNM.

\multimathbb
QWERTYUIOPASDFGHJKLZXCVBNM.

\multical
QWERTYUIOPASDFGHJKLZXCVBNM.

\multimathop
diag dist div dom mean meas sign supp .

\title{On a Cahn--Hilliard equation for the growth and division of 
	chemically active droplets modeling  protocells
	}
\author{Harald Garcke \footnotemark[1] \and Kei Fong Lam \footnotemark[2] \and Robert N\"urnberg \footnotemark[3] \and Andrea Signori \footnotemark[4]}
\date{ }

\begin{document}
	
\maketitle
%\footnote{Other possible titles:\\
%	On a Cahn--Hilliard model for growth an division of active droplets\\
%	Analysis and computation for models of chemically driven growth and division of droplets\\
%	On mathematical models for chemically active droplet models for protocells\\
%	Active droplet formation in  Cahn--Hilliard models with chemical reactions
%	}

\begin{abstract}
\noindent
The Cahn--Hilliard model with reaction terms can lead to situations in which no coarsening is taking place and, in contrast, growth and division of droplets occur which all do not grow larger than a certain size. 
%This phenomenon has been suggested as a model for protocells. 
%
\an{This phenomenon has been suggested as a model for protocells, and a model based on the modified Cahn--Hilliard equation has  been formulated. 
We introduce this equation and show the existence and uniqueness of solutions.}
%We introduce the modified Cahn--Hilliard equation and  show the existence and uniqueness of solutions. 
Then  formally matched asymptotic expansions are used to identify a sharp interface limit using a scaling of the reaction term which becomes singular when the interfacial thickness tends to zero.
We compute planar solutions and study their stability under non-planar perturbations. 
Numerical computations for the suggested model are used to validate the sharp interface asymptotics. In addition, the numerical simulations show  that the reaction terms lead to diverse phenomena such as growth and division of droplets in the obtained solutions\an{, as well as the formation of shell-like structures.}
\end{abstract}



\renewcommand{\thefootnote}{\fnsymbol{footnote}}
\footnotetext[1]{Fakult{\"a}t f\"ur Mathematik, Universit{\"a}t Regensburg, 93040 Regensburg, Germany
({\texttt harald.garcke@ur.de}).}
\footnotetext[2]{Department of Mathematics, Hong Kong Baptist University, Kowloon Tong, Hong Kong ({\texttt akflam@hkbu.edu.hk}).}
\footnotetext[3]{Department of Mathematics, University of Trento, Trento, Italy
({\texttt robert.nurnberg@unitn.it}).}
\footnotetext[4]{Department of Mathematics, Politecnico di Milano, 20133 Milano, Italy ({\texttt andrea.signori@polimi.it}), Alexander von Humboldt Research Fellow.}

\section{Introduction}
It has been proposed recently that chemical reactions in phase separating systems can lead to a suppression of Ostwald ripening and to the growth and division of droplets  \cite{zwickerostwald, Active_drops, bauermann2023formation}.
These systems are away from thermodynamic equilibrium with an external supply of energy enhancing chemical reactions.
%The presence of these chemical reactions can suppress Ostwald ripening, i.e., in particular, multiple droplets can stably coexist.
In \cite{Active_drops} it was even demonstrated that droplets in the presence of chemical reactions can grow and spontaneously split. Then a further growth of divided droplets, using up the fuel from chemical reactions, is possible, leading eventually to further splitting. The model studied in  \cite{zwickerostwald, Active_drops, bauermann2023formation} involves a Cahn--Hilliard model with chemical reactions, and in general,  does not fulfill a  free energy inequality. This is due to the fact that energy is supplied\anold{,} and such systems are called active systems. 
In  \cite{Active_drops} the authors argue that such active systems can play an important role in  the transition
between nonliving and living systems.  Initially, featureless aggregates of abiotic matter evolve and form protocells which  can be the basis for systems that  gain the structure and functions necessary to fulfill
the criteria of life.

It was also shown in subsequent studies that synthetic analogues of such chemically active systems can be developed, see 
\cite{DonauSSetal20}. In such system not only Ostwald ripening is suppressed but also stable liquid shells can form, see \cite{bauermann2023formation, BergmannBBetal23}. Here, a shell of one phase forms with an inside and an outside of a second phase. In fact, it was observed experimentally that spherical, active droplets  \anold{can} undergo a morphological transition into a  spherical shell. In \cite{BergmannBBetal23} it was shown that the mechanism is related to gradients of the  droplet material, and the authors also identify how much   chemical \anold{energy} is necessary to sustain the spherical drop. 
The non-conservative Cahn\anold{--}Hilliard system the authors \anold{of \cite{Active_drops} introduced} is
\begin{subequations}
	\begin{alignat*}{2}
		 \dt \ph & = \div(m(\ph)\nabla  \mu) + S_\eps(\ph),
			\\	 \mu &= \beta (-\eps \Delta \ph + \tfrac 1 \eps \psi'(\ph)),
	\end{alignat*}
\end{subequations}
where in this paper we take $\ph$ to be the normalized concentration difference between the droplet material and 
the background material\an{, which is scaled \mod{such that the two phases are at } $\pm1$. Besides, $\beta$ and $\eps$ are constants, and $m$ and  $S_\eps$ are phase-dependent functions that will be introduced later.} The variable $\mu$  is the chemical potential given as the first variation of a 
 \an{Ginzburg--Landau} {\it free energy}  \an{given by}
\begin{align}
	\label{freeeneintro}
	\an{{\cal E}(\ph) = }\beta \Big(
	\frac \eps 2 \iO |\nabla \ph|^2
	+ \frac 1\eps \iO \psi(\ph)	
	\Big)\an{,}
\end{align}
\an{where $\psi$ represents a suitable double-well free energy density.}
%\comm{[There is some repetition related to the Cahn–Hilliard equation here...see also below...]}
Compared to the classical Cahn--Hilliard model the main new term is the reaction term $S_\eps(\ph)$ which in
\cite{Active_drops} 
was taken to be affine linear outside of the interfacial region that separates the two phases given by the droplet  regions 
and the background region. Within the interfacial \anold{region} an interpolation between these two affine linear functions is
chosen. In the droplet \anold{phase} $\{\ph = 1\}$, $S_\eps(1)$ will be negative, which reflects the fact that the droplet material degrades chemically. In the background phase $\{\ph = -1\}$, $S_\eps(-1)$ will be positive, which takes into account that the material making up the
droplet phase is produced in the background phase by chemical reactions involving a fuel which powers its production, see 
\cite{Active_drops} for details.

In the case without chemical reactions, i.e., $S_\eps=0$\anold{,}  the Cahn--Hilliard model was first formulated
in \cite{Cahn61} using the  free energy
\eqref{freeeneintro} introduced in \cite{CahnH58}.
Since its introduction the Cahn--Hilliard equation has been  the subject of many studies and  has found many 
applications. We refer to \cite{NovickCohen98, Miranville19, BDGP} for detailed overviews.
In particular, it
can be shown that the Cahn--Hilliard model is the $H^{-1}$ gradient flow of the energy \eqref{freeeneintro}, see e.g.\an{,} \cite{Garcke13, BDGP}. Furthermore, it was also shown, first formally by Pego \cite{Pego89} and later rigorously by
Alikakos, Bates and Chen \cite{AlikakosBC94}, that the Cahn--Hilliard model converges to the Mullins--Sekerka sharp interface model as the interfacial thickness converges to zero.
\anold{It was}  also demonstrated \anold{that the Cahn--Hilliard model} can be
used to describe the Ostwald ripening process, where small particles dissolve and larger ones grow, see 
\cite{GarckeNRW03}. There are numerous analytical results  on the Cahn--Hilliard equation and we here only refer to the existence results in \cite{ElliottZ86, ElliottG96} and to \cite{AbelsW07,   Miranville19} who, in particular, discuss the Cahn--Hilliard equation from a semi-group perspective and also study  the physically  relevant 
logarithmic potential.

Several models have been proposed in which a reaction type term $S_\eps(\ph)$ appears.
The simplest one is the Cahn--Hilliard--Oono model in which 
$S_\eps(\ph)= -\omega \an{(\ph-c^*)}$ with a positive constant $\omega$\an{, and $c^* \in (-1,1)$ is given.} 
This term accounts for nonlocal interactions in phase separation, see \cite{Miranville19}. A proliferation term
$S_\eps(\ph)= - \lambda \ph(1-\ph)$ with a constant $\lambda>0$ has been introduced in \cite{KhainS08}. For analytical results in this case we refer to  \cite{Miranville19}.
A source term which depends on $\ph$ but  also depends on the spatial variable $x$
has been proposed in \cite{BertozziEG07} for applications in binary image inpainting and was subsequently analysed in
\cite{BurgerHS09,GLS18}, see also \cite{Wang19} for an application to image segmentation. In addition, in several tumor growth models\anold{,} Cahn--Hilliard type models with source terms appear and are coupled to other equations, see\an{, e.g.,} \cite{CristiniLLW09, Hawkins-DaarudZO12,GLSS}.

In this paper, we mathematically analyze the Cahn--Hilliard model  introduced in
\cite{Active_drops}. We will first carefully introduce the model and then show a well-posedness result for the system. 
We use formally matched asymptotic expansions to relate the diffuse interface Cahn–Hilliard model to a new sharp interface \an{model, which differs} from the sharp interface model proposed in \cite{Active_drops}.
In particular, we will show that asymptotic expansions lead to a quasi-static diffusion problem also possibly involving
source terms stemming from reactions at the interface.
For the sharp interface model, we derive planar stationary solutions.
% and its stability. 
%In addition, we the existence of two  radially symmetric solutions and it will turn out that for certain parameters
%two stationary radial solutions exist. 
%The larger one is stable with respect to radial perturbations but for certain parameters  will
%turn out to be unstable under non-radial perturbations. The stability analysis is similar to the famous
%Mullins--Sekerka stability analysis \cite{MullinsSekerka} and  has be also used in \cite{Active_drops} for the case of
%the  non-quasistationary case.
Finally, we will use finite element computations for the Cahn--Hilliard model with reactions to
numerically verify the matched asymptotics and to illustrate  the stability and instability behavior of solutions. In particular,
we will show several splitting scenarios as well as the formation of shell-like \anold{structures.}
% in two and three spatial dimensions.

\section{Mathematical models}
\subsection{The Cahn--Hilliard model}
Let $\Omega$ be a bounded domain \an{ in $\mathbb{R}^d$, $d\in \{\mod{1,}2,3\}$,} with boundary $\partial \Omega$ containing two chemical species. 
We introduce a normalized difference $\ph$ of the concentrations of two chemical components that is governed by the following Cahn--Hilliard equation with chemical reactions \cite{Active_drops}:
\begin{subequations}\label{eq:sys21}
\begin{alignat}{2}
	\label{sys:1}
	\dt \ph  &= \div(m(\ph)\nabla  \mu) + S_\eps(\ph)
	\qquad && \text{in $Q := (0,T) \times \Omega $,}
	\\
	\label{sys:2}
	 \mu &= \beta (-\eps \Delta \ph + \tfrac 1 \eps \psi'(\ph))
	\qquad && \text{in $Q$,}
	\\
	\label{sys:3}
	\dn \mu  &= \dn \ph =0
	\qquad && \text{on $\Gamma := (0,T) \times \partial \Omega $,}
	\\
	\label{sys:4}
	 \ph(0) &=\ph_0
	\qquad && \text{in $\Omega$.}
\end{alignat}
\end{subequations}
\Accorpa\Sys {sys:1} {sys:4}
Here, $\mu$ is the associated chemical potential, $m\anold{: \erre \to \erre_{>0}}$ is the concentration dependent mobility function, $\beta >0$ is a parameter related to a surface energy density, $\eps>0$ is a small length scale \anold{proportional} to the thickness of the diffuse interface function, $\psi\anold{: \erre \to \erre_{\geq 0}}$ is a double well potential, $ \dn $ is the derivative in the direction of the unit outer normal $\bnn$ to $\partial \Omega$ and $\ph_0$ serve\anold{s} as initial data for $\ph$. We consider the above equations on the space-time cylinder ${Q}$ with a fixed but arbitrary time $T>0$.

The source term $S_\eps:\erre \to \erre$
is given as
\begin{equation}\label{def:source}
S_\eps(r)= S_1(r)+\tfrac 1\eps S_2(r){,\quad r \in \erre}.
\end{equation}
Here, the term $S_2$ will later lead to a fast reaction in the interfacial region.
On choosing $r_c \in (0,1]$, we set for constants $S_+$, $S_-$, $K_+$\anold{, $K_-$, and $L$}
\begin{align}
	\label{def:source1} 
	S_1(r)= 
	\begin{cases}
		\Sp  & \quad \text{if $r \geq {r_c}$,}
		\\
		\Sm + G_1(r)(\Sp-\Sm) 
		& \quad \text{if $r\in (-{r_c},{r_c})$,}
		\\
		\Sm & \quad \text{if $r \leq -{r_c}$,}
	\end{cases}
\end{align}
and
\begin{align}
	\label{def:source2} 
	S_2(r)= 
	\begin{cases}
		- \Kp  (r-1) & \quad \text{if $r \geq {r_c}$,}
		\\
		\hat S_2(r)
%		- {\Km}  G_2(r) -  {\Kp}  G_3(r)+ L G_4(r)
%		{- K_+(r_c-1)G_1(r) - K_-(1 - r_c)(1-G_1(r))}
		& \quad \text{if $r\in (-{r_c},{r_c})$,}
		\\
		- \Km   (r+1)& \quad \text{if $r \leq -{r_c}$,}
	\end{cases}
\end{align}
where we define for $r\in (-{r_c},{r_c})$
$$
\hat S_2(r)=		- {\Km}  G_2(r) -  {\Kp}  G_3(r)+ L G_4(r)
{- K_+(r_c-1)G_1(r) - K_-(1 - r_c)(1-G_1(r))}.
$$
Here,  $G_1,G_2,G_3, G_4 :[-1,1] \to \erre$ are suitable differentiable interpolation functions\anold{,} to be introduced below\anold{,} satisfying
\begin{align}
	& \label{ass:G:1}
	G_1({r_c})=1,
	\quad 
	G_1(-{r_c}) =0,
	\quad 
	G_2(\pm{r_c}) =G_3(\pm{r_c}) =G_4(\pm{r_c}) =0, 
	\\ 
	& \label{ass:G:2}
	G_1'(\pm{r_c})=G_4'(\pm{r_c})=0,
	\quad 
	G_2'({r_c})=G_3'(-{r_c})=0,
	\quad 
	G_2'(-{r_c})=G_3'({r_c})=1,
\end{align}
so that the source term $S_\eps$ is differentiable on $\erre$. 
We often use $r_c=1$ which considerably simplifies the expression for $\hat S_2$ and the matched asymptotic expansions which we use later
to derive a sharp interface limit. However, other choices can also be considered, such as $r_c = \frac12$ \anold{as} in \cite{Active_drops}.
%\footnote{\anold{Oberserve that so far  for the asymptotic analysis we require $S_2(\pm1)=0$, i.e., $r_c=1$.}}

%We will set
%\begin{align}
%	\label{def:source1}
%	S_1(r)=
%	\begin{cases}
%	\Sp  & \quad \text{if $r \geq 1$,}
%	\\
%	\Sm + G_1(r)(\Sp-\Sm)
%	 & \quad \text{if $r\in (-1,1)$,}
%	\\
%	\Sm & \quad \text{if $r \leq -1$,}
%	\end{cases}
%\end{align}
%and
%\begin{align}
%	\label{def:source2}
%	S_2(r)=
%	\begin{cases}
%		- \Kp  (r-1) & \quad \text{if $r \geq 1$,}
%		\\
%		- {\Km}  G_2(r) -  {\Kp}  G_3(r)+ L G_4(r)
%		& \quad \text{if $r\in (-1,1)$,}
%		\\
%		 - \Km   (r+1)& \quad \text{if $r \leq -1$,}
%	\end{cases}
%\end{align}
%for constants $S_{\pm}$ and $K_{\pm}$ while $G_1$, $G_2$, $G_3$ and $G_4 :\erre \to \erre$ are suitable differentiable interpolation functions to be introduced below that satisfy 
%\begin{subequations}\label{ass:G}
%\begin{alignat}{2}
%\label{ass:G:1} & G_1(1) = 1, \quad G_1(-1) = 0, \quad G_1'(\pm 1) = 0, \\
%\label{ass:G:2} & G_2(\pm 1) = 0, \quad G_2'(1) = 0, \quad G_2'(- 1) = 1, \\
%\label{ass:G:3} & G_3(\pm 1) = 0, \quad G_3'(1) = 1, \quad G_3'(-1) = 0, \\
%\label{ass:G:4} & G_4(\pm 1) = 0, \quad G_4'(1) = 0, \quad G_{4}'(-1) = 0,
%\end{alignat}
%\end{subequations}
%so that the source term $S_\eps$ defined in \eqref{def:source} is differentiable on $\erre$.

The system \anold{\eqref{eq:sys21}} is related to the following {\it free energy}
\begin{align*}
	{\cal E}(\ph) = \beta \Big(
	\frac \eps 2 \iO |\nabla \ph|^2
	+ \frac 1\eps \iO \psi(\ph)	
	\Big).
\end{align*}
In what follows we assume that $\psi$ is even, that is $\psi(r)=\psi(-r)$ for $r \in \erre$, and satisfies $\psi(\pm 1) = 0$
and $\psi''(\pm 1) \neq 0$.
Typically, we will choose the  {\it quartic potential}
\begin{align}
	\label{quartic}
	\psi(r) = \tfrac 14 (1-r^2)^2, \quad r \in \erre.
\end{align}

\subsection{Possible choice of the interpolation functions}
As mentioned above, {in the theoretical analysis to follow,} we consider any interpolation function\anold{s} $G_1$, $G_2$, $G_3$ and $G_4$ such that \eqref{ass:G:1}--\eqref{ass:G:2} are fulfilled.
Here, we present a possible choice related to the double well potential $\psi$. We set, for every $r \in [-1,1]$,
\begin{align}
	\label{def:G1}
	G_1(r) { = \hat G_1(\tfrac{r}{r_c}), \quad \hat G_1(r) }
	& =
	\tfrac 34 (r+1)^2- \tfrac 14(r+1)^3,
	\\
	\label{def:G23}
	G_2(r) { = r_c \hat G_2(\tfrac{r}{r_c}), \quad \hat G_2(r) }
	& = -\anold{\tfrac 12}\tfrac {1}{ \sqrt{\psi''(-1)}} (r-1) \sqrt{2 \psi(r)},
	\quad
	G_3 (r) = - G_2(-r),
\end{align}
and observe that  
\begin{align*}
	\hat G_1(1)=1,
	\quad 
\hat 	G_1(-1) =G_1'(\pm1)=0,	
	\quad 
\hat 	G_2(\pm 1) =0,
	\quad 
\hat 	G_2'(1)=0,
	\quad 
\hat 	G_2'(-1) = 1.
\end{align*}
We just provide the details \anold{for} verifying $\hat G_2'(-1) = {1}$ as the others are straightforward. For convenience, let us set $c^*:=-\anold{\frac 12}\frac {1}{ \sqrt{ \psi''(-1)}}$ so that
\begin{align*}
	\an{\hat G_2'(-1) = \frac d {dr} \hat G_2(r)\Big|_{r=-1} }= c^* \Big(\sqrt{2\psi(r)} + (r-1) \frac d{dr} \sqrt{2\psi(r)}\Big)\Big|_{r=-1} = -2 {c^*}\frac d{dr}\sqrt{2\psi(r)}\Big|_{r=-1}  .
\end{align*}
For the argument in the latter expression, using Taylor's expansion, it holds that
\begin{align*}
	\psi(r) =
	\underbrace{\psi(-1)}_{=0}
	+ \underbrace{\psi'(-1)(r+1)}_{=0}
	+ \frac 12\psi''(-1)(r+1)^2
	+ o \big((r+1)^2\big),
	\quad
	r \in \erre,
\end{align*}
whence we infer that
\begin{align*}
	\frac d{dr}\sqrt{2\psi(r)}\Big|_{r=-1}
	= \lim_{r \searrow  -1} \frac {\sqrt{\psi''(-1)} \,|r+1|}{r+1} = \sqrt{ \psi''(-1)}.
\end{align*}
Thus, \an{using} the definition of $\hat G_2$ as \an{stated} above\an{, we infer} that $\hat G_2'(-1)=1$\an{,} as claimed. Note that due to the relation $\hat G_3(r) = - \hat G_2(-r)$ the condition $\hat G_3'(1) = 1$ can be inferred in the same way. In particular, for the quartic potential \eqref{quartic}, we have $\psi''(-1) = 2$, and so
\[
\hat G_2(r) = - \tfrac{1}{4} (1-r^2)(r-1), \quad \hat G_3(r) = - \hat G_2(-r) = - \tfrac{1}{4}(1-r^2)(r+1)
\]
which fulfill \eqref{ass:G:1} and \eqref{ass:G:2}, respectively.
%In particular, let us notice that for the quartic potential \eqref{quartic} we have $\psi''(-1)=2$, so that we obtain
%\begin{align*}
%	\hat G_2 (r)& = 
%	- \tfrac 14 (1-r^2)(r-1)
%	,
%	\quad
%	\hat G_3 (r) = {-} G_2(-r)= {-}
%	\tfrac 14 (1-r^2)(r+1)
%	,
%\end{align*}
%which fulfill \eqref{ass:G:1}--\eqref{ass:G:2}.
In addition, we set
\begin{equation}\label{eq:G4}
	G_4(r) { = \hat G_4(\tfrac{r}{r_c}), \quad \hat G_4(r) }
	= 2\psi(r) {,}\quad r\in \mathbb{R}\an{,}
\end{equation}
\an{which clearly satisfies  \eqref{ass:G:1} and \eqref{ass:G:2}.}

\subsection{The sharp interface model}
We now present the corresponding sharp interface system related to the Cahn--Hilliard model above.
We will later use the method of formally matched asymptotic expansions, see, \an{e.g.,} \cite{BDGP}, to derive this free boundary problem.
We denote by  $\an{\Omega^\pm=}\Omega^\pm(t)$  the two regions occupied by the {pure} phases and by $\Sigma\an{=\Sigma}(t)$  the \an{evolving} interface separating the two phases.
In addition, let
$m_\pm:=m(\pm1)$ be the mobilities in the two phases.
\anold{The} free boundary problem corresponding to \anold{\eqref{eq:sys21}} reads as follows:
\begin{subequations}\label{SharpI}
\begin{alignat}{2}
	\label{sharpI:1:p}
	 - m_+ \Delta \mu &= \Sp - \rho_+ \mu  \quad &&\text{in $\Omega^+$,}
		\\
        \label{sharpI:1:m} - m_- \Delta \mu & =\Sm - \rho_- \mu  \quad && \text{in $\Omega^-$,}
	\\
	\label{sharpI:2}
	\mu &=  \frac{\gamma \beta \kappa}2
	\qquad && \text{on $\Sigma$,}
	\\
	\label{sharpI:3}
	\jump{\mu} & = 0
	\qquad && \text{on $\Sigma$,}
	\\
	\label{sharpI:4}
	 - 2 {\cal V} & = \jump{m\nabla \mu} \cdot {\bnu}
	+ S_I
	\qquad && \text{on $\Sigma$,}
	\\
	\label{sharpI:5}
	 \dn \mu & =0
	\qquad && \text{on $\partial \Omega$,}
\end{alignat}
\end{subequations}
\Accorpa\SharpI {sharpI:1:p} {sharpI:5}
where
 $	\rho_\pm= \frac {K_\pm}{\beta \psi''(\pm1)}$,  $\gamma$ is a constant depending on the choice of the double well potential $\psi$ and \anold{$S_I$ is} the interface reaction term   which  depends on the choice of the interpolation functions $G_2$, $G_3$ and $G_4$ \an{(cf. \eqref{defn:SI})}. In the case of the quartic potential and $r_c=1$, we will obtain
 \[
 \rho_\pm= \frac {K_\pm}{2\beta}, \quad \gamma = \frac{2 \sqrt{2}}{3}, \quad S_I = \frac{1}{\sqrt{2}} \Big (K_{+} - K_{-} + \frac{4}{3} L \Big ).
 \]
In the above, we also use $\kappa$ to denote the mean curvature \anold{of $\Sigma$} that is given as the sum of the principal curvatures of $\Sigma$, $\bnu$ is the unit normal to the interface, \anold{and} ${\cal V}$ \anold{is} the normal velocity of the interface in the direction of the normal $\bnu$.
In addition,
 for $x\in\Sigma(t)$ and a function $u$\an{, we define}  its jump across the interface at $(t,x)$ as
\begin{equation*}
	[u]^+_-(t,x):=\lim_{\substack{y\to x\\y\in\Omega^+(t)}} u(t,y)-
	\lim_{\substack{y\to x\\y\in\Omega^-(t)}} u(t,y)\,.
\end{equation*}
Further information about the notation used can be found in \cite{BDGP}.

Notice that the above system is connected to the well-known {\it Mullins--Sekerka} free boundary problem \cite{BDGP}, with the differences that here we have a constant source term {$S_I$} on the \rhs\ of \eqref{sharpI:4} as well as an affine linear term in
the quasi-static diffusion equations \eqref{sharpI:1:p} and \eqref{sharpI:1:m}.

\subsection{Nondimensionalization for the sharp interface problem}\label{sec:nondim}
We now perform a nondimensionalization {argument} in order to identify important dimensionless parameters.
Choosing units $\tilde x$, $\tilde t$ and $\tilde \mu$ for length, time and chemical potential we introduce
the nondimensional variables
\begin{equation*}
	\hat x =\frac x{\tilde x},\quad \hat t =\frac t{\tilde t}, \quad \hat \mu =\frac \mu{\tilde \mu}\anold{,}
\end{equation*}
\anold{and now consider the rescaled variant of system \eqref{eq:sys21} on the rescaled domains $\hat \Omega^+$, $\hat \Omega^-$\an{, and the corresponding} interface $\hat{\Sigma}$.}
Denoting by $\hat{\nabla}$ and $\hat{\Delta}$ \anold{the} gradient and Laplacian with respect to $\hat{x}$, for the new nondimensional variables\an{,} we obtain from \anold{\eqref{SharpI}} the following system
\begin{subequations}
\begin{alignat*}{3}
%	\label{sharpIn:1:p}
			- \hat\Delta \hat\mu &=
		\frac{(\tilde x)^2}{m_+ \tilde \mu }\Sp -  \frac{(\tilde x)^2\rho_+}{m_+} \hat \mu  \quad && \text{in $\hat \Omega^+$},
		\\
	%	\label{sharpIn:1:m}
	 -  \hat \Delta \hat \mu &= \frac{(\tilde x)^2}{m_- \tilde \mu } \Sm - \frac{(\tilde x)^2\rho_-}{m_-}
		\hat \mu \quad && \text{in $ \hat \Omega^-$,}
	\\
%	\label{sharpIn:2}
	 \hat \mu &=  \frac{\gamma \beta \an{\hat \kappa}}{2\tilde \mu \tilde x} 
	\qquad && \text{on $\hat \Sigma$,}
	\\
%	\label{sharpIn:3}
	 \jump{\hat \mu} &= 0
	\qquad && \text{on $\hat \Sigma$,}
	\\
%	\label{sharpIn:4}
	 - 2 { \cal \hat V} & = \tfrac{\tilde t\tilde \mu } {(\tilde x)^2}\jump{m\hat \nabla \hat\mu} \cdot {\hat \bnu}
	+\tfrac{\tilde t}{\tilde x} S_I
	\qquad && \text{on ${\hat{\Sigma}}$,}
	\\
%	\label{sharpIn:5}
	 \dn \hat \mu &= 0
	\qquad && \text{on ${\partial \hat \Omega}$,}
\end{alignat*}
\end{subequations}
{with $\hat {\calV}$ being the normal velocity of the interface $\hat \Sigma$ in the direction of the normal $\hat \bnu$.}
To obtain simple nondimensional equations, 
\anold{and on assuming that $\rho_->0$ and $S_- >0$,}
we set
\begin{equation*}
	\tilde x =\sqrt{\frac{{m_-}}{\rho_-}}, \quad \tilde \mu = \frac{(\tilde x)^2 S_- }{m_-}= \frac{S_-}{\rho_-}, \quad \tilde t =
	 \frac{(\tilde x)^2}{\tilde \mu m_-} =\frac 1{S_-} .
\end{equation*}
We now define the nondimensional parameter
\begin{equation*}
	\beta^* =\frac{\beta}{\tilde x} \frac 1{\tilde \mu}=
	 \frac{\beta \rho_-^{\tfrac 32}}{m_-^{\frac12} S_-}= \frac {c_{l}}{\tilde x},
\end{equation*}
where we call
\begin{equation}
	\label{cap:lenght}
	c_{l}=  \beta \frac{\rho_-}{S_-}\anold{,}
\end{equation}
in analogy to solidification problems\anold{,} the modified {\it capillary length}.
In addition, we introduce the relative mobility $m^*$,  the relative reaction coefficients $S^*$ and $\rho^*$, and the nondimensional interface reaction term $S_I^*$ as follows
\begin{equation*}
	m^*=\frac {m_+} {m_-},\quad S^*=\frac {S_+} {S_-},\quad
	\rho^*=\frac {\rho_+} {\rho_-}, \quad S_I^* =\frac{\tilde t}{\tilde x}S_I= \frac{1}{S_-} \sqrt{\frac{\rho_-}{m_-}} S_I.
\end{equation*}
We then obtain, dropping the hat notation for convenience, 
\begin{subequations}
\begin{alignat*}{3}
		%\label{sharpIIn:1:m}
			-  m^*\Delta  \mu &=  S^* -
	\rho^* \mu  \quad && \text{in $  \Omega^+$,}
	\\
%	\label{sharpIIn:1:p}
	-  \Delta \mu &=
		 1 -  \mu \quad && \text{in $ \Omega^-$},
		\\
	%\label{sharpIIn:2}
	 \mu & = \frac{\gamma \beta^*\an{\kappa}}2 
	\qquad & &\text{on $ \Sigma$,}
	\\
%	\label{sharpIIn:3}
	\jump{ \mu} &= 0
	\qquad && \text{on $ \Sigma$,}
	\\
%	\label{sharpIIn:4}
	 - 2 { \cal  V} &= m^*\nabla \mu_+ \cdot { \bnu} -
	\nabla \mu_- \cdot { \bnu}
	+ S_I^*
	\qquad && \text{on $\Sigma$,}
	\\
%	\label{sharpIIn:5}
	 \dn \mu &=0
	\qquad && \text{on $\partial \Omega$,}
\end{alignat*}
\end{subequations}
and observe that the evolution critically depends on the nondimensional number $\beta^*$ which relates the influence of surface tension to a generalized supersaturation stemming from chemical reactions.

\section{Well-posedness}
In this section, we address \anold{the} well-posedness of the system \anold{\eqref{eq:sys21}}, aiming to cover a wide spectrum of scenarios. Specifically, we aim to accommodate various configurations without relying on the specific structure of the source \anold{term} $S_\eps$, as long as its growth is under control. Furthermore, in our analysis we can include rather general potentials\anold{,} provided they are regular, nonsingular and \anold{exhibit} polynomial growth.
Let us first specify the notation we need for the well-posedness result.
\subsection{Notation}
Let $\Omega$ be a bounded domain in $\mathbb{R}^d$, $d\in \{2,3\}$, with boundary $\partial \Omega$.  The Lebesgue measure of $\Omega$ is denoted by $|\Omega|$, while the Hausdorff measure of $\partial \Omega$ is denoted by $|\partial \Omega|$.

For any Banach space $X$, its norm is represented as $\| \cdot \|_X$, its dual space as $X^*$, and the duality pairing between $X^*$ and $X$ is denoted by $\langle \cdot, \cdot \rangle_X$. In the case where $X$ is a Hilbert space, the inner product is denoted by $(\cdot,\cdot)_X$.

For each $1 \leq p \leq \infty$ \an{and} $k \geq 0$ the standard Lebesgue and Sobolev spaces defined on $\Omega$ are denoted as $L^p(\Omega)$ and $W^{k,p}(\Omega)$, with their respective norms $\| \cdot \|_{L^p(\Omega)}$ and $\| \cdot \|_{W^{k,p}(\Omega)}$. For simplicity we \anold{may often} use $\| \cdot \|_{L^p}$ instead of $\| \cdot \|_{L^p(\Omega)}$, and employ similar shorthand \anold{notation} for other norms. We adopt the convention $H^k(\Omega) := W^{k,2}(\Omega)$ for all $k \in \mathbb{N}$, and denote the mean value of a functional $h \in \Vp$ as
\[
h_\Omega := \frac{1}{|\Omega|} \langle h,1 \rangle_{H^1}.
\]
We now introduce a tool commonly employed in the investigation of problems associated with equations of Cahn--Hilliard type. Given $\phi \in \Vp$, we seek $u \in \Hx1$ such that
\begin{align}
	\label{weak:neu}
	\int_{\Omega} \nabla u \cdot \nabla v  = \langle \phi , v \rangle_{H^1},
	\quad v \in \Hx1.
\end{align}
This corresponds to the standard weak formulation of the homogeneous Neumann problem for the Poisson equation $-\Delta u = \phi$ for $\phi \in \Lx2$. The solvability of \eqref{weak:neu} for $\phi \in \Vp$ relies on the condition that $\phi$ possesses a zero mean value, that is, $\phi_\Omega=0$. If this condition is satisfied, a unique solution with a zero mean value exists and the operator ${\cal N}: \text{dom}({\cal N})= \{\phi \in \Vp : \phi_\Omega = 0\} \to \{u \in \Hx1 : u_\Omega = 0\}$ defined by mapping $\phi$ to the unique solution $u$ to \eqref{weak:neu} with $u_\Omega = 0$ is well-defined. This operator yields an isomorphism between the mentioned spaces. Additionally, the norm
\begin{align*}
	\phi \mapsto \|\phi\|_{*}^2 := \|\nabla {\cal N} (\phi - {\phi_\Omega})\|^2_{L^2} + |\phi_\Omega|^2,
	\quad
	\phi \in \Vp,
\end{align*}
 is proven to define a Hilbert norm in $\Vp$ \an{that} is equivalent to the standard dual norm. From these definitions, it directly follows that
 \begin{equation*}
 \begin{alignedat}{3}
  &\int_{\Omega} \nabla {\cal N}\phi \cdot \nabla v  = \langle \phi , v \rangle_{H^1} \quad &&\text{for every $\phi \in \text{dom}({\cal N})$ and $v \in \Hx1$},\\
  &\langle \phi , {\cal N}\zeta \rangle_{H^1} = \langle \zeta , {\cal N}\phi \rangle _{H^1}\quad&& \text{for every $\phi, \zeta \in \text{dom}({\cal N})$}, \\
  &\langle \phi , {\cal N}\phi \rangle _{H^1}= \norma{\nabla {\cal N}\phi}_{L^2}^2  = \|\phi\|_{*}^2 \quad&& \text{for every $\phi \in \text{dom}({\cal N})$}.
\end{alignedat}
\end{equation*}
Moreover, it is established that
\begin{align*}
	\int_{0}^{t} \langle \partial_t v(s) , {\cal N} v(s) \rangle_{H^1} \, ds
	= \int_{0}^{t} \langle v(s) , {\cal N}(\partial_t v(s)) \rangle_{H^1} \, ds
	= \frac{1}{2} \|v(t)\|_{*}^2 - \frac{1}{2} \|v(0)\|_{*}^2
\end{align*}
for every $t \in [0, T]$ and $v \in \H1 {\Vp}$ such that $v_\Omega(t) = 0$ for every $t \in [0,T]$.

\subsection{Assumptions} For the well-posedness, we require the following assumptions.
\begin{enumerate}[label=$\boldsymbol{(\mathrm{A \arabic*})}$, ref =$\boldsymbol{({\mathrm{A \arabic*}})}$]
	\item \label{ass:wp:constants}
	The symbols $K_\pm$ and $S_\pm$ denote real valued constants, whereas $\beta$ and $\eps$ denote positive constants.
	
	\item \label{ass:wp:pot}
	The potential $\psi : \erre \to [0,\infty)$ is twice differentiable and can be decomposed as $\psi=\psi_1 + \psi_2$, with $\psi_1$ convex and $\psi_2$ a quadratic perturbation. Namely, we require that there exist a positive constant $C_1$ 
	%and $C_2$ 
	such \an{that it} holds
		\begin{align*}
		%	\psi(r)&  \geq C_1 (|r|^2+1), \quad 
		%	\text{and} \quad 
			|\psi_2(r)|  \leq C_1 (|r|^2 +1)\anold{,}
			\quad 
%			\text{	for all $r \in \erre,$}
			\an{r \in \erre,}
	\end{align*}%
and in  addition, we require
	\begin{align*}
		\forall \eta>0 \, \exists \,C_\eta: \anold{\forall r \in \erre} \quad |\psi'(r)| & \leq \eta \psi(r) + C_\eta .
	\end{align*}
%and
%		\begin{align*}
%	\psi(r)&  \geq C_1 (|r|^2+1), \quad 
%	\text{and} \quad 
%	|\psi_2(r)|  \leq C_2 (|r|^2 +1)\anold{,}
%	\quad \text{	for all $r \in \erre.$}
%	\end{align*}}%
	
	\item \label{ass:wp:source}
	We require the source $S_\eps$ to be Lipschitz continuous. Consequently, there exists a positive constant $C_S$ such that
	\begin{align*}
		|S_\eps(r)|
		\leq
		C_S (|r|+1),
		\quad
		r \in \RRR.
	\end{align*}
	\item \label{ass:wp:mob}
	The mobility function $m:\erre \to \erre$ is \anold{continuous and} there exist positive constants $m_*$ and $M^*$ such that
	\begin{align*}
		0 <m_* \leq m(r) \leq  M^*,
		\quad r \in \erre.
	\end{align*}
	
\end{enumerate}
Let us notice that assumption \eqref{ass:pot:growth} restricts the class of admissible double-well potentials, but still includes the quartic potential in \eqref{quartic}.
\last{For this latter, referring to \ref{ass:wp:pot}, we employ the splitting 
\begin{align*}
	\psi_1(r) = \tfrac 14 r^4,
	\quad 
	\psi_2(r) = \tfrac 14 (1- 2 r^2),
	\quad 
	r \in \erre.
\end{align*}
}

Here is our main result.
\begin{thm}	\label{THM:WP}
	Suppose that \ref{ass:wp:constants}--\ref{ass:wp:mob} are fulfilled. Then, for every given $\ph_0\in {\Hx1}$ there exists a weak solution \an{$(\ph,\mu)$} to \anold{\eqref{eq:sys21}} such that
	\begin{align*}
	\ph & \in \H1 \Vp \cap \L\infty {\Hx1} \cap \L2 {\Hx2},
	\\
	\mu & \in \L2 {\Hx1},
\end{align*}
and satisfy\anold{ing}
\begin{align*}
	& \<\dt \ph, v>_{{H^1}}
	+ \iO m(\ph)  \nabla \mu \cdot \nabla v = \iO S_\eps(\ph) v,
	\\
	& \iO \mu v =
	\beta\eps  \iO \nabla \ph \cdot \nabla v
	+ \frac \beta \eps \iO \psi'(\ph)v,
\end{align*}
for every $v \in \Hx1 $ and almost every $t \in(0,T)$, along with attainment of the initial condition $\varphi(0) = \varphi_0$ holding for almost every $x \in \Omega$.

Moreover, let $\{(\ph_i,\mu_i)\}_i$, $i=1,2$, denote two arbitrary solutions to \anold{\eqref{eq:sys21}} associated \an{with} initial data $\ph_{0,i}\in \Hx1$, $i=1,2$ and to \anold{a} constant mobility $m$. 
In addition, let \an{us assume that there exists $p \in [1,7)$ such that} $\psi_1'$ \an{satisfies} the following pointwise growth condition 
\begin{align}
	\label{ass:pot:growth}
	|\psi_1'(r)-\psi_1'(s)| 
	\leq 
	C (1+ |r|^p+|s|^p) |r-s|,
	\quad r,s \in \erre\an{.}
%	\quad \an{\text{for any $p \in [1,7)$.}}
\end{align}
Then, it holds that
\begin{align*}
	& \norma{(\ph_1-\ph_2)-(\ph_1-\ph_2)_\Omega}_{\L\infty\Vp \cap \L2 {\Hx1} }
	+ \norma{(\ph_1)_\Omega-(\ph_2)_\Omega}_{L^\infty(0,T)}
%	\\ & \qquad 
%	{	+ \norma{(\ph_1)_\Omega-(\ph_2)_\Omega}_{L^\infty(0,T)}^{1/2}}
	\\ & \quad
	\leq \anold{C\an{^*}} 
	\big(
	{\norma{(\ph_{0,1} -\ph_{0,2}) - ((\ph_{0,1})_\Omega - (\ph_{0,2})_\Omega )}_*}
%	{\norma{(\ph_{0,1} - (\ph_{0,1})_\Omega ) -(\ph_{0,2} - (\ph_{0,2})_\Omega )}_*}
	+|(\ph_{0,1})_\Omega-(\ph_{0,2})_\Omega|
%		{	+ |(\ph_{0,1})_\Omega-(\ph_{0,2})_\Omega|^{1/2}}
\big),
\end{align*}%
for a positive constant $\anold{C}\an{^*}$ just depending on $\Omega$, $T$ and the nonlinearity $\psi$. Consequently, 
\anold{under these conditions the weak solution to \eqref{eq:sys21} is unique.}
\end{thm}
%
%Let us notice that assumption \eqref{ass:pot:growth} restricts the class of admissible double-well potentials, but still includes the quartic potential in \eqref{quartic}.
%
\begin{proof}[Proof of Theorem \ref{THM:WP}]
{We begin with the existence part of the theorem.  In the subsequent discussion, we adopt a formal approach, leveraging standard procedures to derive estimates for the solution. While these computations are formal, they suggest that the same estimates can be rigorously applied to the $k$-dimensional system obtained via a Faedo--Galerkin scheme, constructed using the first $k$ eigenfunctions of the Laplace operator with homogeneous \an{Neumann} boundary conditions. These bounds can then be used to pass to the limit as $k\to\infty$, thereby constructing a solution to the problem that satisfies \anold{\eqref{eq:sys21}}. A rigorous proof can be easily adapted within this approximation framework.}
%In the subsequent discussion, we adopt a formal approach, leveraging standard procedures. However, a rigorous proof can be easily adapted within an appropriate approximation framework, such as a Faedo--Galerkin approach. \anold{INSERT COMMENTS...}  
\an{Furthermore, in what follows, since we will need to make numerous estimates, we will agree to use $C$ to denote any nonnegative constant depending on the system's data, which may change its value from line to line.}
	
{\it First estimate:}
Before starting with the {proper} estimates, let us introduce a new function related to the  {phase-dependent} mobility  {function} $m$. We set
\begin{align*}
	M'(z) = \int_0^z \frac 1{m(s)} ds,
	\quad z \in \erre,
\end{align*}
{call $M$ a suitable antiderivative of it} and notice that $M \in C^2(\erre)$ and  $M''(r)=\frac 1 {m(r)}$, $r \in \erre$.
Due to the bounds on $m$ in \ref{ass:wp:mob}, we also have that there exist two positive constants $c_1$ and $c_2$ such that
\begin{align*}
	c_1(1+|r|)
	\leq
	|M'(r)|
	\leq
	c_2(1+|r|),
	\quad r \in \erre.
\end{align*}
We \an{then} test \eqref{sys:1} \an{with} $M'(\ph)$, \eqref{sys:2} \an{with} $-\Delta \ph$ and add the resulting identities to obtain
	\begin{align*}
		\frac 12 \frac d {dt} \norma{M(\ph)}^2_{L^2}
		+ \beta \eps \norma{\Delta \ph}^2_{L^2}
		+ \frac \beta \eps \iO \psi_1''(\ph)|\nabla \ph|^2
		=\iO S_\eps(\ph) M'(\ph)
		- \frac \beta \eps \iO \psi_2''(\ph)|\nabla \ph|^2\an{,}
	\end{align*}
	\an{where we notice that the third term on the \lhs\ is nonnegative due to \ref{ass:wp:pot}.}
	Now, since $M'$ is growing linearly,  using \ref{ass:wp:constants}--\ref{ass:wp:mob}, we readily infer from Young's inequality that the \rhs\ is bounded \anold{from} above by
	$C (\norma{\ph}^2_{L^2} +1)$. Using Gr\"onwall's inequality then produces
	\begin{align*}
		\norma{M(\ph)}_{\L\infty {\Lx2} }
		+ \norma{\Delta \ph}_{\L2 {\Lx2} } \leq C.
	\end{align*}
Besides, since $M'$ is growing linearly, $M$ grows quadratic{ally,} so that from the above inequality we obtain, using also elliptic regularity theory, that
	\begin{align*}
		\norma{\ph}_{\L\infty {\Lx2} \cap \L2 {\Hx2}} \leq C.
	\end{align*}
	
%{\it First estimate:}
%	We test \eqref{sys:1} by $\ph$, \eqref{sys:2} by $-\Delta \ph$ and add the resulting identities to obtain
%	\begin{align*}
%		\frac 12 \frac d {dt} \norma{\ph}^2
%		+ \beta \eps \norma{\Delta \ph}^2
%		+ \frac \beta \eps \iO F_1''(\ph)|\nabla \ph|^2
%		=\iO S_\eps(\ph) \ph
%		- \frac \beta \eps \iO F_2''(\ph)|\nabla \ph|^2.
%	\end{align*}
%	Due to \ref{ass:wp:pot}--\ref{ass:wp:mob}, we readily infer from Young's inequality that the \rhs\ is bounded above by
%	$C (\norma{\ph}^2 +1)$. Using Gronwall's lemma along with elliptic regularity theory  then produces
%	\begin{align*}
%		\norma{\ph}_{\L\infty H \cap \L2 {\Hx2}} \leq C.
%	\end{align*}
	
	{\it Second estimate:}
	Next, recalling \ref{ass:wp:pot},
%	 a comparison argument in \eqref{sys:2} produces
upon testing \eqref{sys:2} with $\frac 1{|\Omega|}$ yields 
	\begin{align*}
		\norma{\mu_\Omega}_{L^2(0,T)} \leq C,
\end{align*}	
where we also used the Sobolev embedding $H^2(\Omega) \emb \Lx\infty$.

		{\it Third estimate:}
		{We now perform the usual energy estimate in the context of the Cahn--Hilliard equation by testing}
%		Next, we test 
\eqref{sys:1} \anold{with} $\mu$, \eqref{sys:2} \anold{with} $-\dt \ph$ and adding the resulting identities \an{to} obtain
		\begin{align*}
			\frac d{dt} {\cal E}(\ph)
			+ \norma{\nabla \mu}^2_{L^2}
			= \iO S_\eps(\ph) \mu
			= \iO S_\eps(\ph) (\mu-\mu_\Omega)
			+ \iO S_\eps(\ph)\mu_\Omega = I_1 + I_2.
		\end{align*}
		Now, {employing the Young inequality} and \ref{ass:wp:source}, we infer that
		\begin{align*}
			|I_1| &\leq C ( \norma{\ph}_{L^2}+1)\norma{\nabla \mu}_{L^2}
			\leq \tfrac 12 \norma{\nabla \mu}^2_{L^2}
			+ C {(}\norma{\ph}^2_{L^2}{+1)},
			\\
			|I_2| &\leq C
			( \norma{\ph}_{L^2}+1)|\mu_\Omega|
			\leq
			C {(}\norma{\ph}^2_{L^2}{+1)}
			+ |\mu_\Omega|^2.
		\end{align*}
		Thus, using {the Poincar\'e--Wirtinger inequality} and also the previous estimates readily entails that
		\begin{align*}
			\norma{\ph}_{\L\infty {\Hx 1}}
			+ \norma{\mu}_{\L2 {\Hx 1}}
			\leq C.
		\end{align*}
		
	{\it Fourth estimate:}
		Finally, it is standard to \an{infer from the previous estimates and the weak formulation \eqref{sys:1} that}
		\begin{align*}
			\norma{\dt \ph}_{\L2 \Vp}
			\leq C.
		\end{align*}
This concludes the existence part of the proof. 
%\anold{In the appropriate approximation framework,  using compactness results and the estimates obtained, we pass to the limit as $k\to \infty$, ensuring the existence of a solution to \anold{\eqref{eq:sys21}}.}

Moving to the uniqueness part, we consider two solutions $\{(\ph_i,\mu_i)\}_i$, $i=1,2$, associated to initial data $\ph_{0,i}$, $i=1,2$ and recall that the mobility is assumed to be \an{constant}. Then, we introduce the notation
\begin{align*}
	\ph= \ph_1-\ph_2,
	\quad
	\mu= \mu_1- \mu_2,
	\quad \ph_0=\ph_{0,1}-\ph_{0,2},
\end{align*}
and consider the system \anold{\eqref{eq:sys21}} written for the differences. Considering the difference of \eqref{sys:1}, recalling the Lipschitz continuity of $S_\eps$, and testing it \an{with} 
%$\ph_\Omega$ produces
%\begin{align}
%	\label{cd:mean}
%	\frac 12 \frac d{dt}|\ph_\Omega|^2
%	\leq
%	C  |\ph_\Omega|^2
%	+C {\norma{\ph}^2_{L^2}}.
%\end{align}
$\ph_\Omega\an{=(\ph_{1})_\Omega-(\ph_{2})_\Omega}$
produces
\begin{align}
	\label{cd:mean}
	\frac d{dt}\Big( \frac 12 |\ph_\Omega|^2 
%	+ |\ph_\Omega|
	\Big)
	\leq
	C  |\ph_\Omega|^2
%	+ C  |\ph_\Omega|
	+ C \norma{\ph}^2_{L^2}.
\end{align}
Then, we consider the difference of \eqref{sys:1} minus its mean value and test it \anold{with} ${\cal N}(\ph-\ph_\Omega)$,
the difference of \eqref{sys:2} \anold{and test it with}  $-(\ph-\ph_\Omega)$ and, upon adding the resulting equalities, we obtain that
\begin{equation}\label{uniq:est}
\begin{aligned}
	& \frac 12 \frac d{dt}\norma{\ph-\ph_\Omega}^2_*
	+ \iO {(\mu-\mu_\Omega)(\ph-\ph_\Omega)} \\
    & \quad = \iO \big( S_\eps (\ph_1) - S_\eps(\ph_2) - (S_\eps (\ph_1) - S_\eps(\ph_2) )_\Omega\big) {\cal N}(\ph-\ph_\Omega)
	\\ & \quad \leq
	C \norma{\ph}^2_{L^2}
	+ C \norma{\ph-\ph_\Omega}^2_* .
\end{aligned}
\end{equation}
Besides, {it holds that}
\begin{align*}
	\iO {(\mu-\mu_\Omega)(\ph-\ph_\Omega)}
	= \beta \eps \norma{\nabla \ph}^2_{L^2}
	+ \frac \beta \eps \iO (\psi'(\ph_1)-\psi'(\ph_2))(\ph-\ph_\Omega).
\end{align*}
For the last term, using \ref{ass:wp:pot}, we have
\begin{align*}
	& \frac \beta \eps \iO (\psi'(\ph_1)-\psi'(\ph_2))(\ph-\ph_\Omega)
	\\
 & \quad= \underbrace{\frac \beta \eps \iO (\psi_1'(\ph_1)-\psi_1'(\ph_2))\ph}_{\geq 0}
	+ \underbrace{\frac \beta \eps \iO (\psi'_2(\ph_1)-\psi'_2(\ph_2))\ph}_{\geq -C \norma{\ph}^2_{L^2}}
	\\
 & \qquad
	-\frac \beta \eps \iO (\psi_1'(\ph_1)-\psi_1'(\ph_2))\ph_\Omega
	-\underbrace{\frac \beta \eps \iO (\psi_2'(\ph_1)-\psi_2'(\ph_2))\ph_\Omega}_{\geq - C (\norma{\ph}^2_{L^2} + |\ph_\Omega|^2)}.
\end{align*}
We then move the third term on the right-hand side \an{of this latter identity} to the \rhs\ of the estimate \eqref{uniq:est} and continue with the estimation.
Recalling the growth condition in \eqref{ass:pot:growth} and the continuous embedding $\Hx1 \emb \Lx6$ \an{holding in three dimensions}, we apply H\"older's inequality to \an{bound}
%\begin{align*}
%	& 
%	-\frac \beta \eps \iO (\psi_1'(\ph_1)-\psi_1'(\ph_2))\ph_\Omega
%	\leq 
%	C  \left(1+ \norma{|\ph_1|^p}_{L^3} + \norma{|\ph_2|^p}_{L^3} \right) \norma{\ph}_{L^6}\norma{|\ph_\Omega|}_{L^2}
%	\\ & \quad 
%	\leq \frac {\beta \eps}{4} \norma{\ph}^2_{H^1}
%	+ C  \left(1+ \norma{\ph_1}_{L^{3p}}^{2p} + \norma{\ph_2}_{L^{3p}}^{2p} \right) |\ph_\Omega|^2.
%\end{align*}
	\begin{align*}
		& 
		\Big|\frac \beta \eps \iO (\psi_1'(\ph_1)-\psi_1'(\ph_2))\ph_\Omega\Big|
		\leq 
		C  \left(1+ \norma{|\ph_1|^p}_{L^{\frac 65}} + \norma{|\ph_2|^p}_{L^{\frac 65}} \right) \norma{\ph}_{L^6} |\ph_\Omega|
		\\ & \quad 
		\leq \frac {\beta \eps}{4} \norma{\ph}^2_{H^1}
		+ C  \Big(1+ \norma{\ph_1}_{L^{\frac {6p}5}}^{2p} + \norma{\ph_2}_{L^{\frac {6p}5}}^{2p} \Big) |\ph_\Omega|^2.
	\end{align*}
Given that 
\begin{align*}
\ph_i \in \L\infty {\Hx1} \cap \L2 {\Hx2} \emb L^{\frac {4q}{q-6}}(0,T; {\Lx{q}})
\end{align*}
for $i=1,2$ and $q \in [6,\infty)$ in three spatial dimensions, with the convention that $\frac {4q}{q-6}:= +\infty$ when $q=6$, we deduce that any exponent up to $p = 5$ is admissible as  we notice that 
\[
t \mapsto \Lambda(t):=(1+ \norma{\ph_1(t)}_{L^6}^{10} + \norma{\ph_2(t)}_{L^6}^{10} ) \in L^{\infty}(0,T).
\]
	Besides, selecting $q = \frac {6p}5$ in the above interpolation embedding, we infer that the resulting time exponent $\frac {4p}{p/5}$ is strictly bigger than $2p$ for any $p \in (5,7)$, entailing that $t \mapsto \Lambda(t) \in L^1(0,T)$ for any $p \in (5,7)$. Regarding the term  $C \norma{\ph}^2$ on the \rhs\ of \eqref{uniq:est}, we apply an interpolation argument to deduce that
\begin{align*}
	C \norma{\ph}^2_{L^2}
	\leq
	\tfrac  {\beta \eps}{4} \norma{\nabla \ph}^2_{L^2}
	+ C \norma{\ph-\ph_\Omega}^2_*
	+ C |\ph_\Omega|^2.
\end{align*}

Thus, rearranging the terms, and adding the above estimates, we end up with
\begin{align*}
	& \frac 12 \frac d{dt} \Big(\norma{\ph-\ph_\Omega}^2_* +  |\ph_\Omega|^2 	
	%{+ 2|\ph_\Omega|}
	\Big)
	{+\frac {\beta \eps}{2}} \norma{\nabla \ph}^2_{L^2}
	%	\\ &\quad  
	\leq
	C \norma{\ph-\ph_\Omega}^2_*
	%	+ C |\ph_\Omega|^2
	%	\\ & \quad
	{+ C  (\Lambda +1 )|\ph_\Omega|^2.}
\end{align*}
Finally, we integrate over time and
%\anold{recall that $t \mapsto (\norma{\psi_1'(\ph_1(t))}_{L^1} + \norma{\psi_1'(\ph_2(t))}_{L^1}\anold{+1}) \in L^1(0,T)$, and} 
employ Gr\"onwall's inequality to conclude the proof.
%
%\anold{\fbox{Andrea}\, It seems to me that the last term $C  (\norma{\psi_1'(\ph_1)}_{L^1} + \norma{\psi_1'(\ph_2)}_{L^1})|\ph_\Omega|$ is not a proper Gronwall term so far. Using the above testing should fix the issue though as on the \lhs\ we will also have the contribution $\frac d{dt}|\ph_\Omega|$. What do you think?}

\end{proof}

\section{Sharp Interface Limit}
%\comm{[The inner part is the {\it solid} part, so that we choose $+1$ for it.]}
In this section, we conduct a formal asymptotic analysis of \anold{the} system \anold{\eqref{eq:sys21}} as $\eps$ approaches zero for potentials $\psi \in C^2(\erre)$ that fulfil
\begin{align}
	\label{ass:pot:sharp}
	\psi(r)= \psi(-r), \quad r \in \erre,
	\quad
	\psi(\pm1)=\psi'(\pm1)= \psi'(0)=0, \quad \psi''(\pm1)\neq 0.
\end{align}
%\comm{[to say: zero is unstable]}
In addition, we present only the case $r_c=1$. An analysis for $r_c\in (0,1)$ is also possible but will lead to 
more intricate computations in order to define $S_I$ \an{(cf. \eqref{defn:SI})}. However, the following analysis also holds for
$r_c\in (0,1)$ without changes in the case that $L=0$ and $K_+=K_-$.
The method integrates outer and inner expansions into the model equations, solving them stepwise, and defines a region for their
matching.
%Of course, the explicit choice \eqref{quartic} fulfills \ref{ass:wp:pot}.
Further elaboration on the methodology can be found in {the} references \cite{AbelsGG12,GLSS,BDGP}. The following assumptions and conventions are in order:
\begin{itemize}
	\item
		It is assumed that for small values of $\eps$, the domain $\Omega$ can be partitioned into two open subdomains, $\Omega{^\pm_\eps}=\Omega{^\pm}(\eps,t)$, separated by an interface $\Sigma_\eps=\Sigma_\eps(t)$ and  that $\Omega{^+_\eps}$ does not intersect with $\partial \Omega$. We assume that the introduced spatial sets evolve over time, though we opt not to specify the temporal dependence for convenience.
	\item
		It is assumed that there exists a family  of solution to \anold{\eqref{eq:sys21}} that  are sufficiently smooth and exhibit an asymptotic expansion in $\eps\in (0,1)$ within the bulk regions away from $\Sigma_\eps$ (referred to as the {\it outer expansion}), and another expansion in the interfacial region adjacent to $\Sigma_\eps$ (referred to as the {\it inner expansion}).	
		Those will be denoted by $\{(\ph_\eps,\mu_\eps)\}_\eps$ and $\{(\Phi_\eps,\Lambda_\eps)\}_\eps$ {in the following}, respectively.
	\item
	We postulate that the level sets $\an{\{ \ph_\eps =0\}=}\{x\in \Omega : \ph_\eps(x) =0\}$, as $\eps \to 0$, converge to a limiting hypersurface  $\Sigma_0{=\Sigma_0(t)}$  that evolves with a normal velocity $\cal V$ and normal $\bnu$.
\end{itemize}
In summary, we will work in the geometric framework
\begin{align*}
	\Omega = \Omega_\eps^+ \cup \Sigma_\eps \cup \Omega_\eps^-	,
	\quad 
	\Sigma_\eps  = \partial \Omega_\eps^+,
	\quad 
	\Omega_\eps^+= \Omega \setminus \ov{\Omega_\eps^-},
\end{align*}
where we set
\begin{align*}
	\Omega_\eps^+ = \{x \in \Omega: \ph_\eps(t,x) >0\},
	\quad \Omega_\eps^- = \{x \in \Omega: \ph_\eps(t,x) <0\}.
\end{align*}

\subsection{Outer expansion}
In what follows, we suppose that the solution variables $\ph_\eps$ and $\mu_\eps$, far away from the interface, can be expressed as
\begin{align*}
	\ph_\eps = \sum_{i=0}^{\infty} \eps^i \ph_i,
	\quad
	\mu_\eps = \sum_{i=0}^{\infty} \eps^i \mu_i.
\end{align*}
Then, we consider equation \eqref{sys:2} to leading order $\eps^{-1}$ to infer that
\begin{align*}
	- \beta \psi'(\ph_0)=0.
\end{align*}
Here, we accounted for the conditions $\psi(\pm1)=\psi'(\pm1)=0$ in \eqref{ass:pot:sharp}.
Since $\ph_0=0$ is an unstable solution, this leads to $\ph_0=\pm 1$. We then set
\begin{align*}
	\Omega^\pm:=\{x \in \Omega: \ph_0(x) = \pm 1\}.
\end{align*}
To the next order $\eps^0$, equation \eqref{sys:2} yields
\begin{align*}
	\mu_0 = \beta \psi''(\ph_0)\ph_1.
\end{align*}
{Besides, to the same order, s}ince $ \ph_0$ \an{equals $\pm1$} in $\Omega^\pm$, we infer from \eqref{sys:1} that
{we} have
\begin{align*}
	- \div(m(\ph_0)\nabla \mu_0) = S_\pm - K_\pm \ph_1.
\end{align*}
Combining these two equations and recalling $\rho_{\pm} = \frac{K_{\pm}}{\beta \psi''(\pm 1)}$ yields
\begin{align*}
- m_+ \Delta \mu_{0} = S_{+} - \rho_+ \mu_0 & \quad \text{ in } \Omega^+, \\
- m_- \Delta \mu_{0} = S_{-} - \rho_- \mu_0 & \quad \text{ in } \Omega^-{,} 
\end{align*}
{where we recall $m_\pm = m(\pm1)$.}

\subsection{Inner expansion and matching conditions}
To explore the behavior of $\Omega_\eps=\an{\{\ph_\eps=0\}}
%\{x \in \Omega: \ph_\eps(x)=0\}
$ as $\eps \to 0$, we introduce a new coordinate system.
Let $d$ denote the signed distance function to $\Sigma_0$, and define $z= \frac d\eps$ the rescaled distance variable. Additionally, we select $d$ in such a way that $d(\xxx)>0$ in $\Omega^+$ and $d(\xxx)<0$ in $\Omega^-$.
Thus, it follows that $\nabla d = \bnu$ is the unit normal of $\Sigma_0$ and points from $\Omega^-$ towards $\Omega^+$.
Next, let us consider a parametrization of $\Sigma_0$ by arc-length denoted as $g(t, s)${.}
%, and let $\bnu$ represent the unit normal vector of the interface, oriented from $\Omega^-$ towards $\Omega^+$.
 Then, within a tubular neighbourhood of $\Sigma_0$, for a sufficiently smooth function ${f=}f(x)$, we obtain the reparametrization rule
\begin{align*}
	f(x) = f(g(t,s)+\eps z \bnu (g(t,s))) =: F(t,s,z).
\end{align*}
Referring to, \an{e.g.,} \cite{GLSS}, this allows us to write derivatives of $f$ using the new reference coordinates \an{leading} to the following identities
\begin{align}
	\label{f:trans}
	\partial_t f \approx - \frac  1 \eps {\cal V} \partial_z F ,
	\quad
	\nabla_x f \approx \frac 1\eps \partial_z F \bnu +
	\nabla_{\Sigma_0} F,
	\quad
	\Delta_x f \approx \frac 1{\eps^2} \partial_{zz} F - \frac 1\eps \kappa \partial_z F.
\end{align}
Here, $\nabla_{\Sigma_0}$ stands for the surface gradient on $\Sigma_0$, $\kappa=-\div_{\Sigma_0} \bnu$ for the corresponding mean curvature, and the symbol $\approx$ indicates that we omitted higher order terms with respect to $\eps$.
The solution in the inner region is assumed to possess the following expansion
\begin{align*}
	\Phi_\eps
	 = \sum_{i=0}^{\infty} \eps^i \Phi_i,
	\quad
	\Lambda_\eps = \sum_{i=0}^{\infty} \eps^i \Lambda_i.
\end{align*}
The postulated convergence of the level sets $\{\ph_\eps=0\}$ to the hypersurface $\Sigma_0$ translates to the condition
\begin{align}
	\label{ass:phz}
	\Phi_0(t,s,z=0)=0.
\end{align}
Furthermore, we assume 
%\comm{[better if we wrote them as limits?]}
\begin{align*}
\lim_{z \to + \infty} 	\Phi_\eps (t,s,z) =1,
	\quad
\lim_{z \to -\infty} 	\Phi_\eps (t,s,z) =-1.
\end{align*}
With reference to \cite{GLSS}, we employ the following matching conditions
\begin{alignat}{2}
	\label{match:1}
	& \lim_{z \to \pm \infty} \Phi_0 (t,s,z) =	\ph_0^{\pm}(t,\xxx),
	\quad
	&& \lim_{z \to \pm \infty} \Lambda_0 (t,s,z) = \mu_0^{\pm}(t,\xxx),
	\\
	\label{match:2}
	& \lim_{z \to \pm \infty} \partial_z \Phi_0 (t,s,z) = 0,
	\quad
	&& \lim_{z \to \pm \infty} \partial_z \Lambda_0 (t,s,z) = 0,
	\\
	& \label{match:3}
	\lim_{z \to \pm \infty} \partial_z \Phi_1 (t,s,z) = \nabla \ph_0^{\pm}(t,\xxx)\cdot \bnu,
	\quad
	&& \lim_{z \to \pm \infty} \partial_z \Lambda_1 (t,s,z) = \nabla \mu_0^{\pm}(t,\xxx)\cdot \bnu,
\end{alignat}
where $\ph_0^\pm (t,\xxx):= \lim_{\delta\to0} \ph_0 (t,\xxx\pm \delta \bnu)$, $\xxx \in \Sigma_0$, and similarly for $\mu_0$. This allows us to introduce the corresponding jump across $\Sigma_0$ by using the notation
\begin{align*}
	\jump{\ph_0}=\jump{\ph_0(t,\xxx))}:=\ph_0^+ (t,\xxx)-\ph_0^- (t,\xxx),
\end{align*}
and similarly for $\mu_0$.
%\comm{[This has already been introduced at the end of page 5. Should we repeat it or cite it with a formula?]}

\subsection{Leading order expansions in the interfacial region}
From \eqref{sys:2}, to leading order $\eps^{-1}$, we find
\begin{align*}
	\partial_{zz} \Phi_0 - \psi'(\Phi_0)=0.
\end{align*}
Using \eqref{ass:phz}, we observe that
this entails  $\Phi_0$ is just a function of $z$, resulting  \anold{in} the ODE relation
\begin{align}
	\label{ODE:phiz}
\partial_{zz}	\Phi_0(z)-\psi'(\Phi_0(z))=0,
	\quad
	\Phi_0(0)=0,
	\quad
	\Phi_0(\pm\infty) = \pm1.
\end{align}
%Recalling that $\psi$ is given by \eqref{quartic}, solving the above ODE yields
%\begin{align*}
%	\Phi_0(z)=
%	\tanh \big(\tfrac z{\sqrt{2}}\big),
%	\quad z \in \erre.
%\end{align*}
Upon testing \eqref{ODE:phiz} \an{with} $\partial_{z}\Phi_0$, we obtain the {so-called {\it equipartition of energy}}:
\begin{align}
	\label{eq:energy}
	\frac 12 |\partial_{z}\Phi_0(z)|^2 = \psi
	(\Phi_0(z)),
	\quad z \in \erre.
\end{align}
Hence, we find the identity
\begin{align}
	\label{def:gamma}
	\intinf |\partial_{z}\Phi_0(z)|^2\,dz =
	\intinf 2 \psi(\Phi_0(z))\,dz
	= \int_{-1}^1 \sqrt{2 \psi(s)} \, ds =: \gamma.
\end{align}
From $\psi'(-z)=-\psi'(z)$, we see that $-\Phi_0(z)=\Phi_0(-z)$.
Let us point out that, in the special case of $\psi$ being the quartic potential \eqref{quartic},
i.e., $\psi(r) = \tfrac 14 (1-r^2)^2,$ we obtain
$$\gamma= {\frac 1{\sqrt{2}}} \int_{-1}^1 (1-s^2 ) \, ds = \frac {2\sqrt{2}}3 .$$
%\footnote{We can now erase the following lines}
% the solution $\Phi_0(z)$ to the ODE problem \eqref{ODE:phiz} is $\Phi_0(z) = \tanh (\frac{z}{\sqrt{2}})$ and thus the interfacial parameter $\gamma$, see, e.g., \cite{GLSS}, is calculated as
%\begin{align*}
%%	\label{def:gamma}
%	\gamma=
%%	\begin{cases}
%	\frac 12 \intinf
%		 {\rm sech}^4 \big (\frac{z}{\sqrt{2}} \big)
%	\,dz=\frac {2\sqrt{2}}3 .
%%	\quad & \text{for the regular potential}
%%	\\
%%	\int_{-\frac \pi2}^{\frac \pi2}
%%		\cos^2(z)
%%	\,dz =\frac \pi2 \quad & \text{for the  double obstacle potential}.
%%	\end{cases}
%\end{align*}

Next, considering \eqref{sys:1} to order $\eps^{-2}$ produces
\begin{align*}
	\partial_{z}(m(\Phi_0)\partial_z \Lambda_0 ) = 0. 
\end{align*}
Upon integration and using the matching condition \eqref{match:2} to $\Lambda_0$, we infer that
\begin{align}\label{cond:lambdaz}
	m(\Phi_0)\partial_{z} \Lambda_0(t,s,z) =0,
	\quad  {z \in \erre}.
\end{align}
Since $m(\Phi_0)>0$, \eqref{cond:lambdaz} implies $\partial_{z} \Lambda_0(t,s,z) = 0$. We integrate once more and use the matching condition \eqref{match:1} to obtain the jump condition 
\[
\jump{\mu_0}=0.
\]

\subsection{Higher  order expansions in the interfacial region}
%\comm{[Change the title: it is not just first order. Remove the section title entirely?]}
Moving to \an{$\eps^0$} order in \eqref{sys:2},  we derive that
\begin{align}
	\Lambda_0 = \beta \psi''(\Phi_0)\Phi_1
	- \beta \partial_{zz}\Phi_1
	+ \beta \kappa\partial_{z}\Phi_0. \label{eq:Phi1}
\end{align}
%\comm{[Here we should make a choice on notation. It would be best to use  $\partial_{zz}\Phi_1 = \Phi_1''$ to be consistent with $\Phi_0' = \partial_{z}\Phi_0$...]}
Testing the above by $\pd_z \Phi_0$ and integrating from $-\infty$ to ${+\infty}$ produces
\begin{align*}
	\intinf
	\Lambda_0(t,s) \partial_{z}\Phi_0(z)\,dz
	= \beta \intinf \big( \partial_{z} (\psi'(\Phi_0(z))) \Phi_1
	-  \partial_{zz}\Phi_1 \partial_{z}\Phi_0(z)
	+  \kappa|\partial_{z}\Phi_0(z)|^2 \big) \,dz.
\end{align*}
Using \eqref{match:1} and \eqref{match:2} for $\Phi_0$, integrating by parts and using ${\psi'(\pm1)=0}$, we infer
\begin{align*}
	& \intinf
	 \partial_{z} (\psi'(\Phi_0(z)))\Phi_1  -
	\partial_{zz}\Phi_1 \partial_{z}\Phi_0(z)
	 \,dz
	 \\ & \quad
	  =
	 \big[\psi'(\Phi_0)\Phi_1 -
	\partial_{z}\Phi_1 \partial_{z}\Phi_0\big]^{+\infty}_{-\infty}
	 -\intinf
		\partial_z \Phi_1 (\psi'(\Phi_0(z))-\partial_{zz}\Phi_0(z))	
	 \, dz=0,
\end{align*}
{as both terms on the \rhs\ vanish,}
whence, recalling \eqref{eq:energy} {and} \eqref{def:gamma}, this entails that 
\[
2 \mu_0 = \gamma\beta \kappa.
\]

In order to utilize the equations to order $\eps^{-1}$, we now 
{exploit the preliminary assumption}
%take for simplicity 
$r_c=1${. From \eqref{sys:1}, we} obtain 
%Next, to order $\eps^{-1}$ we find from \eqref{sys:1} that
\begin{align*}
	-{\cal V} \partial_{z}\Phi_0 = \partial_{z}(m(\Phi_0)\partial_z \Lambda_1) -[ \Km G_2(\Phi_0)-\Kp G_2(-\Phi_0)]+L G_4(\Phi_0) .
\end{align*}
Integrating from $-\infty$ to ${+\infty}$ and using \eqref{match:3} for $\Lambda_1$ leads to
\begin{align*}
		-2 {\cal V} & =
		\jump{m(\ph_0)\nabla \mu_0} \cdot \bnu +    \int_{- \infty}^{{+\infty}} \Kp G_2(-\Phi_0(z)) - \Km G_2(\Phi_0(z)) + L  G_4(\Phi_0(z))\, dz.
%		\\ &
%		= \jump{m(\ph_0)\nabla \mu_0} \cdot \bnu +   S_K,
\end{align*}
As $\Phi_0$ is a fixed function, the integral on the \rhs\ yields a constant $S_I$ depending on $K_\pm$, $L$ and the double well potential $\psi$.
By 
%\eqref{def:G2}, \eqref{def:G3} and 
$-\Phi_0(z)=\Phi_0(-z)$, it holds that
\begin{align}
%	\non
\label{def:sec}
	S_I & :=   \int_{- \infty}^{{+\infty}} \Kp G_2(-\Phi_0(z)) - \Km G_2(\Phi_0(z)) \,dz +
	L \int_{- \infty}^{{+\infty}}  G_4(\Phi_0(z))\, dz\\ \non
	&=
	{(\Kp-  \Km)} \int_{- \infty}^{{+\infty}} G_2(\Phi_0(z)) \, dz  +
	L \int_{- \infty}^{{+\infty}} G_4(\Phi_0(z))\, dz,
\end{align}
where, upon setting $\tilde c := \tfrac {1}{2 \sqrt{\psi''(-1)}}$, we realize with the help of the equipartition of energy \eqref{eq:energy} that
\begin{align*}
	& \int_{- \infty}^{{+\infty}} G_2(\Phi_0(z)) \, dz	
	=
	 -\tilde c  \int_{-\infty}^{{+\infty}}
	(\Phi_0(z)-1)\sqrt{2 \psi(\Phi_0(z))}\, dz
	\\ 
	& \quad
  =-\tilde c \int_{-\infty}^{{+\infty}}  (\Phi_0(z)-1)\partial_{z}\Phi_0(z)\, dz
	= -\tilde c \int_{-1}^1 (s-1)\,ds =  2 \tilde c.
\end{align*}
In addition, we notice that
\[
G_4(\Phi_0(z))=2 \psi(\Phi_0(z))=\sqrt{2 \psi(\Phi_0(z))}
\sqrt{2 \psi(\Phi_0(z))} =\sqrt{2 \psi(\Phi_0(z))}\partial_{z} \Phi_0(z)
\]
and compute
\[
\int_{- \infty}^{{+\infty}}  G_4(\Phi_0(z))\, dz= \int_{- \infty}^{{+\infty}} \sqrt{2 \psi(\Phi_0(z))} \partial_{z}\Phi_0(z) {\, dz}=
 \int_{-1}^1 \sqrt{2 \psi(s)} \, ds=\gamma.
 \]
%\begin{align}
%	& \int_{- \infty}^{{+\infty}} G_2(\Phi_0(z)) \, dz	\non
%	=
%	 -\frac 14  \int_{-\infty}^{{+\infty}}
%	(1-\Phi_0^2(z))(\Phi_0(z)-1)\, dz
%	\\ \label{SK}
%	& \quad
%%	\int_{-\infty}^{{+\infty}}
%%	\Big(
%%	- 2 {\hat K}_- \Big(\frac {\Phi_0(z)+1}{2}\Big)
%	=-\frac 1{4 \hat c} \int_{-\infty}^{{+\infty}}  \Phi_0'(z)(\Phi_0(z)-1)\, dz
%	= -\frac 1{4 \hat c} \int_{-1}^{1} (s-1)\, ds =  \frac 1{2 \hat c},
%\end{align}
%where $\hat c = \frac {\sqrt 2}2 $ is such that $ \Phi_0'(z) ={\hat c} (1- \Phi_0^2(z))$ as can be checked from \eqref{eq:energy}, so that
Thus, recalling the definition of $\tilde c$, we find that
\begin{align}\label{defn:SI}
S_I=  \frac {{\Kp-  \Km}}{\sqrt{\psi''(-1)}}+\gamma L.
\end{align}
This concludes our analysis concerning the sharp interface system \anold{\eqref{SharpI}}. \an{Let us point out that}, for the \an{quartic potential \eqref{quartic}}, it holds that $\tilde c = \frac {\sqrt 2}4$, and so 
\begin{align}\label{defn:SI:eg}
S_I= \frac {\sqrt 2}2 ({\Kp-  \Km})+\frac{2\sqrt{2}}3  L.
\end{align}
We remark that in the case that $r_c\in (0,1)$ the term $S_I$ can be computed 
if we insert $\Phi_0$ in the definition of the source term, see \eqref{def:source},  and integrate from $-\infty$ to $\infty$.
This would then replace the integrals in \eqref{def:sec}.

{Combining the calculations above, we have demonstrated that as $\varepsilon \to 0$, the phase-field system \anold{\eqref{eq:sys21}} formally converges to the  limit described by the free-boundary problem \anold{\eqref{SharpI}}.}

\section{Planar solutions and their stability}
\subsection{Setting}
A simple solution can be computed in a planar geometry. In  particular, we will later use planar solutions to validate 
the asymptotic analysis from the previous section. More precisely, we will compare numerical computations for the phase field model with planar solutions of the sharp interface limit.
%\comm{[Do we really do that? In this case, it might be better to explain it more clearly.]}
We now consider the free boundary problem \anold{\eqref{SharpI}} in the {special} domain
\[
\Omega =(0,\Lpanarh)\times (0,\Lpanarv)^{d-1}{,}
\]
for $\Lpanarh, \Lpanarv>0$ and \anold{look} for a planar solution under the geometry
\begin{align*}
\Omega_+(t)=(0,q(t)) \times (0, \Lpanarv)^{d-1}, \quad
\Omega_-(t)=(q(t),\Lpanarh) \times (0, \Lpanarv)^{d-1},
\end{align*}
where $q(t)$ encodes the location of the \an{moving} interface $\Sigma_0$. In addition, we require a $90^\circ$ degree boundary condition at points where the interface meets the
external boundary. For $x\in \erre^d ${,} we write $x=(z,\hat x)$ with $z\in\erre$ and $\hat x\in \erre^{d-1}$. 
As $\bm{\nu}=(-1,{\0})^\top$\an{,} we get
\[
{\cal V} = -\frac {dq(t)}{dt} 
{=-\dot q}\anold{,}
\]
\anold{with the dot denoting the time derivative.}
For this planar setting we make the ansatz 
\an{\[
\an{
	\mu_\pm (t, x )=
}
\mu_\pm (t, (z,\hat x) )  =\hat\mu_\pm (t,z)
,
\quad 
x=(z,\hat x) \in (0, \Lpanarh)\times(0, \Lpanarv)^{d-1}.
\]}%
%for $z\in (0, \Lpanarh)$ and $\hat x \in (0, \Lpanarv)^{d-1}$. 
On the interface we obtain
\[
\nabla \mu_\pm  \an{(t,x)}\cdot \bm{\nu} =- \hat \mu_\pm^\prime (t,z),
\]
where prime denotes the partial derivative with respect to $z$.
In what follows, we drop the hat-notation \an{for convenience}. Due to the planar setting, we have $\kappa=0$ and hence \eqref{sharpI:2} and \eqref{sharpI:3} can be replaced by
\begin{equation}\label{sharp:3prime}
\mu_+(t,q(t)) =\mu_-(t,q(t)) =0,
\end{equation}
while \eqref{sharpI:4} can be \an{reformulated} \mod{as}
\begin{equation}
	2 {\dot q} = - \jump{m \mu'} +S_I.
	\label{sharp:4prime}
	\end{equation}
Equation \eqref{sharpI:5} can be replaced by
\begin{equation}
	\mu_+^\prime (t,0)=0, \quad 	\mu_-^\prime (t,\Lpanarh)=0.
	\label{sharp:5prime}
\end{equation}
In addition, \eqref{sharpI:1:p} and \eqref{sharpI:1:m} become
\begin{align}\label{ODEplan1}
	-m_+ \mu''_+ &=S_+  -\rho_+ \mu_+ \qquad \text{for } z\in (0,q(t)) ,\\ \label{ODEplan2}
	-m_- \mu''_- &=S_-  -\rho_-\mu_- \qquad \text{for } z\in (q(t),\Lpanarh ) .
	\end{align}
%	We now intruduce the nondimensional setting and
%	hence set
%	$m_-=S_- = \rho_-=1$ in the governing equations
%	and replace $m_+, S_+, \rho_+$ by
%	$m_*, S_*, \rho_*$ and in addition define
%	$$d^*= \frac{S_*}{\rho^*}, \qquad (\Lambda^* )^2= \frac{\rho^*}{m^*}.
%	$$
%	The ordinary differential equations
%	\eqref{ODEplan1},  \eqref{ODEplan2} are now given by
%	\begin{align*}
%		-\frac1{(\Lambda^* )^2} (\mu_+)'' &=S^*  -\rho^* \mu_+ \qquad \text{for} z\in (0,q(t)) ,\\
%		- \mu''_- &=1  -\mu \qquad \text{for} z\in (q(t),\Lpanarh ) .
%	\end{align*}
%	\footnote{no nondim at this point}
%	
\subsection{Solution formula and evolution equation for the  interface position}
From now on, we restrict ourselves to the case $\rho_\pm >0$ which is \an{relevant for applications, as shown in \cite{Active_drops}.}
We obtain \an{the following} solutions
\begin{align}
	\mu_+(t,z)&= d_+ \Big (1-\frac {\cosh{(\Lp  z)}}{\cosh{(\Lp q{(t)})}} \Big ),\label{muplan1}\\
	\mu_-(t,z)&= d_- \Big (1-\frac {\cosh{(\Lm (\Lpanarh-z))}}{\cosh{(\Lm (\Lpanarh-q{(t)}))}} \Big ){,}\label{muplan2}
\end{align}
where the constants $d_{\pm}$ and $\anold{\lambda_\pm}$ are defined as
\[
d_{\pm} := \frac{S_{\pm}}{\rho_{\pm}} = \beta \frac{S_{\pm} \psi''(\pm 1)}{K_{\pm}}, \quad \anold{\lambda_\pm} := \sqrt{\frac{\rho_{\pm}}{m_{\pm}}}.
\]
It can be \an{readily} verified that \eqref{muplan1} and \eqref{muplan2} solve the ordinary differential equations \eqref{ODEplan1} and \eqref{ODEplan2}, and fulfill the boundary conditions \eqref{sharp:3prime} and \eqref{sharp:5prime}.

Meanwhile, the evolution equation \eqref{sharp:4prime} for the interface position ${q=}q(t)$ becomes
\begin{equation}\label{eq:dqdt}
\begin{aligned}
	 {\dot q} &= \frac{1}{2} \Big (-m_+ \mu'_+ (q) + m_- \mu'_- (q) +S_I \Big ) \\
 &= \frac{1}{2} \Big ( d_+m_+ \Lp  \tanh{(  \Lp  q)}+ d_-m_- \Lm  \tanh{(  \Lm  (\Lpanarh-q))}  +S_I \Big ) \\
 & =: {\cal H}(q).
 \end{aligned}
\end{equation}
We notice that
\begin{align*}
{\cal H} (0) = \frac 12[d_-m_- \Lm  \tanh (\Lm  \Lpanarh)  +S_I], \quad {\cal H} (\Lpanarh) = \frac 12[d_+m_+ \Lp \tanh(\Lp 
\Lpanarh )  +S_I].
\end{align*}
In the case where $S_I = 0$, by having $d_+ < 0$ and \anold{since the} other parameters $m_{\pm}$, $\anold{\lambda_\pm}$ and $d_-$ \anold{are} positive, it holds that ${\cal H}(0) > 0$, ${\cal H}(\Lpanarh) < 0$ and ${\cal H}'(q) < 0$. Due to \anold{the} continuity of ${\cal H}(q)$\an{,} we are guaranteed the existence of exactly one root $q^*$ where ${\cal H}(q^*) = 0$. For example, setting $S_+ = -1$ and all other parameters equal to 1, we find that $q^* = \frac{\Lpanarh}{2}$ is a root of ${\cal H}$. Note that when $S_I \neq 0$ it is possible that roots of ${\cal H}$ may not exist at all. However, after fixing parameters $d_{\pm}$ and $\anold{\lambda_\pm}$ (hence fixing also $K_{\pm})$, one can adjust the parameter $L$ in $S_I$ to help ensure the existence of a root for ${\cal H}(q)$.

\subsection{Linear stability of planar solutions}\label{subsec:linstabpl}

\an{We now aim to analyze the stability of the planar solutions, with their position denoted by $q^*$ and the corresponding chemical potentials denoted by $\mu_\pm^*$. Specifically, we consider a perturbed interface of the form $w := q^* + \badeps \rad$, where $0 < \badeps < 1$ and $\rad = \rad(t, \hat x)$. The idea is to start from the stationary front characterized by $q^*$, as identified above, and proceed with a stability analysis around this equilibrium.
%We now want to analyse the stability of the planar solutions whose position we denote by $q^*$\anold{. The} corresponding {chemical} potentials are called $\mu_\pm^* $. Namely we consider a perturbed interface of the form $w := q^* + \badeps \rad$, with $0 < \badeps < 1$ and $\rad = \rad(t, \hat x)$. 
Thus}, we define perturbed domains as
\begin{align*}
	\Omega_{\rad,\badeps}^+ & := \{
	x \in \Omega\,\, : \,\,
	{z}  <w(t, \hat x)
	\},
	\an{\quad \text{and} \quad }
	\Omega_{\rad,\badeps}^- := \{
	x \in \Omega \,\, : \,\,
	{z}  > w(t, \hat x)
	\}.
\end{align*}
We make the ansatz
\[
\mu_\pm (t,x) = \mu_\pm^* ({z}) + \badeps u_\pm (t,x),
\]
and demand that those solve the free boundary problem on the perturbed domains, \an{given by}
\begin{subequations}\label{lin:planar}
\begin{alignat}{2}
&- m_+ \Delta (\mu_{+}^* + \badeps u_+) = S_+ - \rho_+ (\mu_+^* + \badeps u_+) \quad && \text{ in } \Omega_{\rad,\badeps}^+, \\
&- m_- \Delta (\mu_{-}^* + \badeps u_-) = S_- - \rho_- (\mu_-^* + \badeps u_-) \quad & & \text{ in } \Omega_{\rad,\badeps}^-, \\
&\mu_{+}^* + \badeps u_+ = \mu_{-}^* + \badeps u_- \quad && \text{ on } \{{z} = w\}, \\
& 2 (\mu_{\pm}^* + \badeps u_{\pm}) = \gamma \beta \kappa \quad && \text{ on } \{{z} = w\}, \\
& - 2 {\cal V} = [m \nabla (\mu^* + \badeps u)]_{-}^+ \cdot \bm{\nu} + S_I \quad && \text{ on } \{{z} = w\}, \\
& (\mu_{+}^* + \badeps u_+)'(t, 0) = 0, \quad (\mu_-^* + \badeps u_-)'(t, {\cal L}) = 0. &&
\end{alignat}
\end{subequations}
Linearizing the above equations about  the original interface $\{{z} = q^*\}$, while using \eqref{sharp:3prime}--\eqref{ODEplan2}, as well as
\[
\mu_{\pm}^*(w) = \mu_{\pm}^*(q^*) + (\mu_{\pm}^*)'\vert_{{z} = q^*} (w - q^*) + \text{ h.o.t.} = 0 + \badeps (\mu_{\pm}^*)' \vert_{{z} = q^*} \rad + \text{ h.o.t.}, 
\]
\[
(\mu_{\pm}^*)'(w) = (\mu_{\pm}^*)'(q^*) + (\mu_{\pm}^*)''\vert_{{z} = q^*} (w - q^*) + \text{ h.o.t.} = (\mu_{\pm}^*)'(q^*) + \badeps (\mu_{\pm}^*)'' \vert_{{z} = q^*} \rad + \text{ h.o.t.}, 
\]
we obtain the following system for $\rad$ and $u_{\pm}$:
\begin{subequations}\label{perturb:planar}
\begin{alignat}{2}
\label{pertplan:1} &- m_+ \Delta u_+ = - \rho_+  u_+ \quad && \text{ in } \{{z} < q^*\}, \\
\label{pertplan:2} & - m_- \Delta u_- = - \rho_- u_- \quad && \text{ in } \{{z} > q^*\}, \\
\label{pertplan:3} & (\mu_{+}^*)' \vert_{{z} = q^*} \rad + u_+  = (\mu_{-}^*)' \vert_{{z} = q^*} \rad + u_- \quad && \text{ on } \{{z} = q^*\}, \\
\label{pertplan:4} &2( (\mu_{\pm}^*)' \vert_{{z} = q^*} \rad + u_{\pm}) =- \gamma \beta \Delta_{\hat x} \rad \quad && \text{ on } \{ {z} = q^*\}, \\
\label{pertplan:5}  &2 {\dot Y} = -m_+ ((\mu_+^*)'' \vert_{{z} = q^*} \rad + u_+') + m_- ((\mu_-^*)'' \vert_{{z} = q^*} \rad +u_-')  \quad && \text{ on } \{ {z} = q^* \}, \\
\label{pertplan:6} & (u_+)'(t, 0) = 0 , \quad (u_-)'(t, {\cal L}) = 0, && 
\end{alignat}
\end{subequations}
where the linearization of the mean curvature operator gives
\[
\kappa(t, \hat{x}) \approx -\Delta_{\hat{x}} (\badeps \rad(t, {\hat x}))\anold{,}
\]
with $\Delta_{\hat x}$ denoting the Laplace operator with respect to the $(d-1)$-dimensional coordinates $\hat{x}$. Suppressing the explicit dependence on $t$, for each $t$ we make the ansatz
\[
u_{\pm}(x) = v_{\pm}(z) W(\hat x)
\]
and choose $W$ as an eigenfunction of the $\Delta_{\hat x}$-operator with Neumann boundary conditions such that for $\hat{\bm {\ell}} = (\ell_2, \dots, \ell_d) \in \enne_0^{d-1}$,
\begin{equation*}
\begin{cases}
\displaystyle \Delta_{\hat x} W  =  \frac{\zeta_{\hat{\bm\ell}, d}}{(\widetilde{\cal L})^2} W & \text{ in } (0, \widetilde{\cal L})^{d-1},\\
\dn W  = 0 & \text{ on } \partial (0, \widetilde{\cal L})^{d-1},
\end{cases}
\end{equation*}
where $\frac{\zeta_{\hat{\bm\ell}, d}}{(\widetilde{\cal L})^2}$ serves as a corresponding eigenvalue. A possible eigenfunction is
\[
W(x_2, \dots, x_d) = \cos \Big ( \frac{\pi}{\widetilde{\cal L}} \ell_2 x_2 \Big ) \times \cdots \times \cos \Big ( \frac{\pi}{\widetilde{\cal L}} \ell_d x_d \Big )
\]
with
\[
\zeta_{\hat{\bm \ell}, d} = - \pi^2 (\ell_2^2 + \cdots + \ell_d^2).
\]
With this ansatz for $u_{\pm}$ we find that \eqref{pertplan:1}--\eqref{pertplan:2} reduces to 
\[
- m_{\pm} (v_{\pm}'' W + v_{\pm} \Delta_{\hat x} W) = - \rho_{\pm} v_{\pm} W
\]
and this yields
\[
v_{\pm}'' = \Big (\anold{\lambda_\pm^2} -  \frac{\zeta_{\hat{\bm\ell}, d}}{(\widetilde{\cal L})^2} \Big ) v_{\pm} = (\Gamma_{\pm}^{\hat{\bm\ell}})^2 v_{\pm}{,}
\]
where we recall $\anold{\lambda_\pm^2} = \frac{\rho_{\pm}}{m_{\pm}}$ and set $\Gamma_{\pm}^{\hat{\bm\ell}}= \sqrt{\anold{\lambda_\pm^2} -  \frac{\zeta_{\hat{\bm\ell}, d}}{(\widetilde{\cal L})^2}}$. We \an{then} consider
\begin{align*}
v_{+}(z) = a_+ \cosh (\Gamma_{+}^{\hat{\bm\ell}} z), \quad v_{-}(z) = a_- \cosh (\Gamma_{-}^{\hat{\bm\ell}}({\cal L} -  z)),
\end{align*}
for some \an{unknown functions} $a_+(t)$ and $a_-(t)$ to be determined. Similarly, we make the ansatz
\[
\rad(t,\hat{x}) = y(t) W(\hat{x})
\]
and get
\[
\displaystyle \Delta_{\hat x} \rad = y(t) \Delta_{\hat x} W(\hat{x}) = y(t) \frac{\zeta_{\hat{\bm\ell}, d}}{(\widetilde{\cal L})^2} W(\hat{x}) = \frac{\zeta_{\hat{\bm\ell}, d}}{(\widetilde{\cal L})^2} \rad,
\]
which also leads to perturbed interfaces that fulfil a $90^\circ$ angle condition at points where the interface intersects the outer boundary. We now need to ensure that the boundary conditions \eqref{pertplan:3}--\eqref{pertplan:6} are fulfilled. On recalling \eqref{muplan1} and \eqref{muplan2} we see that 
\begin{equation*}
\begin{alignedat}{3}
(\mu_+^*)'(q^*) & = - d_+ \Lp  \tanh(\Lp  q^*), \quad && (\mu_+^*)''(q^*)&& = - d_+ \Lp ^2, \\
(\mu_-^*)'(q^*) & = d_- \Lm  \tanh(\Lm ( {\cal L} - q^*)), \quad && (\mu_-^*)''(q^*) &&= -d_- \Lm ^2,
\end{alignedat}
\end{equation*}
and so \eqref{pertplan:3} and \eqref{pertplan:4} become
\begin{align*}
& - d_+ \Lp  \tanh(\Lp  q^*) + a_+ \cosh (\Gamma_{+}^{\hat{\bm\ell}} q^*) \\
& \quad = d_- \Lm  \tanh(\Lm  ({\cal L} - q^*)) + a_- \cosh (\Gamma_{-}^{\hat{\bm\ell}} ({\cal L} - q^*)) = - \frac{\gamma \beta}{2} \frac{\zeta_{\hat{\bm\ell}, d}}{(\widetilde{\cal L})^2}.
\end{align*}
Hence, we \an{derive the following:}
\begin{align*}
a_+ (q^*)& = \frac{1}{\cosh (\Gamma_{+}^{\hat{\bm\ell}} q^*)} \Big ( d_+ \Lp  \tanh(\Lp  q^*) - \frac{\gamma \beta}{2} \frac{\zeta_{\hat{\bm\ell}, d}}{(\widetilde{\cal L})^2} \Big ), \\
a_-(q^*) & = -\frac{1}{\cosh (\Gamma_{-}^{\hat{\bm\ell}} ({\cal L} - q^*))} \Big ( d_- \Lm  \tanh(\Lm  ({\cal L} - q^*)) + \frac{\gamma \beta}{2} \frac{\zeta_{\hat{\bm\ell}, d}}{(\widetilde{\cal L})^2} \Big ).
\end{align*}
Meanwhile, for \an{the evolution equation} \eqref{pertplan:5}\an{,} we have 
\begin{align*}
2 {\dot Y} & = m_+ d_+ \Lp ^2 \rad - m_- d_- \Lm ^2 \rad \\
& \quad - m_+ a_+(q^*) \Gamma_+^{\hat{\bm\ell}} \sinh(\Gamma_+^{\hat{\bm\ell}} q^*) \rad - m_- a_-(q^*) \Gamma_{-}^{\hat{\bm\ell}} \sinh(\Gamma_-^{\hat{\bm\ell}} ({\cal L} - q^*)) \rad \\
& = (S_+ - S_- -m_+ a_+(q^*) \Gamma_+^{\hat{\bm\ell}} \sinh(\Gamma_+^{\hat{\bm\ell}} q^*) -  m_- a_-(q^*) \Gamma_{-}^{\hat{\bm\ell}} \sinh(\Gamma_-^{\hat{\bm\ell}} ({\cal L} - q^*)) ) \rad,
\end{align*}
where we recall the definition $d_{\pm} = \frac{S_{\pm}}{\rho_{\pm}} = \frac{S_{\pm}}{\anold{\lambda_\pm^2} m_{\pm}}$. The amplification factor is the \an{prefactor coefficient on the \rhs\ of the above identity}
\begin{align}\label{planar:amp}
S_+ - S_- -m_+ a_+(q^*) \Gamma_+^{\hat{\bm\ell}} \sinh(\Gamma_+^{\hat{\bm\ell}} q^*) -  m_- a_-(q^*) \Gamma_{-}^{\hat{\bm\ell}} \sinh(\Gamma_-^{\hat{\bm\ell}} ({\cal L} - q^*)).
\end{align}
For the rest of this section, we consider specific parameters \an{ensuring} the amplification factor \eqref{planar:amp} is positive so as to induce the development of instabilit\anold{ies} from interface perturbations. We fix 
%$d=2$ so that $\Omega = (0,{\cal L}) \times (0, \widetilde{\cal L})$, 
$\hat{\bm \ell} = (\ell_2,...,\ell_d)$, and 
\[
S_+ = -1, \quad S_{-} = m_{\pm} = \rho_{\pm} = 1{,}
\]
so that the stationary state is $q^* = \frac{{\cal L}}{2}$, along with $d_{+} = -1$, $d_{-} = 1$, $\anold{\lambda_\pm} = 1$,
\begin{align*}
& \Gamma_{\pm}^{\hat{\bm\ell}} = \sqrt{1 + \frac{|\hat{\bm\ell}|^2 \pi^2}{\widetilde{\cal L}^2}}, \quad a_+(q^*) = \frac{-\tanh( \frac{{\cal L}}{2}) + \frac{\gamma \beta}{2}  |\hat{\bm\ell}|^2 \pi^2/\widetilde{{\cal L}}^2 } {\cosh{\Big(} \frac{{\cal L}}{2} \sqrt{1 + |\hat{\bm\ell}|^2 \pi^2/\widetilde{\cal L}^2}{\Big)}}, \quad \\ & 
a_-(q^*) = \frac{-\tanh( \frac{{\cal L}}{2}) +\frac{\gamma \beta}{2}  |\bm\ell|^2 \pi^2/\widetilde{{\cal L}}^2}{\cosh{\Big(} \frac{{\cal L}}{2} \sqrt{1 + |\hat{\bm\ell}|^2 \pi^2/\widetilde{\cal L}^2}{\Big)}}\an{,}
\end{align*}
and the amplification factor \eqref{planar:amp} reduces in this setting to 
\begin{align}\label{planar:amp:2}
-2 + \sqrt{1 + \tfrac{|\hat{\bm\ell}|^2 \pi^2}{\widetilde{{\cal L}}^2}} \tanh \left( \tfrac{{\cal L}}{2}\sqrt{1 + \tfrac{|\hat{\bm\ell}|^2 \pi^2}{\widetilde{{\cal L}}^2}} \right) \left(2 \tanh \left( \tfrac{{\cal L}}{2} \right) - 
%\frac
{\gamma \beta}
%{2} 
\tfrac{|\hat{\bm\ell}|^2 \pi^2}{\widetilde{\cal L}^2} \right).
\end{align}
In the case $\hat{\bm \ell} =\an{(0,...,0)}$, this corresponds to interface perturbations of the form $w = q^* + \badeps$, i.e., translational perturbations\anold{. We} immediately see that the amplification factor becomes
\[
-2 + 2\tanh^2 \Big ( \frac{{\cal L}}{2} \Big ) = -2 \mathrm{sech}^2 \Big ( \frac{{\cal L}}{2} \Big ) < 0.
\]
Thus, we obtain stability with respect to translational perturbations. For $|\hat{\bm\ell}| > 0$, we obtain   interface perturbations of the form $w = q^* + \badeps \cos (\pi \ell_2 x_2 / \widetilde{{\cal L}})\times \last{\cdots }\times \cos (\pi \ell_d x_d / \widetilde{{\cal L}})$. We see that \eqref{planar:amp:2}, for a given perturbation whose wave length is related to  $\hat{\bm\ell}$, can be positive if $\beta< \beta_{\mathrm{crit}}(\hat{\bm\ell})$, where
\begin{align}\label{planar:beta:crit}
\beta_{\mathrm{crit}} (\hat{\bm\ell})= \frac{2}{\gamma} \frac{\widetilde{{\cal L}}^2}{|\hat{\bm\ell}|^2 \pi^2} \left ( \tanh \Big ( \frac{{\cal L}}{2} \Big ) - \frac{1}{\sqrt{1 + |\hat{\bm\ell}|^2 \pi^2/\widetilde{\cal L}^2} \tanh \big ( \frac{{\cal L}}{2} \sqrt{1 + |\hat{\bm\ell}|^2 \pi^2 / \widetilde{\cal L}^2} \big )} \right ).
\end{align}
We remark that in this setting $\beta$ is the modified capillary length \last{$c_l$} introduced \last{in \eqref{cap:lenght}} in Subsection \ref{sec:nondim}, i.e.,
we obtain instability in cases where the modified capillary length is small enough.
As $|\hat{\bm\ell}|^2 = (\ell_2)^2+\cdots+(\ell_d)^2$ with ${\hat{\bm\ell}=(} \ell_2,\dots,\ell_d) \in \mathbb{N}_0^{d-1}$\an{,} we obtain that 
the possible perturbations lead to amplification factors via sum of squares, namely:
\[
|\hat{\bm \ell}|^2 \in \begin{cases}
\{0, 1, 4, 9, 16, \dots \} & \text{ if } d = 2, \\
\{0, 1, 2, 4, 5, 8, 9, 10, \dots \} & \text{ if } d = 3.
\end{cases}
\]

\section{Numerical Computations}

%\an{Let us now proceed to the validation of the analyses carried out using numerical simulations.}
\mod{In this section, we state numerical computations that show several phenomena discussed in the introduction. In particular,  we will observe a suppression of Ostwald ripening, splitting scenarios and instabilities of flat fronts and growing particles. The numerical simulations also support the sharp interface asymptotics,  as we get a good agreement between phase field computations and exact solutions of the sharp interface problem. }
All our numerical simulations are for the quartic potential \eqref{quartic} 
and the interpolation functions \eqref{def:G1}, \eqref{def:G23} and 
\eqref{eq:G4}, for a fixed value $r_c \in (0,1]$.

\subsection{Finite element method}

We assume that $\Omega$ is a polyhedral domain and let $\mathcal{T}_h$ be a regular triangulation of $\Omega$ into disjoint open simplices. Associated with $\mathcal{T}_h$ is the piecewise linear finite element space
\[
S^h = \{ \zeta \in C^0(\overline{\Omega}) \, : \, \zeta \vert_{o} \in P_1(o) \quad \forall o \in \mathcal{T}_h\},
\]
where we denote by $P_1(o)$ the set of all affine linear functions on $o$,
\an{see} \cite{Ciarlet78}. 
Let $(\cdot,\cdot)^h$ be the usual mass lumped $L^2$-inner product on $\Omega$ associated with $\mathcal{T}_h$, and let $\pi^h : C^0(\overline{\Omega}) \to S^h$
be the standard interpolation operator. 
In addition, let $\tau$ denote a chosen uniform time step size. Then our finite element approximation of \eqref{eq:sys21} 
is given as follows. Let $\phi_h^0 \in S^h$, e.g.\an{,} $\phi_h^0 = \pi^h \varphi_0$
if $\varphi_0 \in C^0(\overline\Omega)$. 
Then, for $n \geq 0$, find $(\phi_h^{n+1}, \mu_h^{n+1}) \in S^h \times S^h$ such that
\begin{subequations} \label{eq:FEA}
\begin{alignat}{2}
\frac{1}{\tau}(\phi_h^{n+1}-\phi_h^n, \chi_h)^h + (m(\phi_h^n) \nabla \mu_h^{n+1}, \nabla \chi_h)^h & = (S_\eps(\phi_h^n), \chi_h)^h
\quad &&\forall\chi_h\in S^h, \\
\beta \varepsilon (\nabla \phi_h^{n+1}, \nabla \eta_h) + \frac{\beta}{\eps} (\psi'(\phi_h^{n+1}), \eta_h)^h - (\mu_h^{n+1}, \eta_h)^h & = 0
%- \frac{\beta}{\eps} (\psi_2'(\phi_h^n), \eta_h)^h 
\quad &&\forall\eta_h\in S^h.
\end{alignat}
\end{subequations}
We implemented the scheme \eqref{eq:FEA} with the help of the finite element 
toolbox ALBERTA, see \cite{Alberta}. 
To increase computational efficiency, we employ adaptive meshes, which have a 
finer mesh size $h_{f}$ within the diffuse interfacial regions and a coarser 
mesh size $h_{c}$ away from them. In particular, we use the 
strategy from \cite{voids3d,voids}, and refine an element $o$ if
$\eta_o = |\max_o |\phi_h^n| - 1 | > 0.5$, unless it is already of size $h_f$,
and similarly coarsen an element $o$ if $\eta_o < 0.1$, unless it is already of
size $h_c$. For simplicity we assume from now on that 
$\Omega = \prod_{i=1}^d (0,L_i) \subset \mathbb R^d$, with 
$L_1 \geq L_2 \geq \anold{ \dots \geq } L_d$, and then let $h_f = \frac{L_d}{N_f}$,
$h_c = \frac{L_d}{N_c}$ for two chosen integer parameters $N_f > N_c$. 
Here, unless otherwise specified, for the computations with a phase field
parameter $\eps = (2^k\pi)^{-1}$, $k \in \mathbb N$, we choose
$N_f = 8N_c = 2^{3+k} L_d$. This ensures that the interfacial regions are
accurately resolved, while using a relatively coarser mesh in the pure
regions. For the time discretization we choose $\tau=10^{-3}$, unless stated
otherwise.

The nonlinear system of equations arising at each time level of 
\eqref{eq:FEA} are solved with the help of Newton's method. 
The resulting linear systems at each iteration in two spatial dimensions
are solved by direct factorization using the package UMFPACK, 
see \cite{Davis04}, 
and in three spatial dimensions with a $V$-cycle multigrid solver using a block Gauss--Seidel smoother.

For the initial data $\varphi_0$ we in general choose a diffuse interface
representation of a desired sharp interface, with signed distance function
$d_0 : \overline\Omega \to \mathbb R$. In particular, unless otherwise stated, 
we let
%\footnote{Robert: In the past we also used \eqref{eq:tanh0} with
%$\varphi_+ \not= 1$, $\varphi_- \not= -1$. E.g.\
%$\varphi_- = -1 + \eps \frac{S_-}{K_-}$, where it was best to ensure
%$\eps \frac{S_-}{K_-} < 1 - 3^{-\frac12} \approx 0.42$ to prevent
%spontaneous nucleation in the bulk.}
\begin{equation} \label{eq:tanh0}
 \varphi_0(x) = %\frac12 (\varphi_+ + \varphi_-) 
%+ \frac12 (\varphi_+ - \varphi_-) 
\tanh \Big ( \frac{d_0(x)}{\eps \sqrt{2}} \Big ) .
\end{equation}
%In the past, we also sometimes used the initial data
%\footnote{Robert: Make sure this is never used in the included numerical
%results and then remove this equation.}
%\begin{equation} 
% \varphi_0(x) = \tanh\frac{d_0(x)}{\eps \sqrt{2}} 
%+ \frac12 \frac\eps\beta \mu_0(r_0, R; x),
% \label{eq:tanh0w0}
%\end{equation}
%where $\mu_0$ is the true solution for the potential for 
%a circular interface with radius $r_0$ in the domain $\mathbb B_R(0)$.

For the definition of $S_2$ in \eqref{def:source}, recall \eqref{def:source2},
we need to specify $K_\pm$ and $L$.  
For convenience, for the numerical simulations to follow, we will define the
relations $\rho_\pm= \frac {K_\pm}{2\beta}$, from which $K_\pm$ can then be 
inferred.
Unless otherwise stated, we set $m\an{_\pm}=\rho\an{_\pm}=r_c=1$ and $L=0$ throughout.

\subsection{Numerical computations: Planar solutions and their stability}

We begin with a convergence experiment for the exact solution of a moving 
flat vertical front in $\Omega = (0,1)^2$, whose $x$-position $q(t)$ satisfies
the ODE \eqref{eq:dqdt}. To this end, we solve the ODE numerically and compare
the obtained approximation of $q(t)$ with the position of the diffuse front
for solutions of our finite element scheme \eqref{eq:FEA} for decreasing values
of $\eps$. In particular, for a given $\phi^n_h$ we will compare $q(t_n)$, 
where $t_n = n\tau$, with the unique value $q_h(t_n) \in (0,1)$ such that 
$\phi^n(q_h(t_n), 0) = 0$. 
In particular, for the physical parameters we set 
$\beta=0.1$, $S_-=4$, $S_+=-1$, $\rho_-=0.1$, $\rho_+=1$ and
$L=-1$, and we compute the evolutions over the time interval $[0,1]$. For the
initial position of the flat interface we choose $q(0) = 0.3$.
Then we calculate the evolutions for \eqref{eq:FEA} for $\eps=(2^{k}\pi)^{-1}$,
$k=2,\ldots,6$. A comparison between $q$ and the various $q_h$ can be seen in
Figure~\ref{fig:conv2d}, where we note an excellent agreement between the true
solution $q$ and the numerical approximations $q_h$ when $\eps$ is 
sufficiently small.
%In these experiments we let $N_f = 8N_c = 2^{3+k}$ and $\tau=10^{-3}$.
\begin{figure}
% ../plotflatrho
% ../plotflatconvergence && cp flatconvergence.png ~/tex/glns/ool/figures && scpp flatconvergence.png e23:tex/glns/ool/figures
\center
\includegraphics[angle=-90,width=0.5\textwidth]{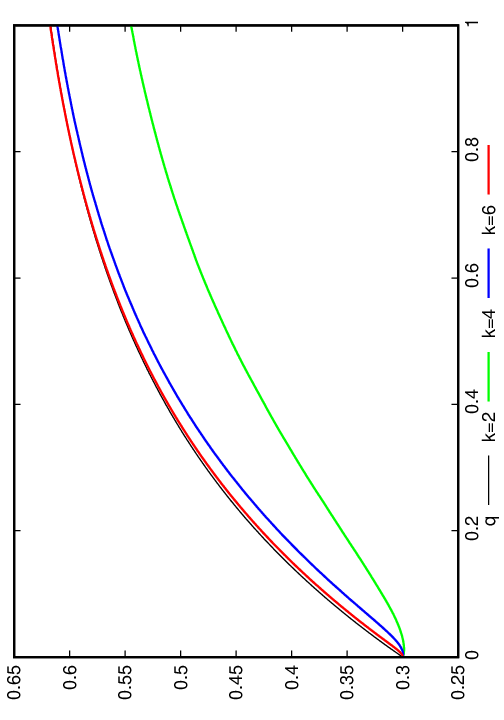}
\caption{Convergence experiment for a moving front in $(0,1)^2$. We compare the
true solution $q$ \mod{of the sharp interface problem} with the discrete approximations $q_h$ \mod{of the Cahn--Hilliard equation} for
$\eps=(2^k\pi)^{-1}$, $k=2,4,6$.}
\label{fig:conv2d}
\end{figure}%
In order to more closely investigate the convergence behaviour of the phase
field solution to the sharp interface solution as $\eps\to0$, we also compute
the error between $q(T)$ and $q_h(T)$ at time $T=1$ and display these values in
Table~\ref{tab:conv2d}. The experimential order of convergence suggests a
quadratic convergence in $\eps$.
This is backed by the fact that the first order correction
$\Phi_1$, solving \eqref{eq:Phi1}, is zero in the case that the curvature $\kappa$ is zero which holds for a flat interface, see also
\cite{GarckeS06}.
\begin{table}
% grep error ../2d.aug*pi.0101.flat03.beta01_4-1_011_L-1/specs.used
\center
\begin{tabular}{r|c|c}
$\eps^{-1}$ & $|(q - q_h)(T)|$ & EOC \\
\hline
$4\pi$  & 7.2539e-02 & --- \\
$8\pi$  & 1.9990e-02 & 1.86\\
$16\pi$ & 6.0682e-03 & 1.72\\
$32\pi$ & 1.4667e-03 & 2.05\\
$64\pi$ & 3.3276e-04 & 2.14\\
% $(128\pi)^{-1}$& 2.0773e-04 & 0.68 \\
\end{tabular}
\caption{Convergence experiment for a moving front in $(0,1)^2$, over the 
time interval $[0,1]$.
We also display the experimental order of convergence (EOC).
}
\label{tab:conv2d}
% ~/hpc_cluster/data/alberta/ool/2d.aug4pi.0101.flat03.beta01_4-1_011_L-1
% ~/hpc_cluster/data/alberta/ool/2d.aug8pi.0101.flat03.beta01_4-1_011_L-1
% ~/hpc_cluster/data/alberta/ool/2d.aug16pi.0101.flat03.beta01_4-1_011_L-1
% ~/hpc_cluster/data/alberta/ool/2d.aug32pi.0101.flat03.beta01_4-1_011_L-1
% ~/hpc_cluster/data/alberta/ool/2d.aug64pi.0101.flat03.beta01_4-1_011_L-1
% ~/hpc_cluster/data/alberta/ool/2d.aug128pi.0101.flat03.beta01_4-1_011_L-1
% ~/hpc_cluster/data/alberta/ool/2d.aug128pi.0101.flat03.beta01_4-1_011_L-1_fine
\end{table}%

For completeness\an{,} we repeat the same convergence experiment also in three dimensions on the
unit cube $\Omega=(0,1)^3$. As expected, the observed results are  nearly
identical to the earlier two-dimensional computations, see
Figure~\ref{fig:conv3d} and Table~\ref{tab:conv3d}.
\begin{figure}
% ../plotflatrho
% ../plotflatconvergence3d && cp flatconvergence3d.png ~/tex/glns/ool/figures && scpp flatconvergence3d.png e23:tex/glns/ool/figures
\center
\includegraphics[angle=-90,width=0.5\textwidth]{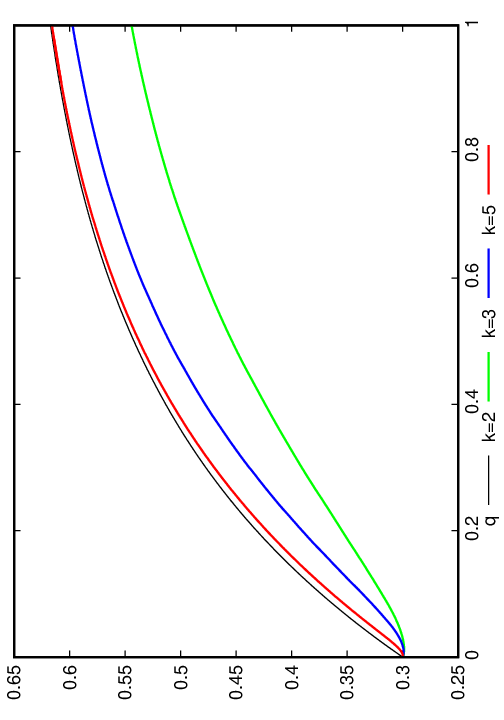}
\caption{Convergence experiment for a moving front in $(0,1)^3$. We compare the
true solution $q$ \mod{of the sharp interface problem} with the discrete approximations $q_h$ \mod{of the Cahn--Hilliard equation } for
$\eps=(2^k\pi)^{-1}$, $k=2,3,5$.}
\label{fig:conv3d}
\end{figure}%
\begin{table}
% grep error ../3d.jan*pi.01.flat03.beta01_4-1_011_L-1/specs.used
\center
\begin{tabular}{r|c|c}
$\eps^{-1}$ & $|(q - q_h)(T)|$ & EOC \\
\hline
$4\pi$  & 7.2942e-02 & --- \\
$8\pi$  & 1.9824e-02 & 1.88\\
$16\pi$ & 5.8611e-03 & 1.76\\
$32\pi$ & 1.2907e-03 & 2.18\\
\end{tabular}
\caption{Convergence experiment for a moving front in $(0,1)^3$, over the 
time interval $[0,1]$.
We also display the experimental order of convergence (EOC).
}
\label{tab:conv3d}
% ~/hpc_cluster/data/alberta/ool/3d.jan4pi.01.flat03.beta01_4-1_011_L-1
% ~/hpc_cluster/data/alberta/ool/3d.jan8pi.01.flat03.beta01_4-1_011_L-1
% ~/hpc_cluster/data/alberta/ool/3d.jan16pi.01.flat03.beta01_4-1_011_L-1
% ~/hpc_cluster/data/alberta/ool/3d.jan32pi.01.flat03.beta01_4-1_011_L-1
\end{table}%

For stability investigations, we let $\Omega=(0,1)^2$ and set
$\beta = 0.1$, $S_\pm = \mp8$. This encourages growth of the $2$-mode, as
can be seen in Figure~\ref{fig:flat01awiggles32pi} for a run with
$\eps=\frac1{32\pi}$. Here as initial data we use a flat front at the middle of
the domain, with an added perturbation of magnitude less than $0.1$
given by a sum of modes from 1 to 20 with
random coefficients. \mod{Computing the amplification factors in \eqref{planar:amp} shows that ${\hat{\bm\ell}}=(\ell_2)=(2)$ is the most unstable mode. This agrees with the numerical solution plotted in Figure~\ref{fig:flat01awiggles32pi}, where the $2$-mode is the one which is most amplified.}

\begin{figure}
% paraview --data=uw_h..vtk
%          choose white, blue, black colour palette, save animation
% cp jan32piflat01a_01_8-8_0[01][012][01].png ~/tex/glns/ool/figures && scpp jan32piflat01a_01_8-8_0[01][012][01].png e23:tex/glns/ool/figures
\center
\mbox{
\includegraphics[angle=-0,width=0.18\textwidth]{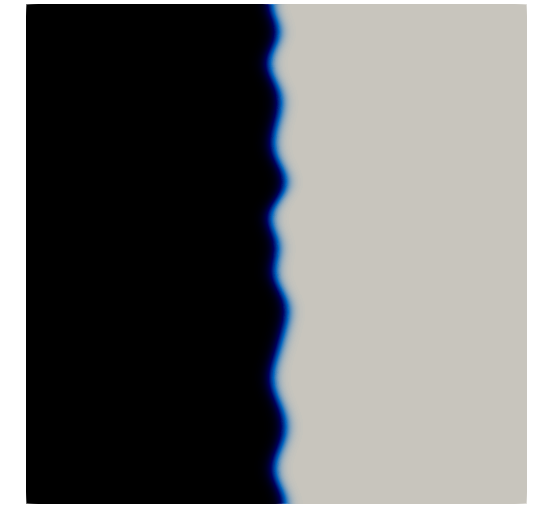}
\includegraphics[angle=-0,width=0.18\textwidth]{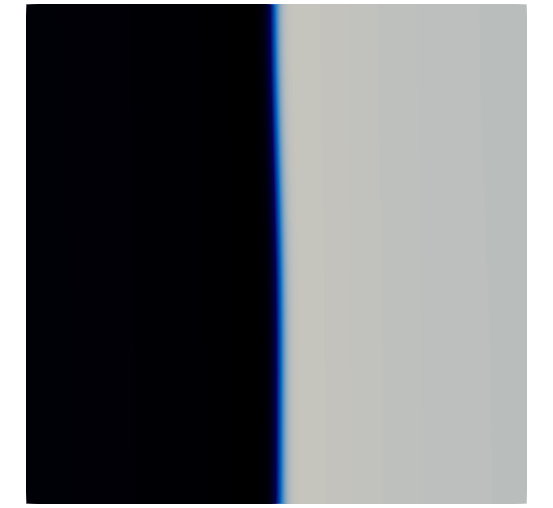}
\includegraphics[angle=-0,width=0.18\textwidth]{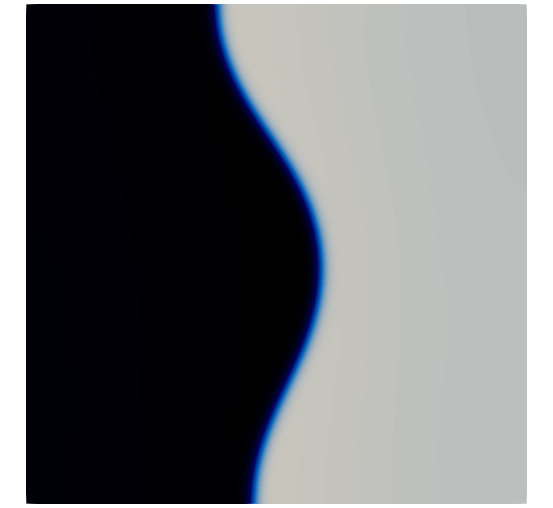}
\includegraphics[angle=-0,width=0.18\textwidth]{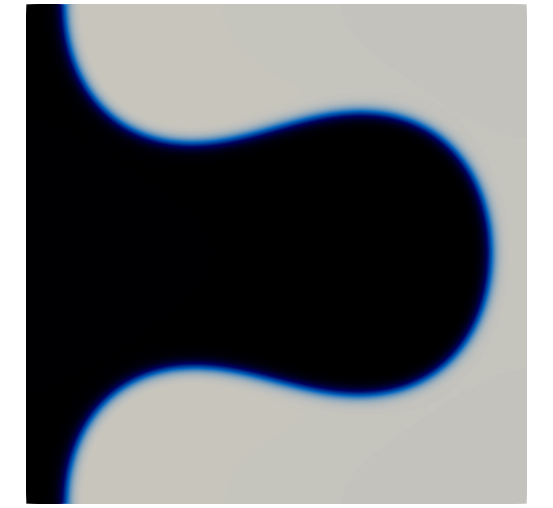}
\includegraphics[angle=-0,width=0.18\textwidth]{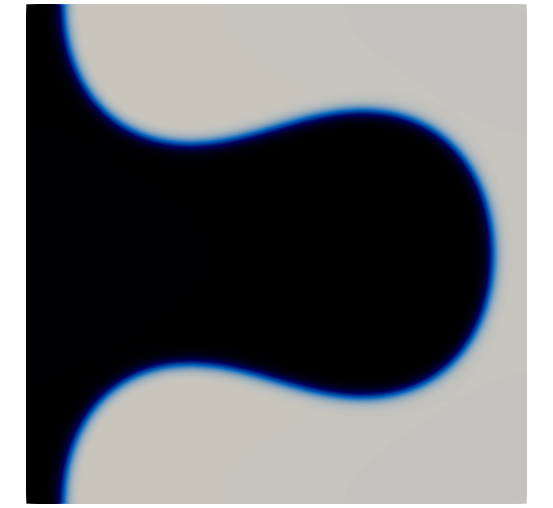}
\includegraphics[angle=-0,width=0.05\textwidth]{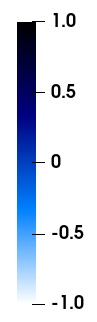}}
\caption{($\eps=\frac1{32\pi}$, $\Omega=(0,1)^2$) 
Evolution for $\beta=0.1$, $S_-=8$, $S_+ = -8$. 
We show the solution at times $t=0,0.1,1,2,10$.} 
\label{fig:flat01awiggles32pi}
% ~/hpc_cluster/data/alberta/ool/2d.jan32pi.0101.flat05wiggles202.beta01_8-8
\end{figure}%

As an alternative, we compute on $\Omega = (0,4) \times (0,2)$ with
$\beta=0.1$, $S_\pm = \mp1.5$.
For the same type of initial perturbation as before, 
this once again encourages growth of the $2$-mode, see
Figure~\ref{fig:flat0402wiggles16pi} for a run with
$\eps=\frac1{16\pi}$. Also in this case, the linear stability analysis   predicts that the  $2$-mode is most unstable.
\begin{figure}
	% paraview --data=uw_h..vtk
	%          choose white, blue, black colour palette, save animation
	% cp jan16piflatw0402i_01_15-15_0[01][0125][0].png ~/tex/glns/ool/figures && scpp jan16piflatw0402i_01_15-15_0[01][0125][0].png e23:tex/glns/ool/figures
	\center
	\mbox{
		\includegraphics[angle=-0,width=0.2\textwidth]{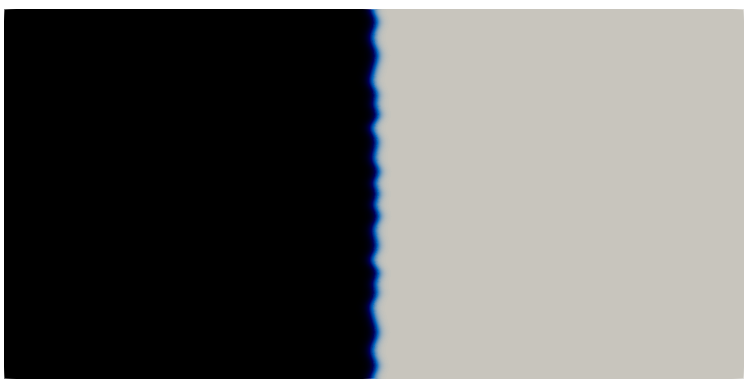}
		\includegraphics[angle=-0,width=0.2\textwidth]{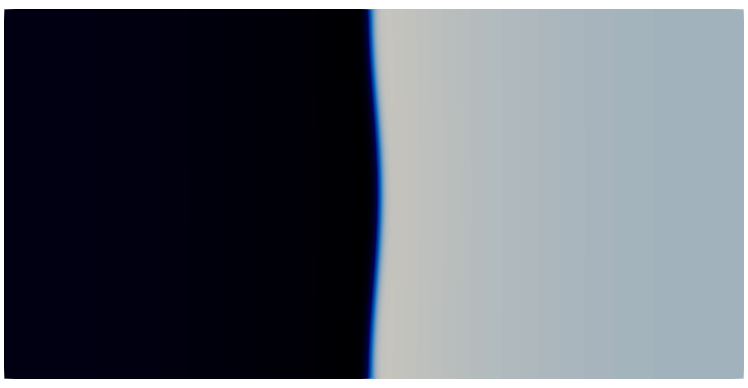}
		\includegraphics[angle=-0,width=0.2\textwidth]{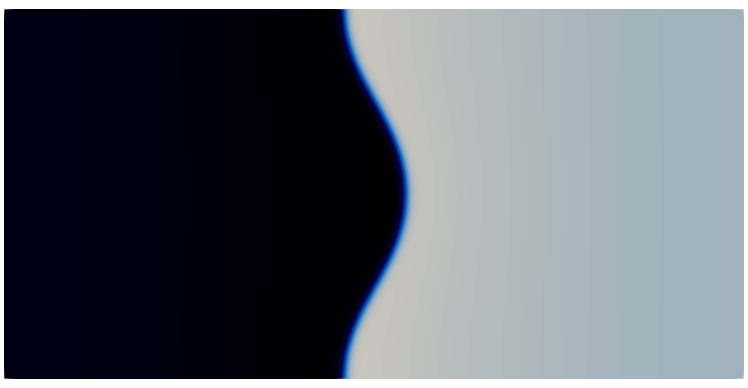}
		\includegraphics[angle=-0,width=0.2\textwidth]{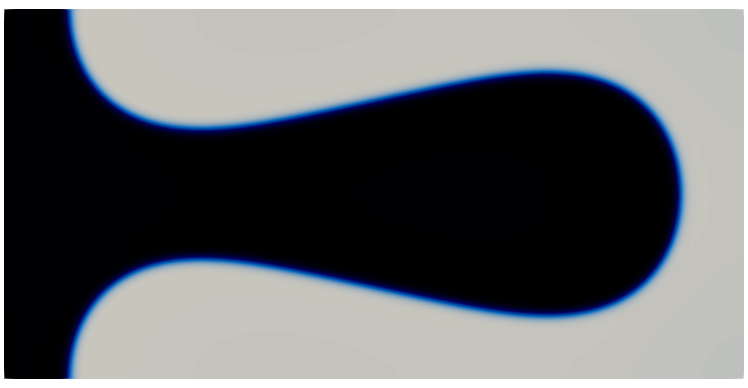}
		\includegraphics[angle=-0,width=0.2\textwidth]{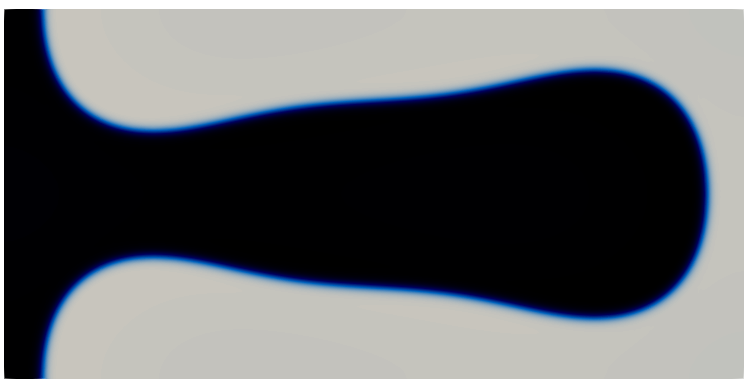}}
	\caption{($\eps=\frac1{16\pi}$, $\Omega= (0,4) \times (0,2)$)
		Evolution for $\beta=0.1$, $S_-=1.5$, $S_+ = -1.5$. 
		We show the solution at times $t=0,1,2,5,10$.} 
	\label{fig:flat0402wiggles16pi}
	% ~/hpc_cluster/data/alberta/ool/2d.jan16pi.0402.flat05wiggles02i.beta01_15-15
\end{figure}%

Moreover, on the domain $\Omega = (0,4) \times (0,1)$ we compute with
$\beta=0.1$, $S_\pm = \mp3$. 
See Figure~\ref{fig:flat0401wiggles32pi} for a simulation with 
$\eps=\frac1{32\pi}$, where we observe the growth of the 1-mode, \mod{which is also predicted by the linear stability analysis when computing  the numbers for the amplification factors  in \eqref{planar:amp} for different values of ${\hat{\bm\ell}}$.} At later times
the long and nearly horizontal interface becomes unstable for higher modes.
Here as initial data we use a flat front at the middle of
the domain, with an added perturbation of magnitude less than $0.025$
given by a sum of modes from 1 to 20 with
random coefficients.
\begin{figure}
% paraview --data=uw_h..vtk
%          choose white, blue, black colour palette, save animation
% cp jan32piflat0401_01_3-3_0[01][0125][0].png ~/tex/glns/ool/figures && scpp jan32piflat0401_01_3-3_0[01][0125][0].png e23:tex/glns/ool/figures
\center
\mbox{
\includegraphics[angle=-0,width=0.2\textwidth]{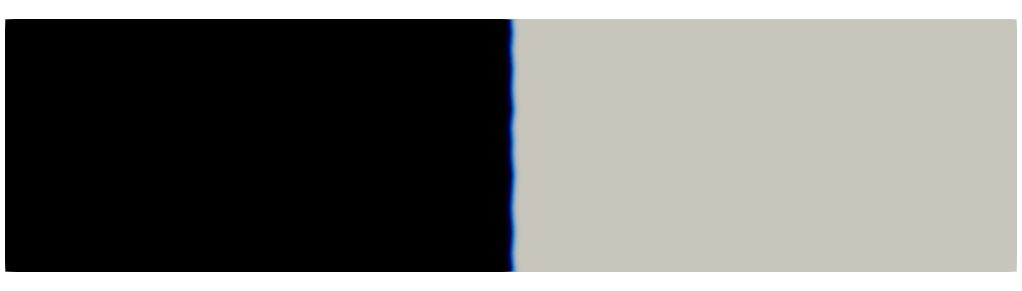}
\includegraphics[angle=-0,width=0.2\textwidth]{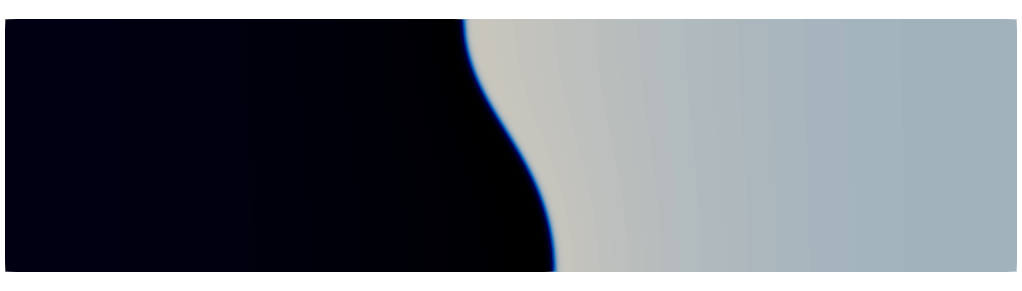}
\includegraphics[angle=-0,width=0.2\textwidth]{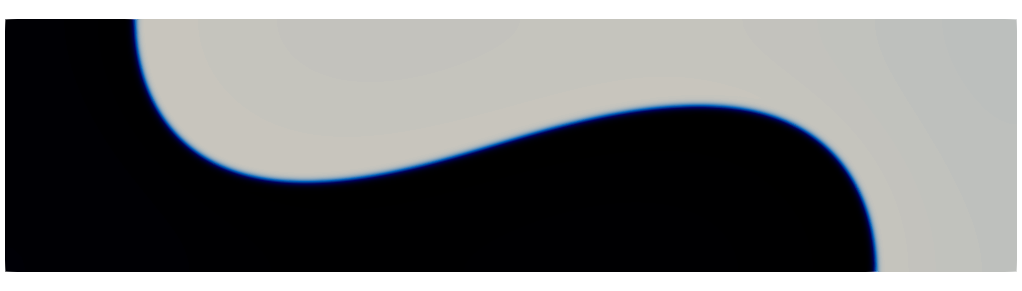}
\includegraphics[angle=-0,width=0.2\textwidth]{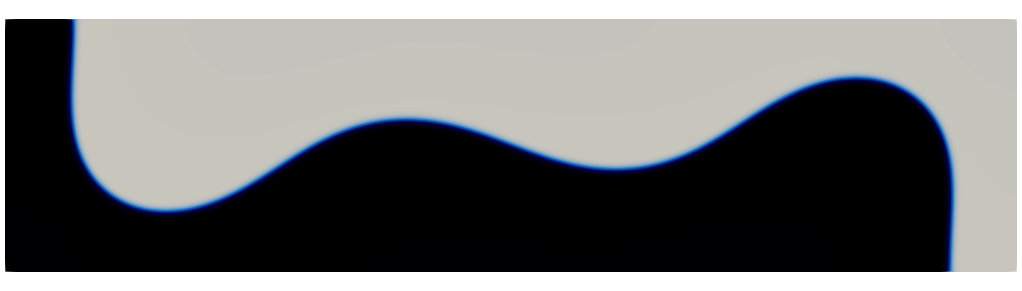}
\includegraphics[angle=-0,width=0.2\textwidth]{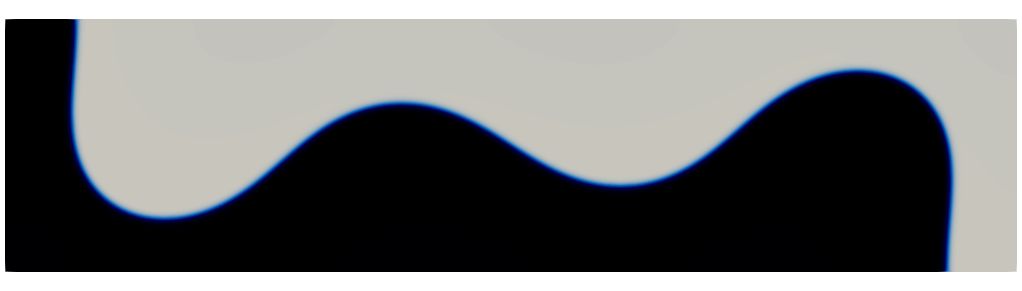}}
\caption{($\eps=\frac1{32\pi}$, $\Omega= (0,4) \times (0,1)$) 
Evolution for $\beta=0.1$, $S_-=3$, $S_+ = -3$. 
We show the solution at times $t=0,1,2,5,10$.} 
\label{fig:flat0401wiggles32pi}
% ~/hpc_cluster/data/alberta/ool/2d.jan32pi.0401.flat05wiggles2.beta01_3-3
\end{figure}%
\last{
Moreover, on the domain $\Omega = (0,2)^2$ we compute with
$\beta=0.01$, $S_\pm = \mp1.3$. 
See Figure~\ref{fig:flat64pi_2x2_13} for a simulation with 
$\eps=\frac1{64\pi}$, where we observe the growth of the 5-mode.
Here as initial data we use a flat front at position $0.5$, 
with an added perturbation of magnitude less than $0.025$
given by a sum of modes from 1 to 20 with random coefficients.
It is worth noticing that the 5-mode is also predicted to grow the most by the amplification factors obtained in the linear stability analysis. In this case, we computed the values in \eqref{planar:amp} for a $q(0)$  which leads to a planar solution which is not stationary.
\begin{figure}
% paraview --data=uw_h..vtk
%          choose white, blue, black colour palette, save animation
% cp aug64piflat_2x2_13left_00[012][01235].png ~/tex/glns/ool/figures && scpp aug64piflat_2x2_13left_00[012][01235].png e23:tex/glns/ool/figures
\center
\mbox{
\includegraphics[angle=-0,width=0.2\textwidth]{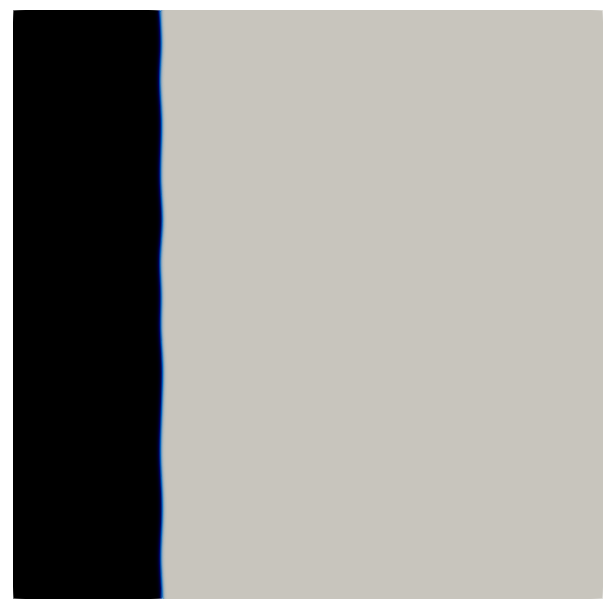}
\includegraphics[angle=-0,width=0.2\textwidth]{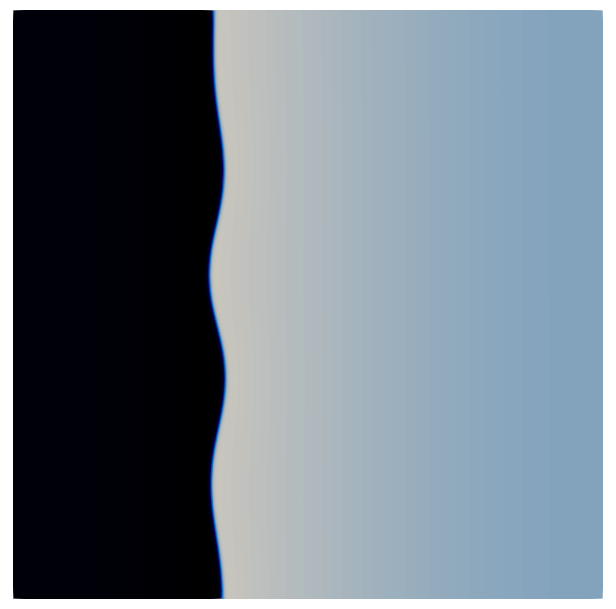}
\includegraphics[angle=-0,width=0.2\textwidth]{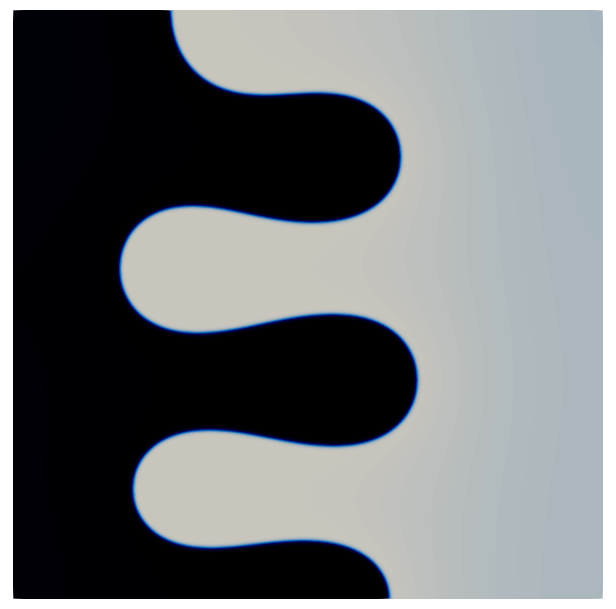}
\includegraphics[angle=-0,width=0.2\textwidth]{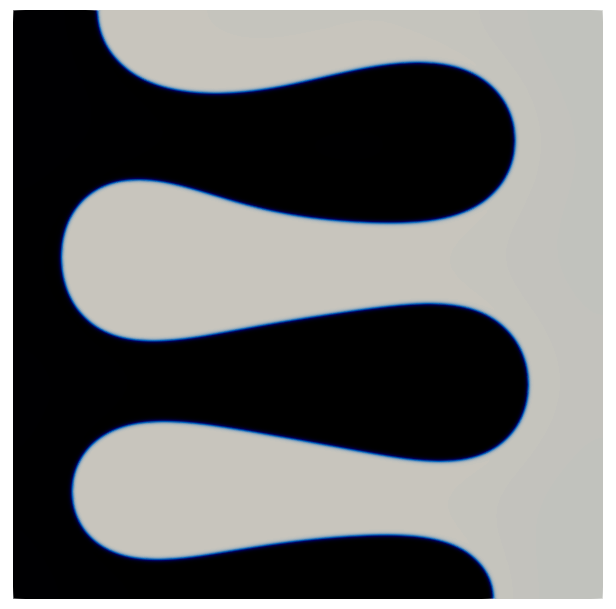}
\includegraphics[angle=-0,width=0.2\textwidth]{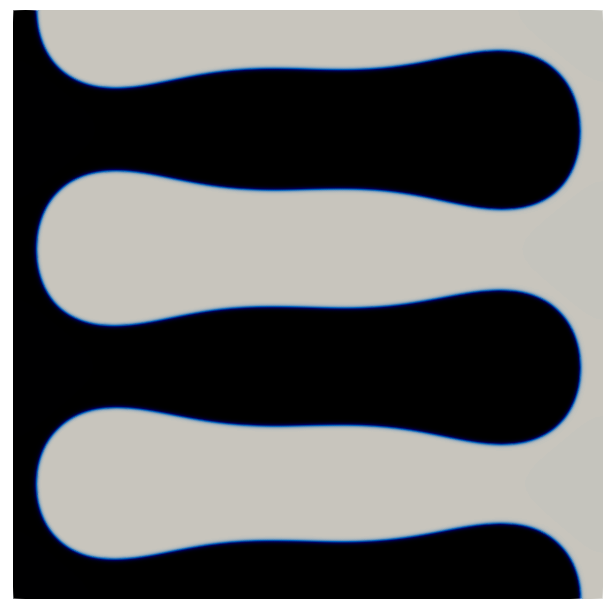}}
\caption{($\eps=\frac1{64\pi}$, $\Omega= (0,2)^2$) 
Evolution for $\beta=0.01$, $S_-=1.3$, $S_+ = 1.3$.
We show the solution at times $t=0, 1, 2, 3, 20$.}
\label{fig:flat64pi_2x2_13}
% ~/hpc_cluster/data/alberta/ool/2d.aug64pi.0202.flat1wiggles.beta001_13-13
\end{figure}%
%\begin{figure}
%% paraview --data=uw_h..vtk
%%          choose white, blue, black colour palette, save animation
%% cp aug64piflat_2x2_13_00[012][01235].png ~/tex/glns/ool/figures && scpp aug64piflat_2x2_13_00[012][01235].png e23:tex/glns/ool/figures
%\center
%\mbox{
%\includegraphics[angle=-0,width=0.2\textwidth]{figures/aug64piflat_2x2_13_0000}
%\includegraphics[angle=-0,width=0.2\textwidth]{figures/aug64piflat_2x2_13_0002}
%\includegraphics[angle=-0,width=0.2\textwidth]{figures/aug64piflat_2x2_13_0003}
%\includegraphics[angle=-0,width=0.2\textwidth]{figures/aug64piflat_2x2_13_0005}
%\includegraphics[angle=-0,width=0.2\textwidth]{figures/aug64piflat_2x2_13_0020}}
%\caption{($\eps=\frac1{64\pi}$, $\Omega= (0,2)^2$) 
%Evolution for $\beta=0.01$, $S_-=1.3$, $S_+ = 1.3$.
%We show the solution at times $t=0, 2, 3, 5, 20$.}
%\label{fig:flat64pi_2x2_13}
%% ~/hpc_cluster/data/alberta/ool/2d.aug64pi.0202.flat1wiggles.beta001_13-13
%\end{figure}%
%Moreover, on the domain $\Omega = (0,2)^2$ we compute with
%$\beta=0.01$, $S_\pm = \mp1.3$. 
%See Figure~\ref{fig:flat64pi_2x2_13} for a simulation with 
%$\eps=\frac1{64\pi}$, where we observe the growth of the 5-mode.
%Here as initial data we use a flat front at position $1$, 
%with an added perturbation of magnitude less than $0.025$
%given by a sum of modes from 1 to 20 with random coefficients.
}% 

For a three-dimensional analogue of Figure~\ref{fig:flat01awiggles32pi}, we
use the parameters $\beta = 0.1$, $S_\pm = \mp4.5$, $m_\pm=0.2$
on the unit cube $\Omega=(0,1)^3$. The initial perturbation of a flat interface
at position $q=0.5$ is made up of a single mode with maximal magnitude 
$0.2$. % = 0.05 * 2 * 2.
The evolution is shown in Figure~\ref{fig:3dflat05wiggleiv}.
%\footnote{Will probably be modified.}
\begin{figure}
% paraview --data=uw_h..vtk
%          choose white, blue, black colour palette, save animation
% cp jan16piflat_1x1x1_cos_00[012][015].png ~/tex/glns/ool/figures && scpp jan16piflat_1x1x1_cos_00[012][015].png e23:tex/glns/ool/figures
\center
\mbox{
\includegraphics[angle=-0,width=0.18\textwidth]{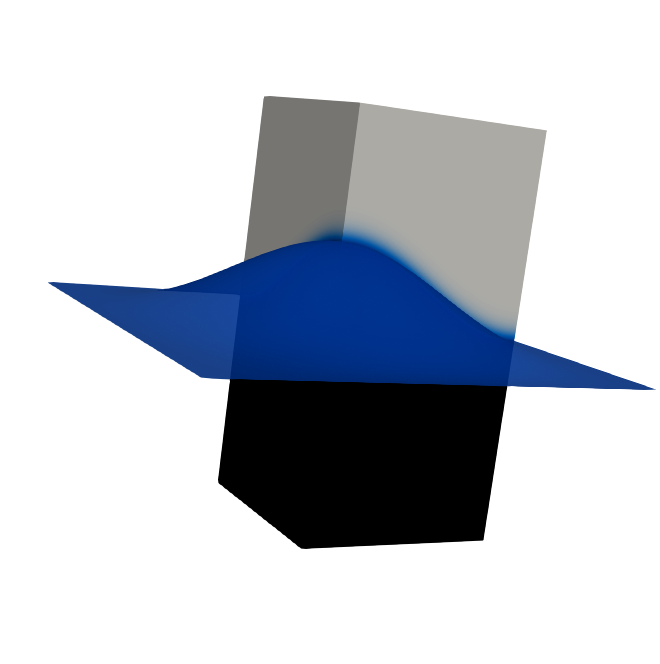}
\includegraphics[angle=-0,width=0.18\textwidth]{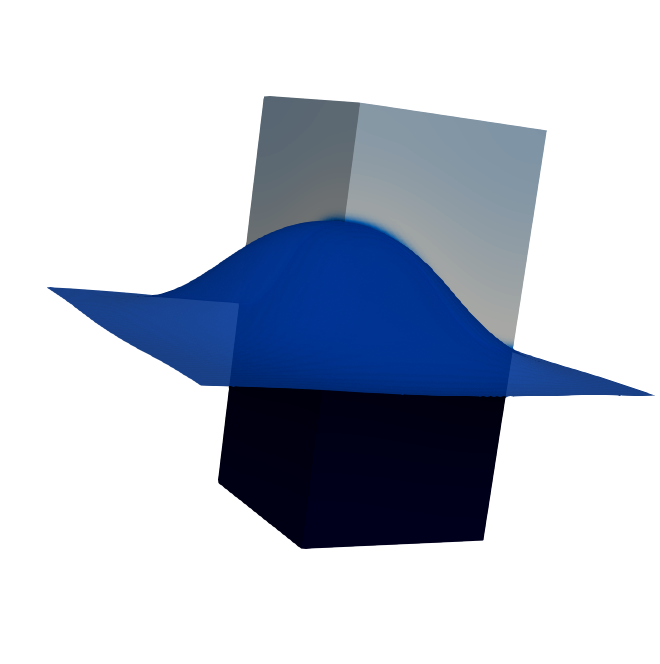}
\includegraphics[angle=-0,width=0.18\textwidth]{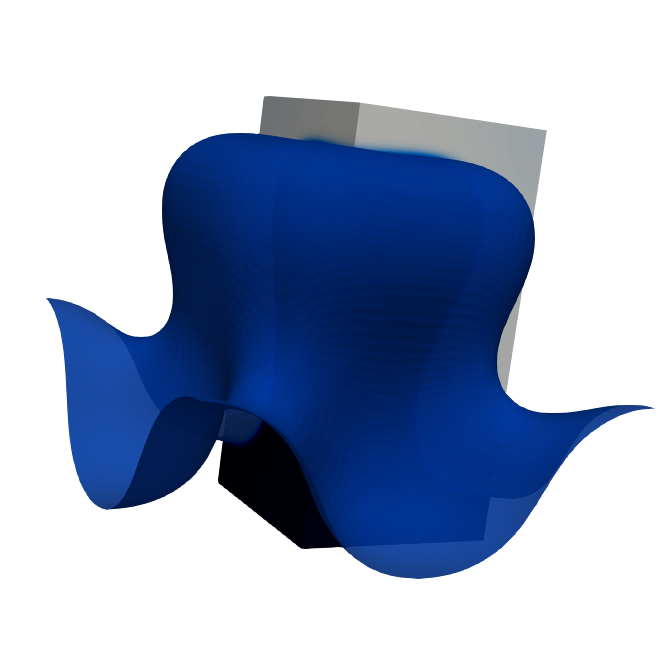}
\includegraphics[angle=-0,width=0.18\textwidth]{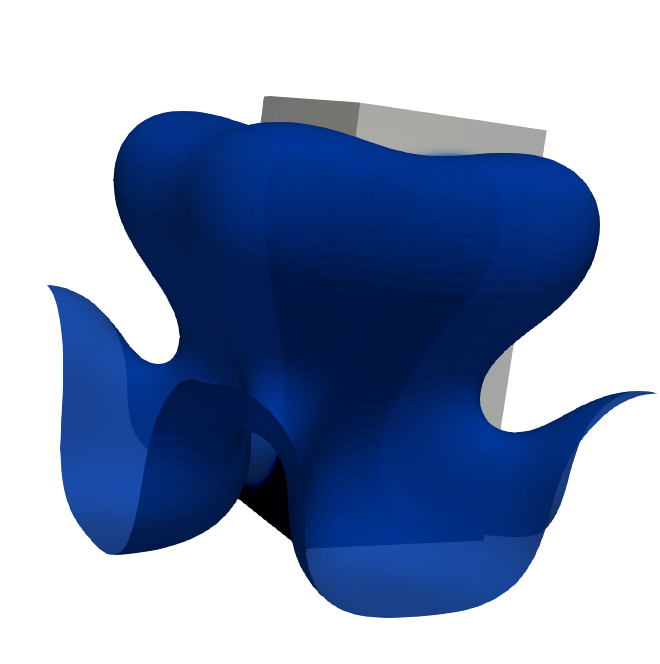}
\includegraphics[angle=-0,width=0.18\textwidth]{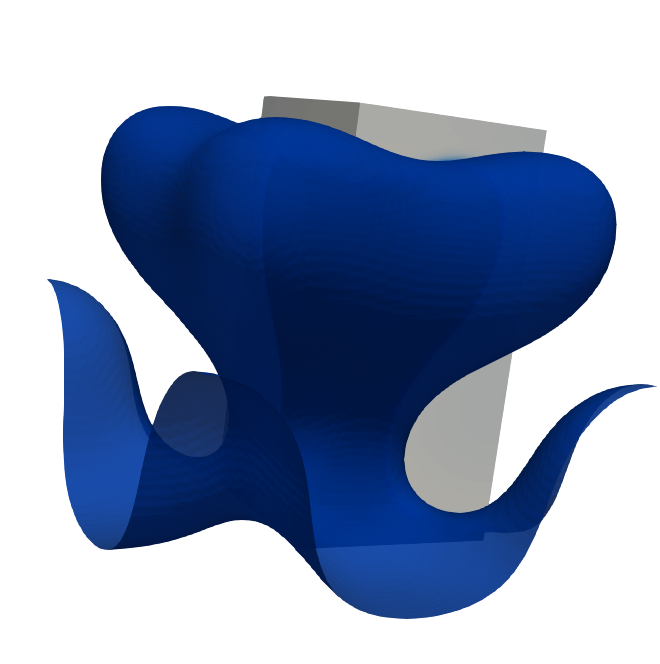}}
\caption{($\eps=\frac1{16\pi}$, $\Omega=(0,1)^3$) 
Evolution for $\beta=0.1$, $m_\pm = 0.2$, $S_-=4.5$, $S_+ = -4.5$.
We show the solution at times $t=0,0.1,0.5,1,1.5$.} 
\label{fig:3dflat05wiggleiv}
% ~/hpc_cluster/data/alberta/ool/3d.jan16pi.01.flat05wiggleiv.beta01_4.5-4.5_m0202 
\end{figure}%

\subsection{Numerical computations: Spinodal decomposition}

In this subsection we are interested in simulations that demonstrate spinodal
decomposition. To this end, we choose for the discrete initial data 
$\varphi_0^h$ a random function with zero mean and values inside 
$[-0.1, 0.1]$. On the domain $\Omega=(0,4)^2$ we then choose the physical
parameters $\beta=0.002$, $S_-=0.25$, $S_+ = -4$, and let 
$\eps = \frac1{16\pi}$. 
The simulation is shown in Figure~\ref{fig:square2_16pinewR2sd_beta0002}.
\begin{figure}
% ../plotool; mv energy.png aug16pinewsquare2sd0002_025-4_e.png
% paraview --data=uw_h..vtk
%          choose white, blue, black colour palette, save animation
% cp aug16pinewsquare2sd0002_025-4_[01][01][0125][015].png ~/tex/glns/ool/figures && scpp aug16pinewsquare2sd0002_025-4_[01][01][0125][015].png e23:tex/glns/ool/figures
\center
\newcommand\localwidth{0.16\textwidth}
\mbox{
\includegraphics[angle=-0,width=\localwidth]{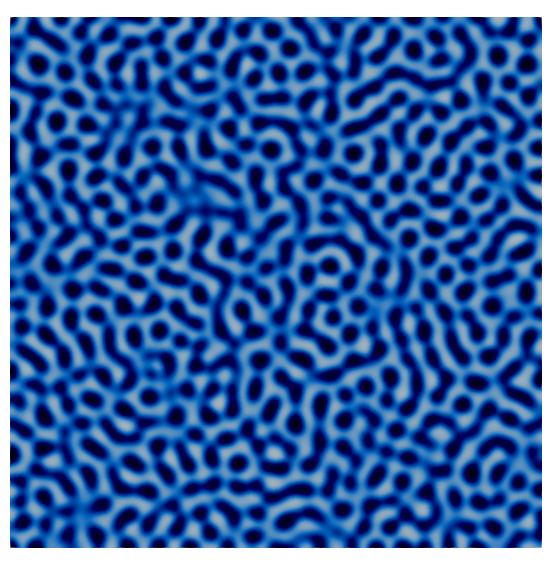}
\includegraphics[angle=-0,width=\localwidth]{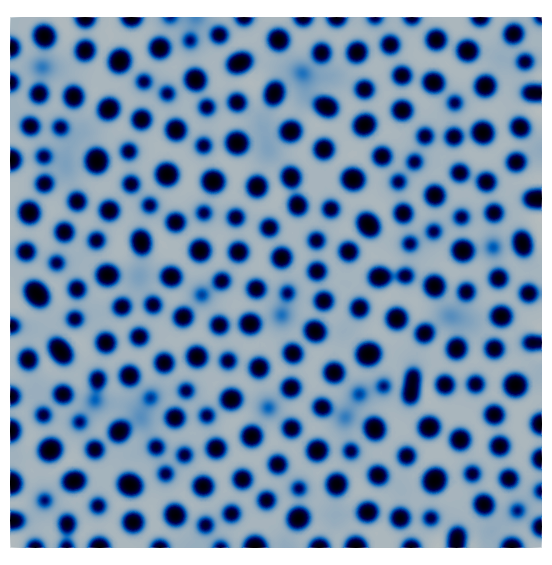}
\includegraphics[angle=-0,width=\localwidth]{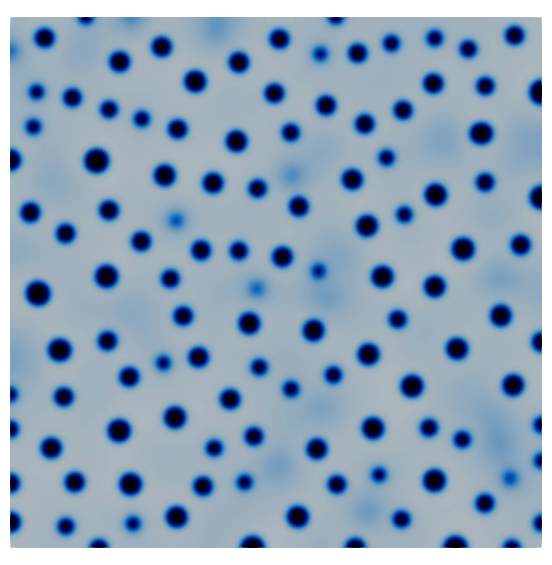}
\includegraphics[angle=-0,width=\localwidth]{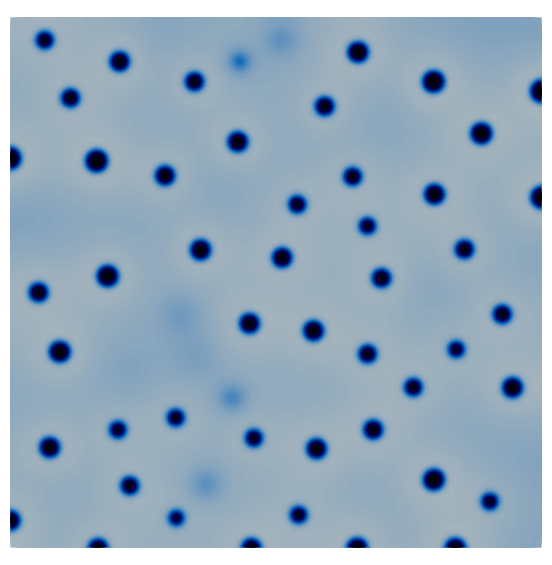}
\includegraphics[angle=-0,width=\localwidth]{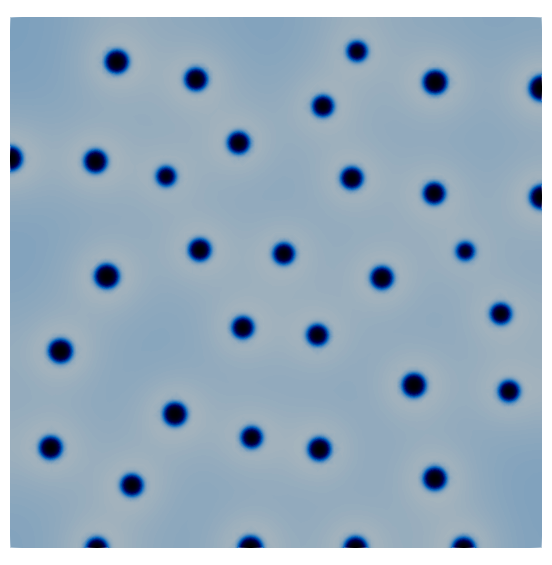}
\includegraphics[angle=-0,width=\localwidth]{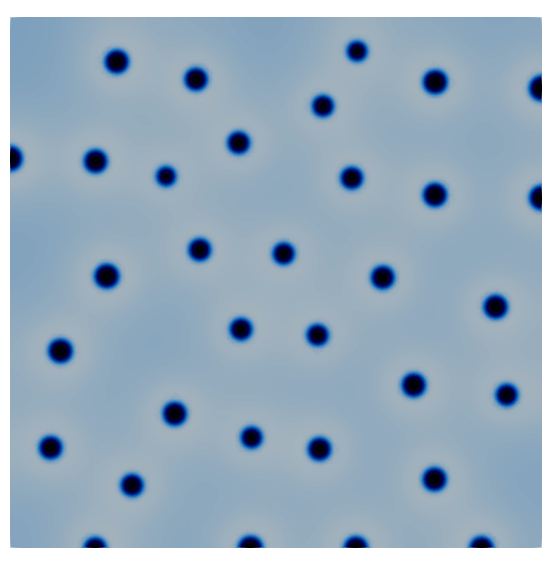}}
\caption{($\eps=\frac1{16\pi}$, $\Omega=(0,4)^2$) 
Evolution for $\beta=0.002$, $S_-=0.25$, $S_+ = -4$. 
We show the solution at times $t=0.1,0.5,1,2,10,100$.
{\anold{In our numerical computations the solution at the final time appears to be a steady state.}}
} 
\label{fig:square2_16pinewR2sd_beta0002}
% ~/hpc_cluster/data/alberta/ool/2d.aug16pi.square2sd.beta2e-3_025-4_new
\end{figure}%
Increasing the value of $\beta$ to $0.02$ yields the results shown in
Figure~\ref{fig:square2_16pinewR2sd_beta002}.
\begin{figure}
% ../plotool; mv energy.png aug16pinewsquare2sd002_025-4_e.png
% paraview --data=uw_h..vtk
%          choose white, blue, black colour palette, save animation
% cp aug16pinewsquare2sd002_025-4_[01][01][0125][015].png ~/tex/glns/ool/figures && scpp aug16pinewsquare2sd002_025-4_[01][01][0125][015].png e23:tex/glns/ool/figures
\center
\newcommand\localwidth{0.16\textwidth}
\mbox{
\includegraphics[angle=-0,width=\localwidth]{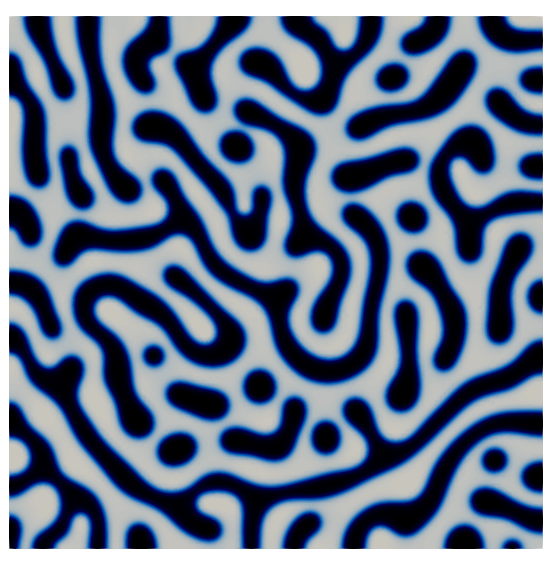}
\includegraphics[angle=-0,width=\localwidth]{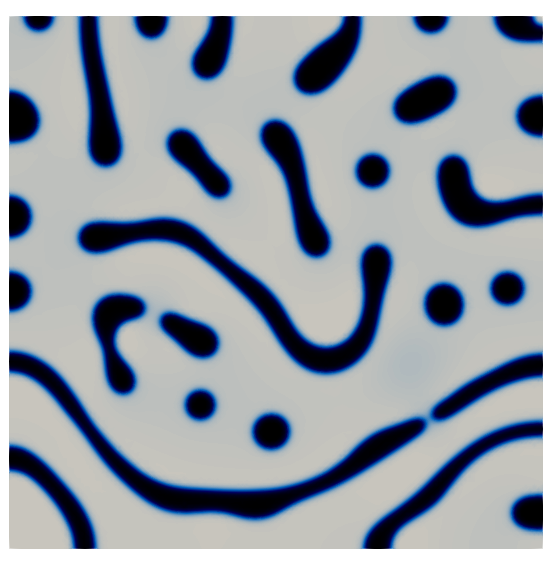}
\includegraphics[angle=-0,width=\localwidth]{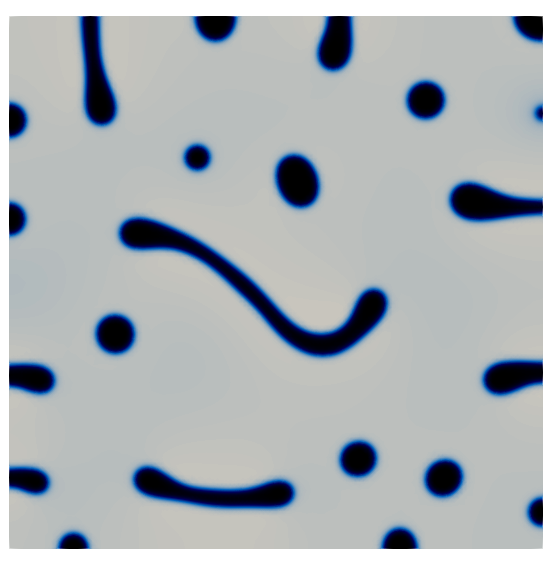}
\includegraphics[angle=-0,width=\localwidth]{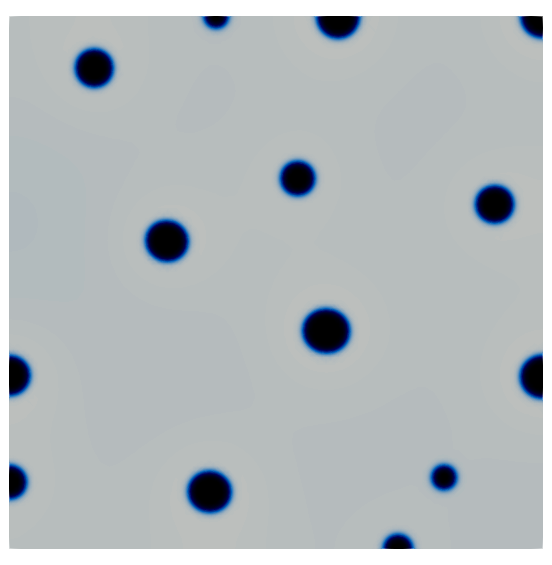}
\includegraphics[angle=-0,width=\localwidth]{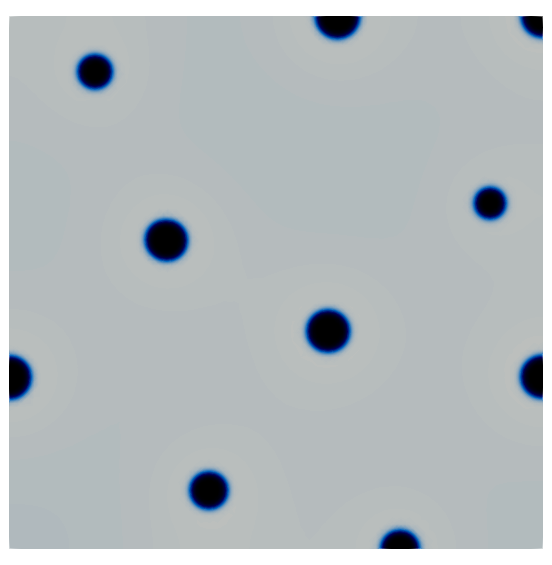}
\includegraphics[angle=-0,width=\localwidth]{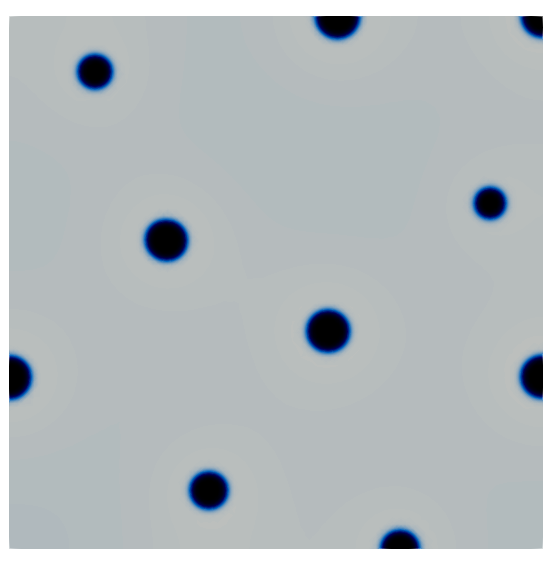}}
\caption{($\eps=\frac1{16\pi}$, $\Omega=(0,4)^2$) 
Evolution for $\beta=0.02$, $S_-=0.25$, $S_+ = -4$. 
We show the solution at times $t=0.1,0.5,1,2,10,100$.
{\anold{In our numerical computations the solution at the final time appears to be a steady state.}}
} 
\label{fig:square2_16pinewR2sd_beta002}
% ~/hpc_cluster/data/alberta/ool/2d.aug16pi.square2sd.beta2e-2_025-4_new
\end{figure}%
\begin{comment}
The same experiment on the unit square $\Omega=(0,1)^2$ leads to the evolution
shown in Figure~\ref{fig:square_16pinewR2sd_beta002}. 
\begin{figure}
% ../plotool; mv energy.png aug16pinewsquaresd002_025-4_e.png
% paraview --data=uw_h..vtk
%          choose white, blue, black colour palette, save animation
% cp aug16pinewsquaresd002_025-4_0[01][015][01235].png ~/tex/glns/ool/figures && scpp aug16pinewsquaresd002_025-4_0[01][015][01235].png e23:tex/glns/ool/figures
\center
\newcommand\localwidth{0.16\textwidth}
\mbox{
\includegraphics[angle=-0,width=\localwidth]{figures/aug16pinewsquaresd002_025-4_0001}
\includegraphics[angle=-0,width=\localwidth]{figures/aug16pinewsquaresd002_025-4_0002}
\includegraphics[angle=-0,width=\localwidth]{figures/aug16pinewsquaresd002_025-4_0005}
\includegraphics[angle=-0,width=\localwidth]{figures/aug16pinewsquaresd002_025-4_0012}
\includegraphics[angle=-0,width=\localwidth]{figures/aug16pinewsquaresd002_025-4_0013}
\includegraphics[angle=-0,width=\localwidth]{figures/aug16pinewsquaresd002_025-4_0100}
}
\caption{($\eps=\frac1{16\pi}$, $\Omega=(0,1)^2$) 
Evolution for $\beta=0.02$, $S_-=0.25$, $S_+ = -4$. 
We show the solution at times $t=0.1,0.2,0.5,1.2,1.3,10$.
{\anold{In our numerical computations the solution at the final time appears to be a steady state.}}
}
\label{fig:square_16pinewR2sd_beta002}
% ~/hpc_cluster/data/alberta/ool/2d.aug16pi.squaresd.beta2e-2_025-4_new
\end{figure}%
\end{comment}

To understand the observed behaviour at the end of the simulations shown in
Figures~\ref{fig:square2_16pinewR2sd_beta0002} and 
\ref{fig:square2_16pinewR2sd_beta002} a bit better, we consider an experiment
with the same physical parameters as in 
Figure~\ref{fig:square2_16pinewR2sd_beta002}, but starting from three circular
initial blobs with radii $0.29$, $0.3$ and $0.31$. As can be seen from the
evolution shown in Figure~\ref{fig:3balls2d}, the three blobs very quickly
reduce in size and move apart from each other. Soon after they settle on an
arrangement that in our numerical computations is a steady state.
\begin{figure}
% ../plotool; mv energy.png aug16pi3balls_002_025-4_e.png
% paraview --data=uw_h..vtk
%          choose white, blue, black colour palette, save animation
% cp aug16pi3balls_002_025-4_0[01][01][01].png ~/tex/glns/ool/figures && scpp aug16pi3balls_002_025-4_0[01][01][01].png e23:tex/glns/ool/figures
\center
\newcommand\localwidth{0.19\textwidth}
\mbox{
\includegraphics[angle=-0,width=\localwidth]{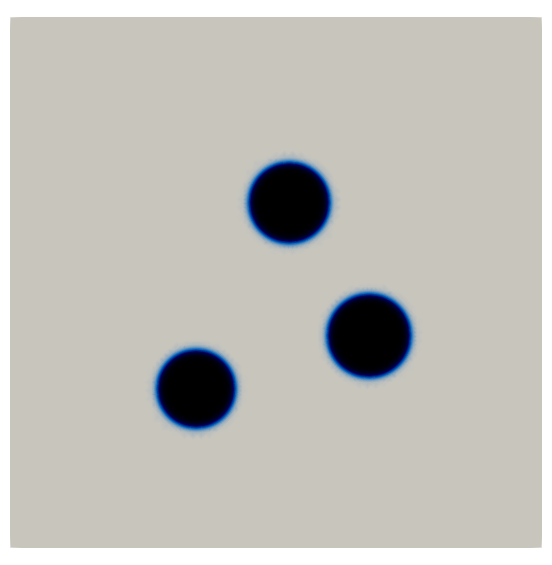}
\includegraphics[angle=-0,width=\localwidth]{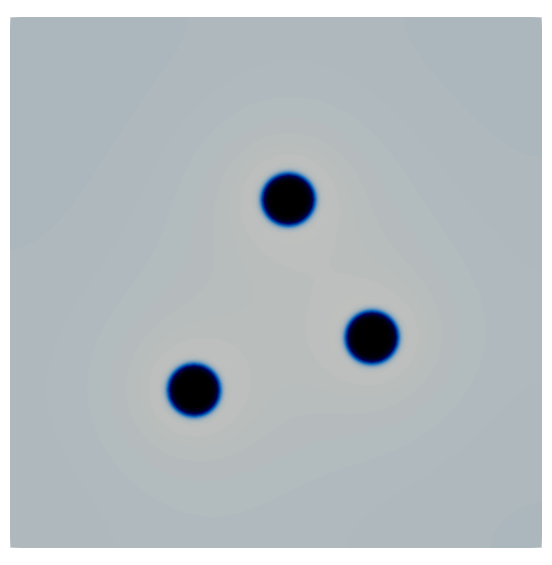}
\includegraphics[angle=-0,width=\localwidth]{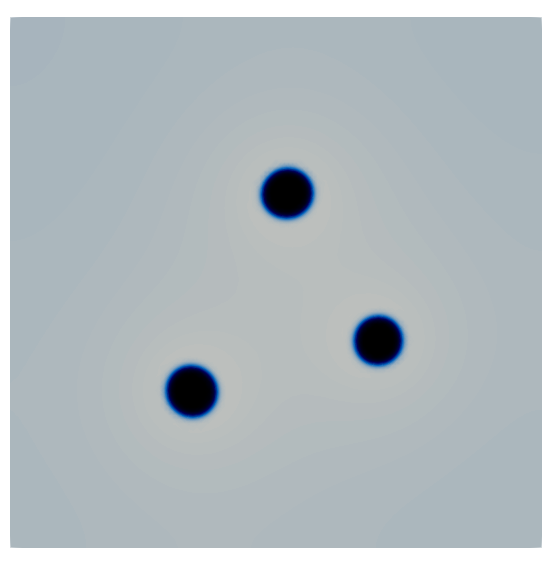}
\includegraphics[angle=-0,width=\localwidth]{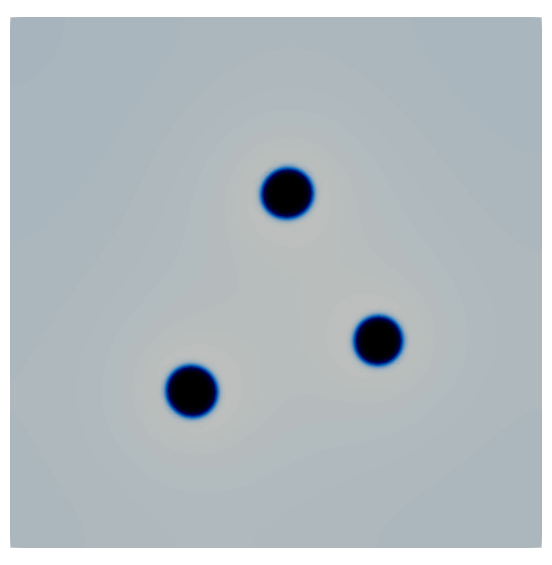}
}
\caption{($\eps=\frac1{16\pi}$, $\Omega=(0,4)^2$) 
Evolution for $\beta=0.02$, $S_-=0.25$, $S_+ = -4$. 
We show the solution at times $t=0,1,10,100$.
{\anold{In our numerical computations the solution at the final time appears to be a steady state.}}
} 
\label{fig:3balls2d}
% ~/hpc_cluster/data/alberta/ool/2d.aug16pi.square2balls3.beta2e-2_025-4_new
\end{figure}%

Next we consider some analogous simulations in 3d. For the simulations in
Figures~\ref{fig:3d_8pisd_beta002},
% \ref{fig:3d_8pisd_beta001}, 
\ref{fig:3d_8pisd_beta002_1-4} and \ref{fig:3d_8pisd_beta001_1-4} we let
$\Omega=(0,1)^3$ and always choose $S_+=-4$. For the remaining physical
parameters we choose $(\beta,S_-) = (0.02, 0.25), (0.1, 0.25), (0.02, 1),
(0.1,1)$, respectively. For the phase field parameter we choose
$\eps=\frac1{8\pi}$, and set $\tau=10^{-4}$.
\begin{figure}
% paraview --data=uw_h..vtk
%          choose white, blue, black colour palette, save animation
% cp 3d8pisd002_025-4*_00[0-5][0125].png ~/tex/glns/ool/figures && scpp 3d8pisd002_025-4*_00[0-5][0125].png e23:tex/glns/ool/figures
\center
\mbox{
\includegraphics[angle=-0,width=0.2\textwidth]{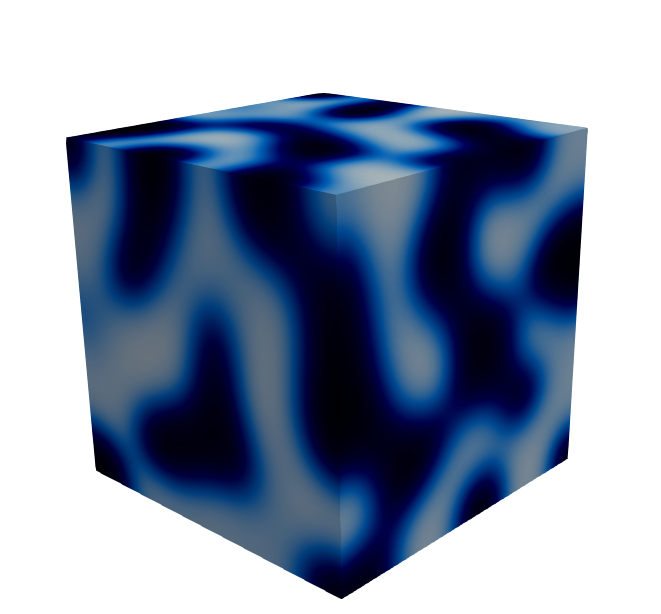}
\includegraphics[angle=-0,width=0.2\textwidth]{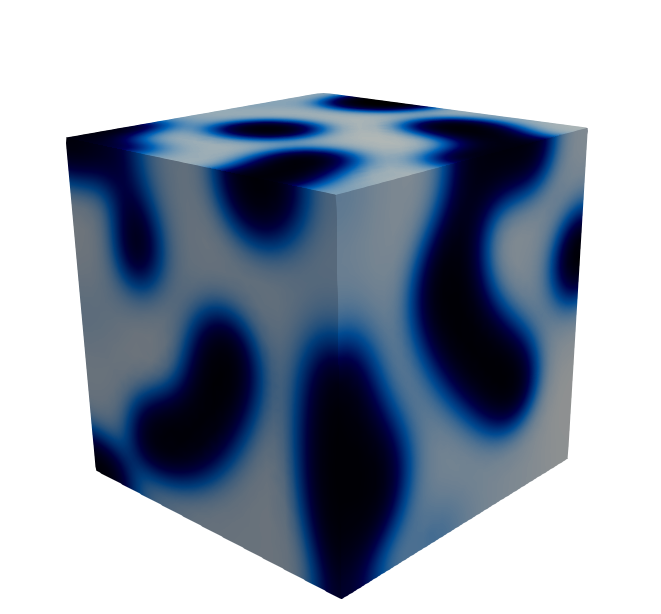}
\includegraphics[angle=-0,width=0.2\textwidth]{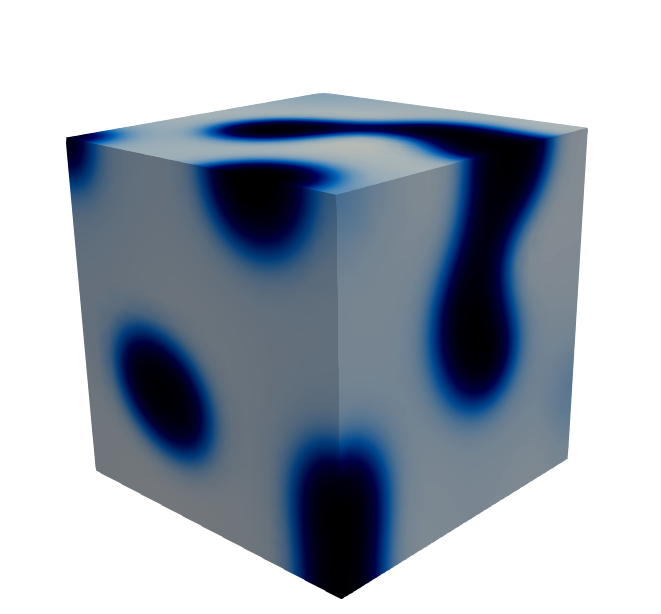}
\includegraphics[angle=-0,width=0.2\textwidth]{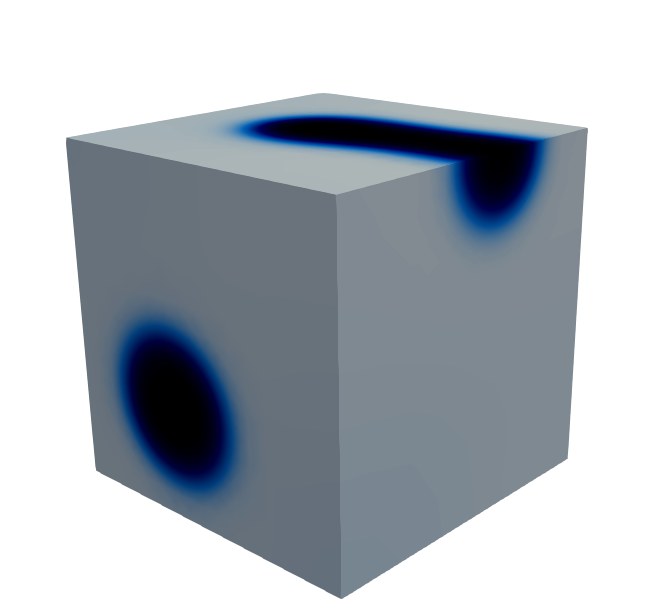}
\includegraphics[angle=-0,width=0.2\textwidth]{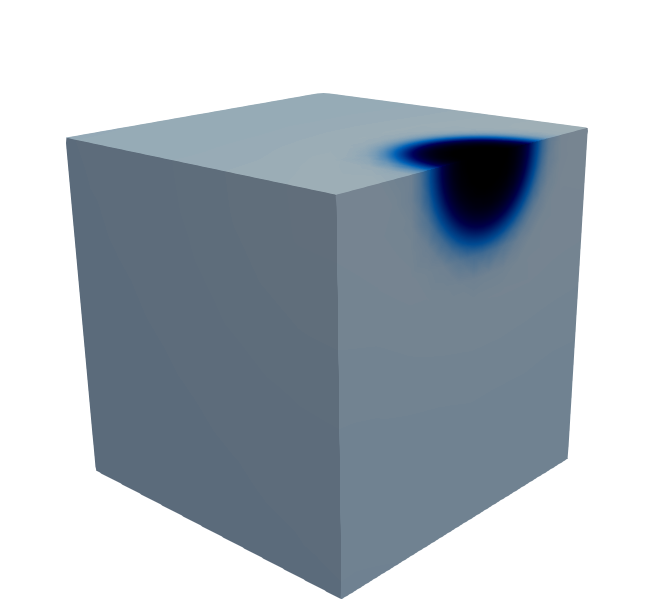}}
\mbox{
\includegraphics[angle=-0,width=0.2\textwidth]{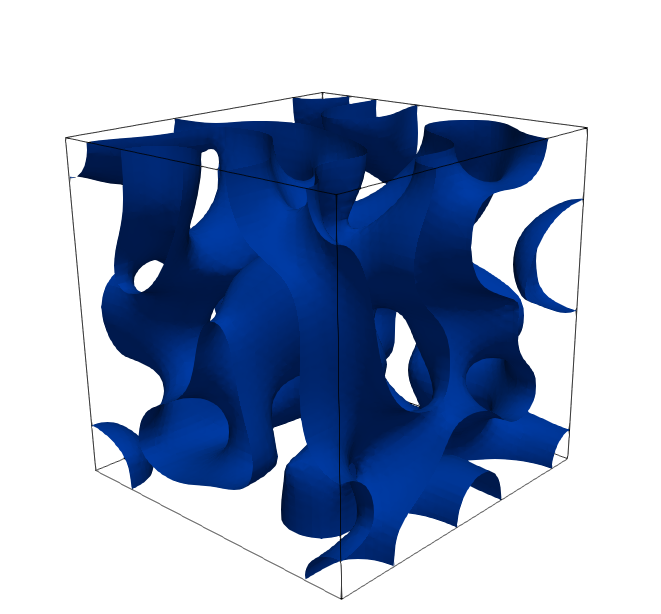}
\includegraphics[angle=-0,width=0.2\textwidth]{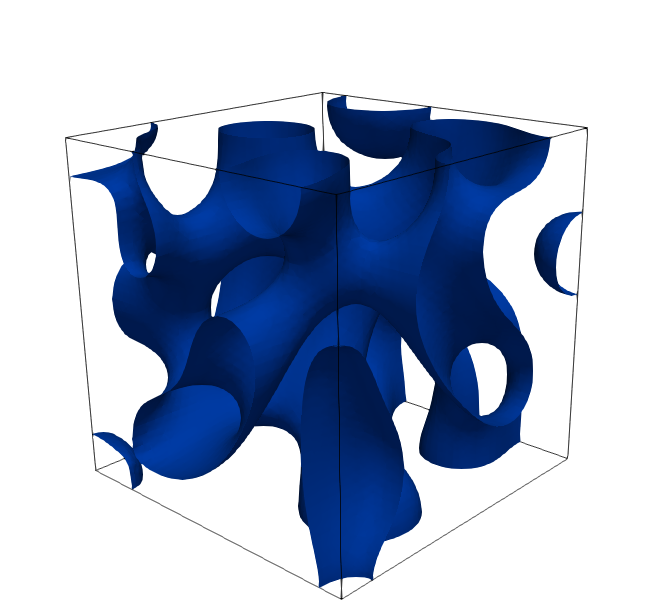}
\includegraphics[angle=-0,width=0.2\textwidth]{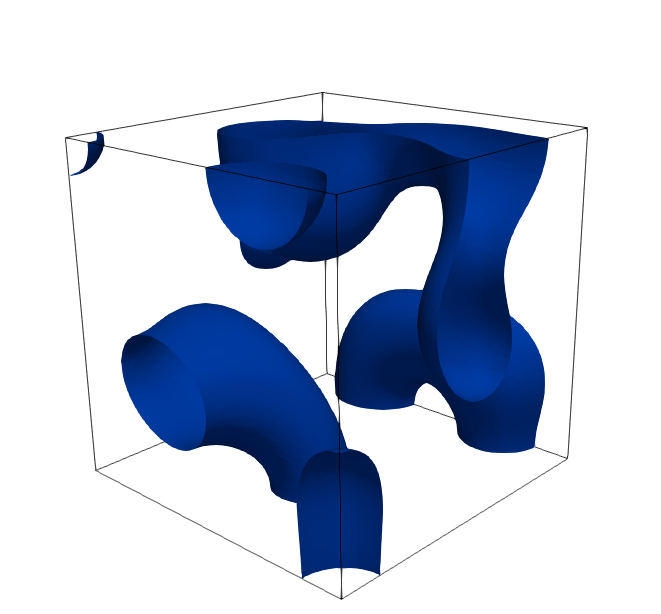}
\includegraphics[angle=-0,width=0.2\textwidth]{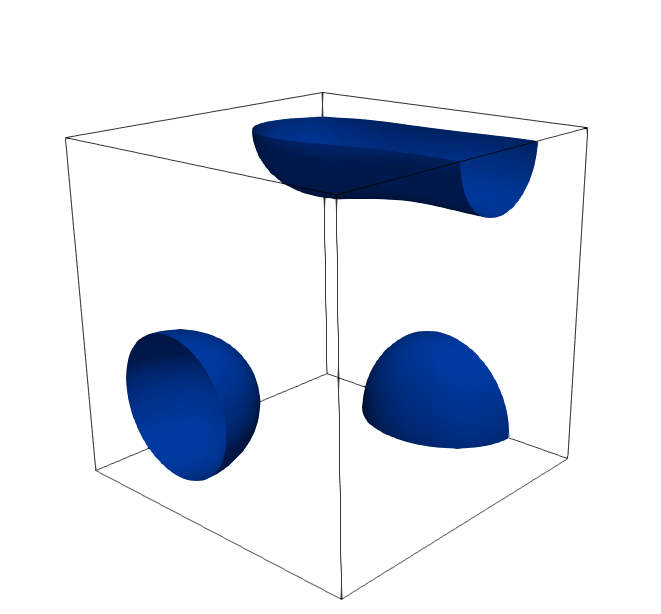}
\includegraphics[angle=-0,width=0.2\textwidth]{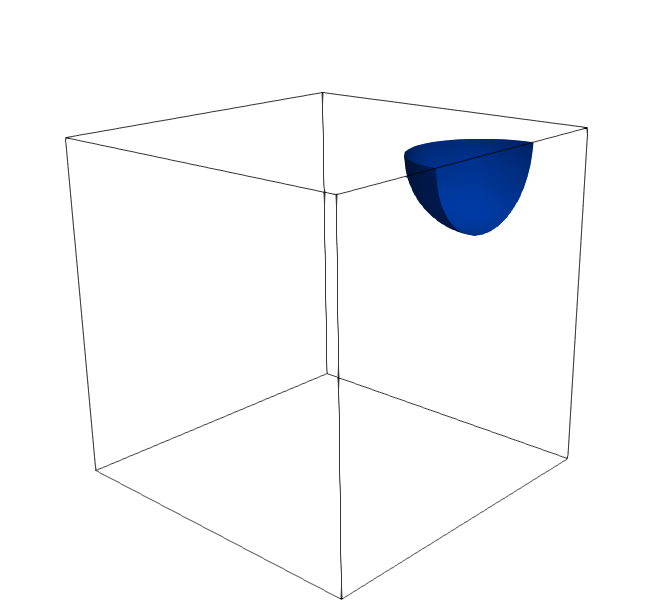}}
\caption{($\eps=\frac1{8\pi}$, $\Omega=(0,1)^3$) 
Evolution for $\beta=0.02$, $S_-=0.25$, $S_+ = -4$. 
We show the solution at times $t=0.1,0.2,0.5,1,5$.}
\label{fig:3d_8pisd_beta002}
% ~/hpc_cluster/data/alberta/ool/3d.8pi.squaresd.beta2e-2_025-4_new
\end{figure}%
%\begin{figure}
%% paraview --data=uw_h..vtk
%%          choose white, blue, black colour palette, save animation
%% cp 3d8pisd001_025-4*_0[0-5][015][02].png ~/tex/glns/ool/figures && scpp 3d8pisd001_025-4*_0[0-5][015][02].png e23:tex/glns/ool/figures
%\center
%\mbox{
%\includegraphics[angle=-0,width=0.2\textwidth]{figures/3d8pisd001_025-4_0002}
%\includegraphics[angle=-0,width=0.2\textwidth]{figures/3d8pisd001_025-4_0010}
%\includegraphics[angle=-0,width=0.2\textwidth]{figures/3d8pisd001_025-4_0050}
%\includegraphics[angle=-0,width=0.2\textwidth]{figures/3d8pisd001_025-4_0100}
%\includegraphics[angle=-0,width=0.2\textwidth]{figures/3d8pisd001_025-4_0500}}
%\mbox{
%\includegraphics[angle=-0,width=0.2\textwidth]{figures/3d8pisd001_025-4ls_0002}
%\includegraphics[angle=-0,width=0.2\textwidth]{figures/3d8pisd001_025-4ls_0010}
%\includegraphics[angle=-0,width=0.2\textwidth]{figures/3d8pisd001_025-4ls_0050}
%\includegraphics[angle=-0,width=0.2\textwidth]{figures/3d8pisd001_025-4ls_0100}
%\includegraphics[angle=-0,width=0.2\textwidth]{figures/3d8pisd001_025-4ls_0500}}
%\caption{($\eps=\frac1{8\pi}$, $\Omega=(0,1)^3$) 
%Evolution for $\beta=0.1$, $S_-=0.25$, $S_+ = -4$. 
%We show the solution at times $t=0.02,0.1,0.5,1,5$.}
%\label{fig:3d_8pisd_beta001}
%% ~/hpc_cluster/data/alberta/ool/3d.8pi.squaresd.beta1e-1_025-4
%\end{figure}%
%\footnote{I removed one 3d figure}
\begin{figure}
% paraview --data=uw_h..vtk
%          choose white, blue, black colour palette, save animation
% cp 3d8pisd002_1-4*_0[0-5][015][015].png ~/tex/glns/ool/figures && scpp 3d8pisd002_1-4*_0[0-5][015][015].png e23:tex/glns/ool/figures
\center
\mbox{
\includegraphics[angle=-0,width=0.2\textwidth]{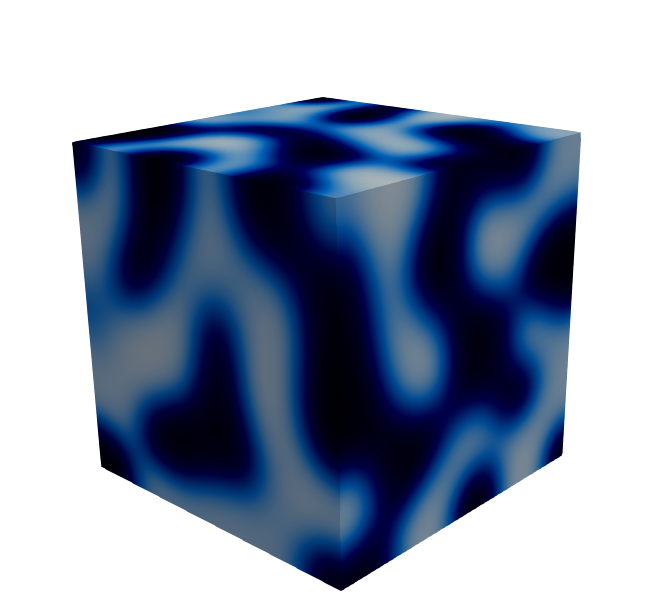}
\includegraphics[angle=-0,width=0.2\textwidth]{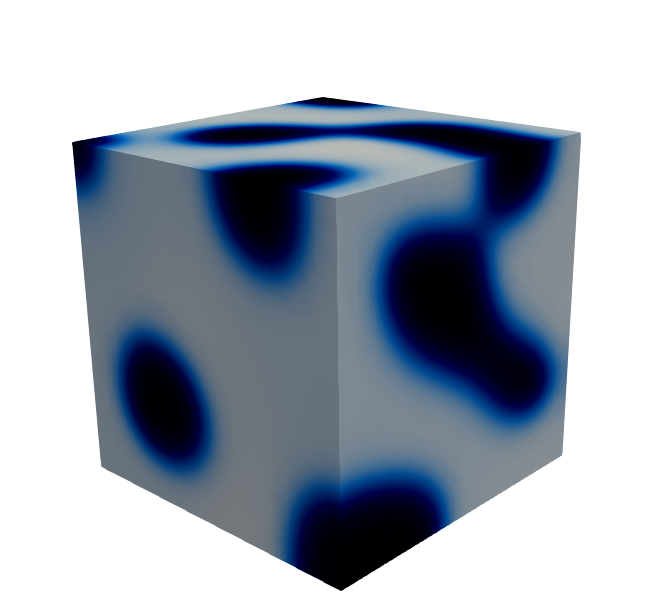}
\includegraphics[angle=-0,width=0.2\textwidth]{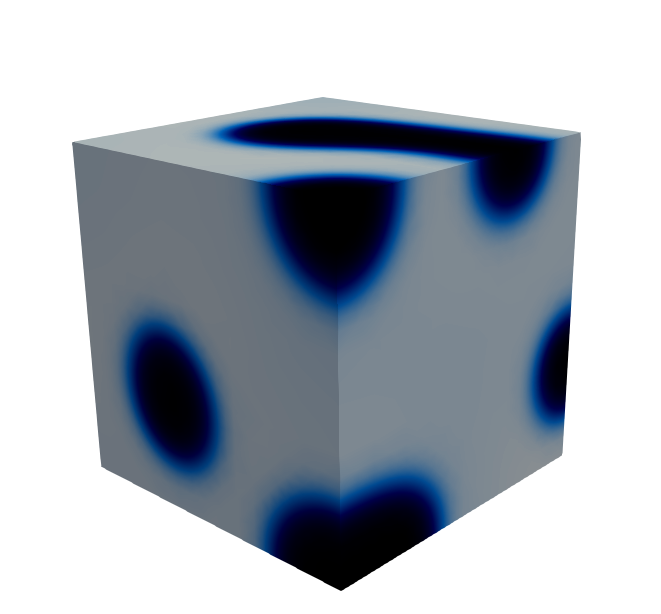}
\includegraphics[angle=-0,width=0.2\textwidth]{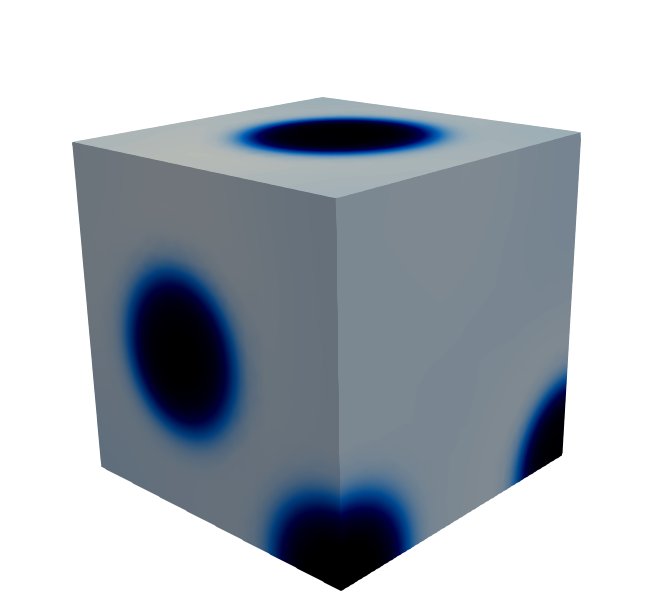}
\includegraphics[angle=-0,width=0.2\textwidth]{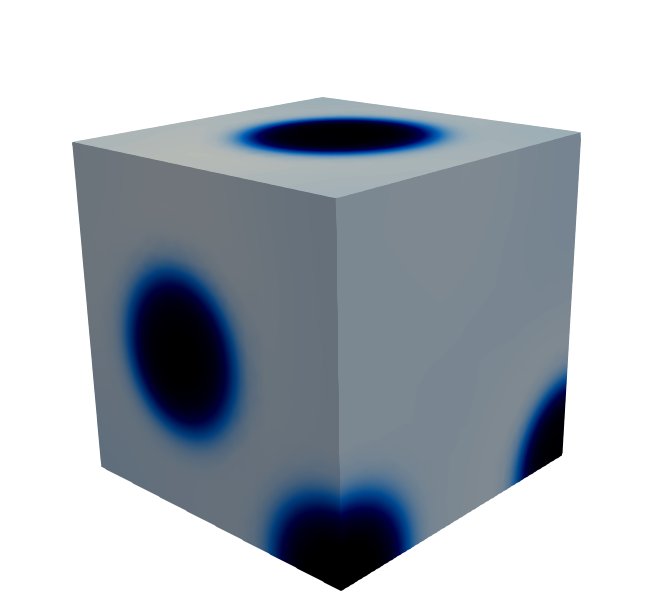}}
\mbox{
\includegraphics[angle=-0,width=0.2\textwidth]{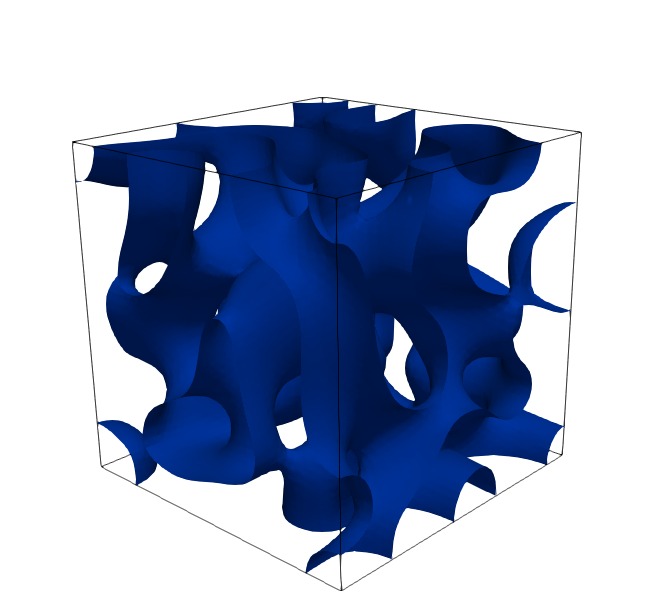}
\includegraphics[angle=-0,width=0.2\textwidth]{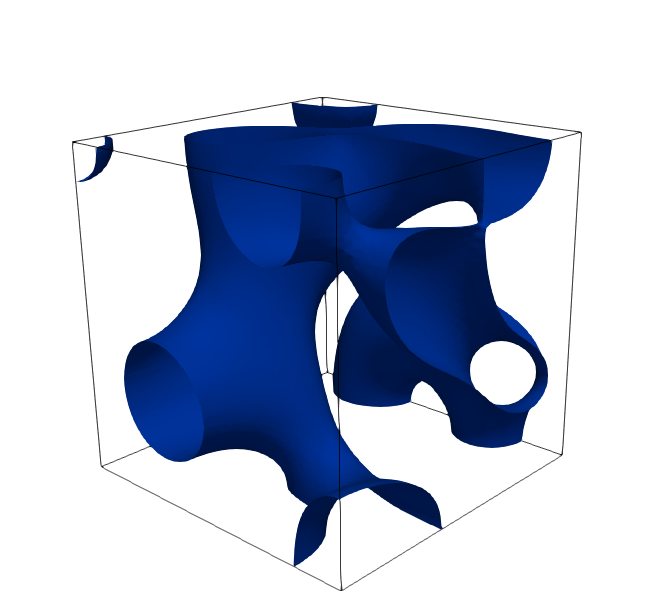}
\includegraphics[angle=-0,width=0.2\textwidth]{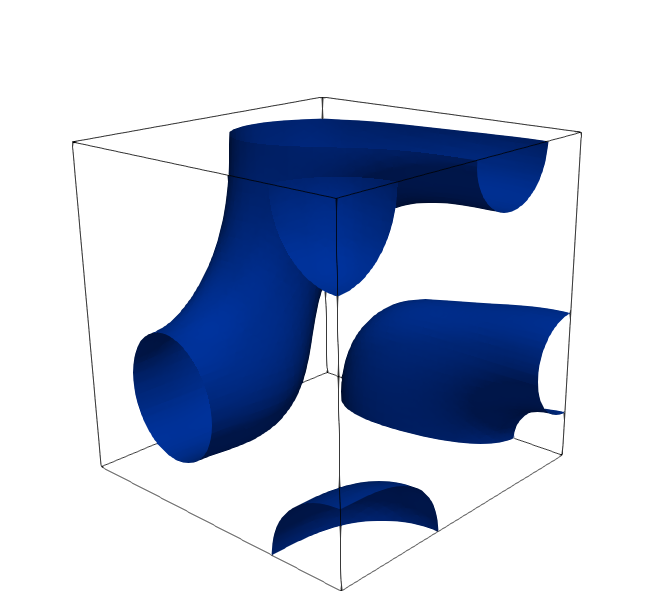}
\includegraphics[angle=-0,width=0.2\textwidth]{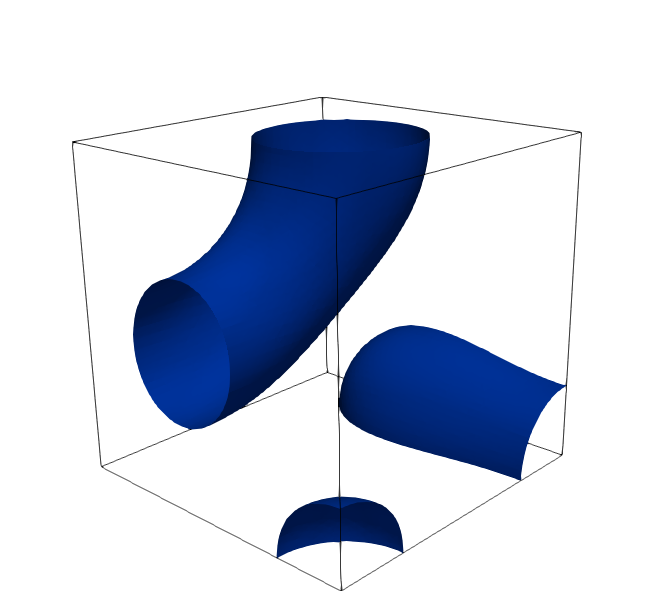}
\includegraphics[angle=-0,width=0.2\textwidth]{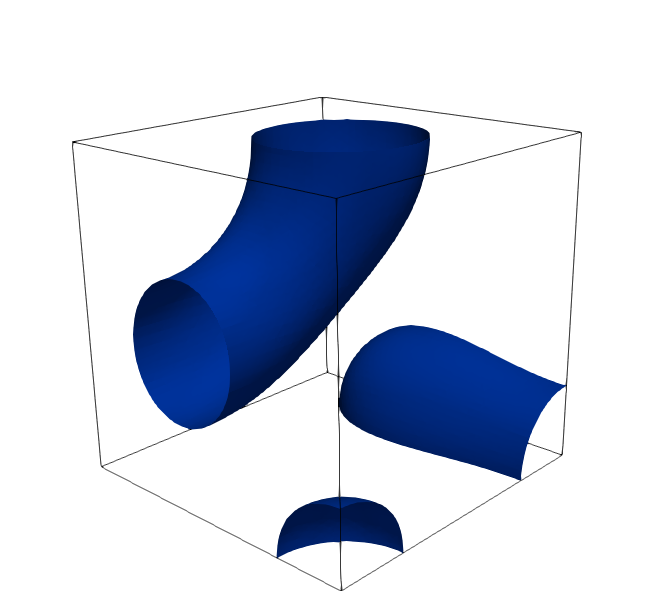}}
\caption{($\eps=\frac1{8\pi}$, $\Omega=(0,1)^3$) 
Evolution for $\beta=0.02$, $S_-=1$, $S_+ = -4$. 
We show the solution at times $t=0.01,0.05,0.1,0.5,5$.}
\label{fig:3d_8pisd_beta002_1-4}
% ~/hpc_cluster/data/alberta/ool/3d.8pi.squaresd.beta2e-2_1-4
\end{figure}%
\begin{figure}
% paraview --data=uw_h..vtk
%          choose white, blue, black colour palette, save animation
% cp 3d8pisd001_1-4*_0[0-5][015][02].png ~/tex/glns/ool/figures && scpp 3d8pisd001_1-4*_0[0-5][015][02].png e23:tex/glns/ool/figures
\center
\mbox{
\includegraphics[angle=-0,width=0.2\textwidth]{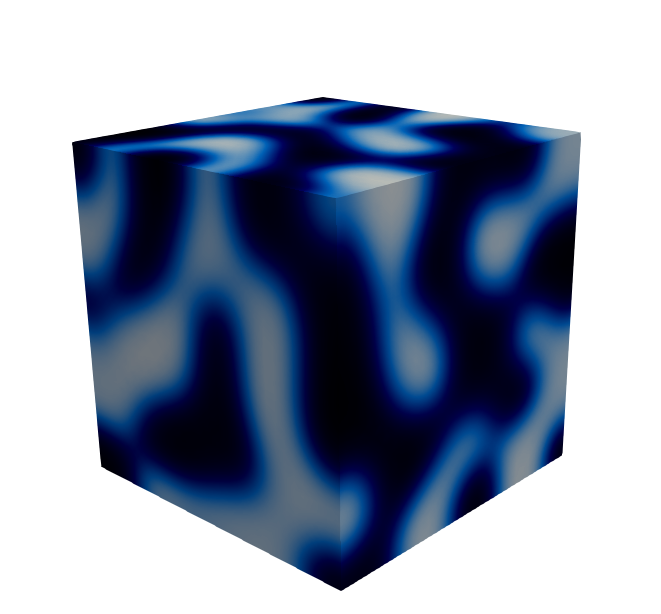}
\includegraphics[angle=-0,width=0.2\textwidth]{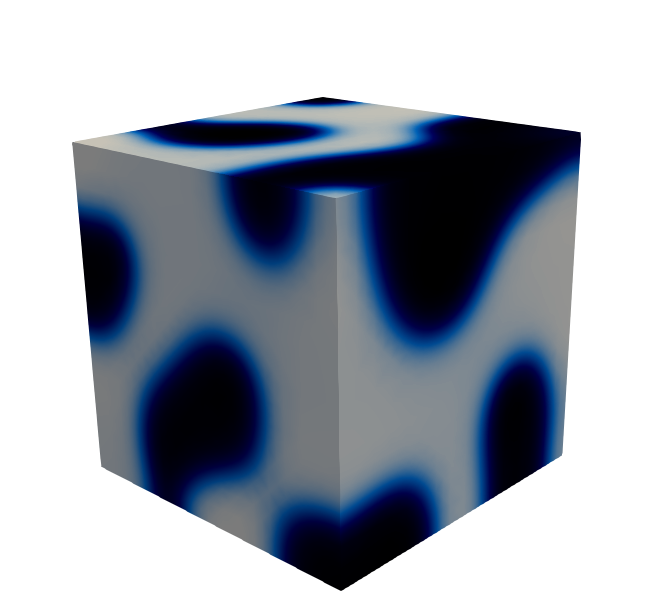}
\includegraphics[angle=-0,width=0.2\textwidth]{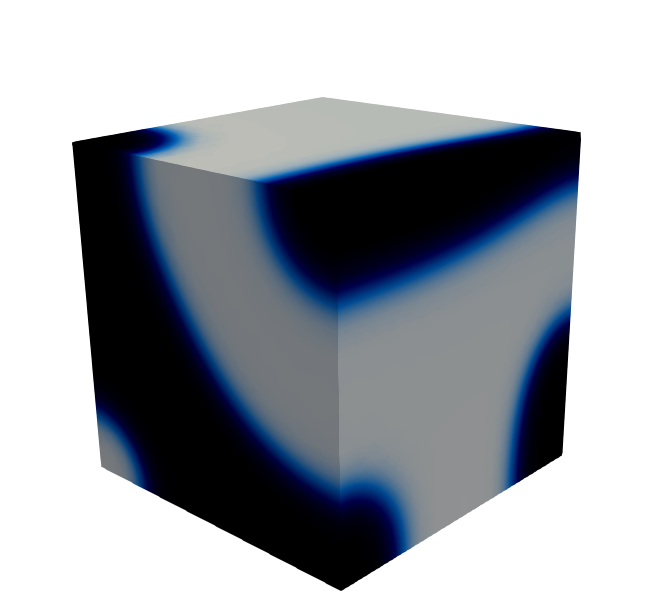}
\includegraphics[angle=-0,width=0.2\textwidth]{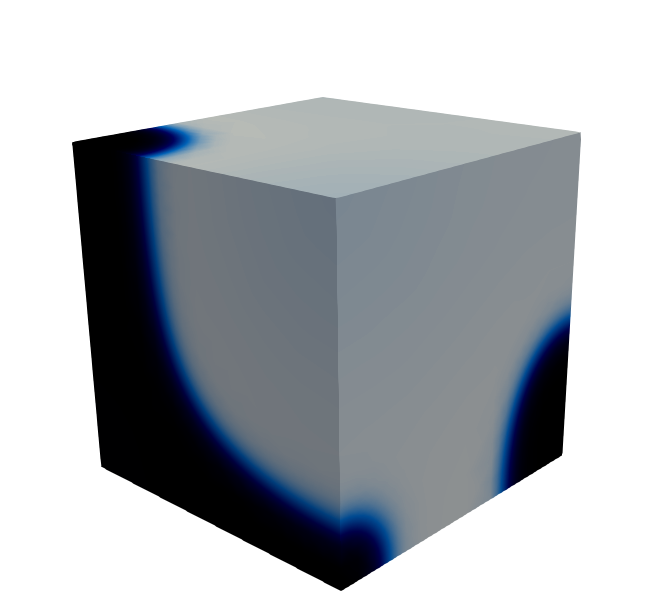}
\includegraphics[angle=-0,width=0.2\textwidth]{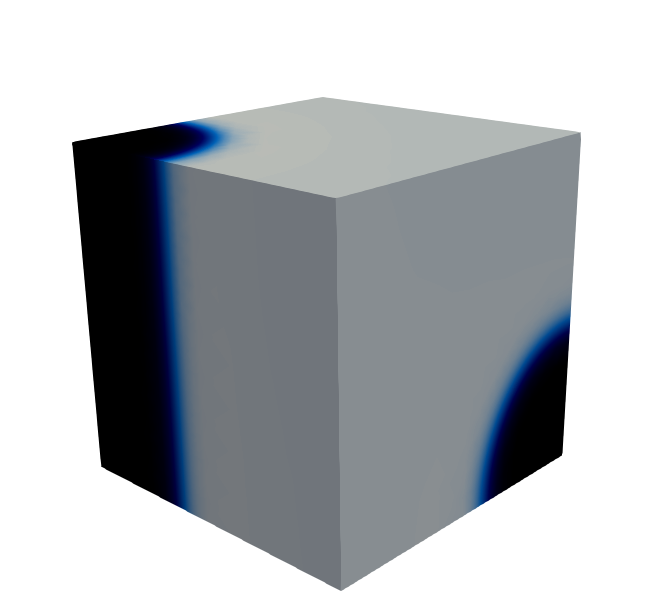}}
\mbox{
\includegraphics[angle=-0,width=0.2\textwidth]{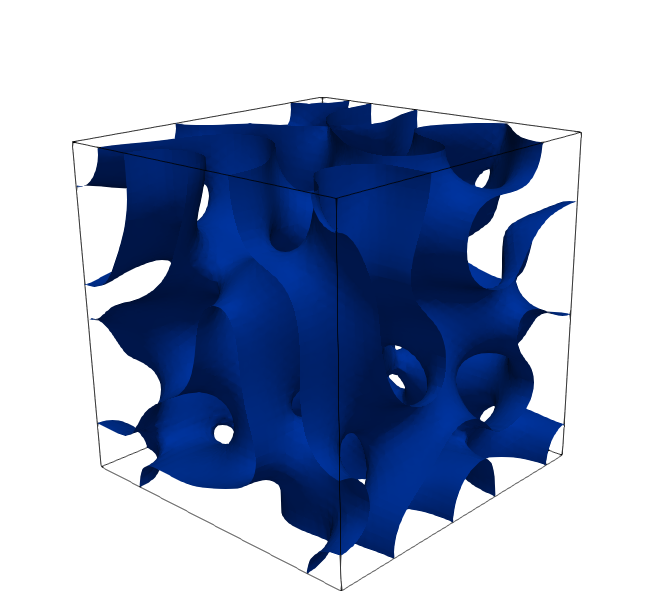}
\includegraphics[angle=-0,width=0.2\textwidth]{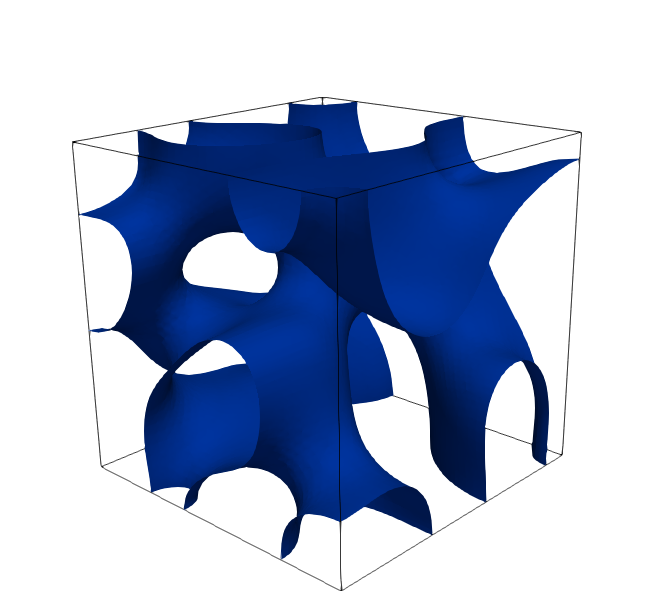}
\includegraphics[angle=-0,width=0.2\textwidth]{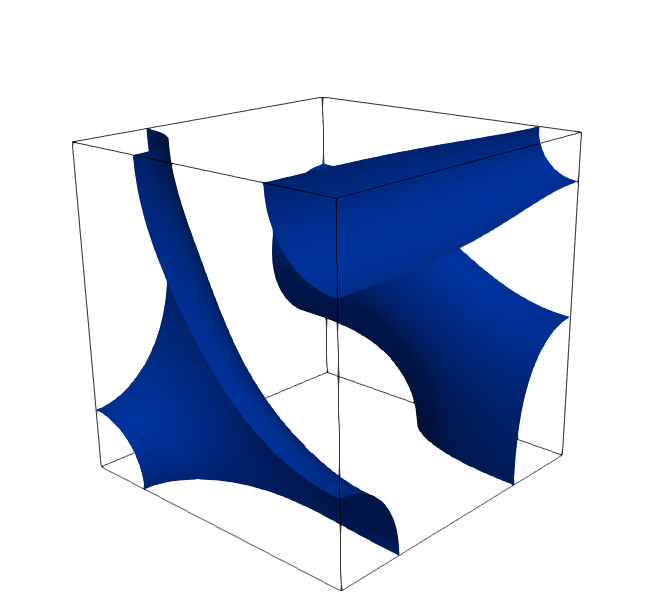}
\includegraphics[angle=-0,width=0.2\textwidth]{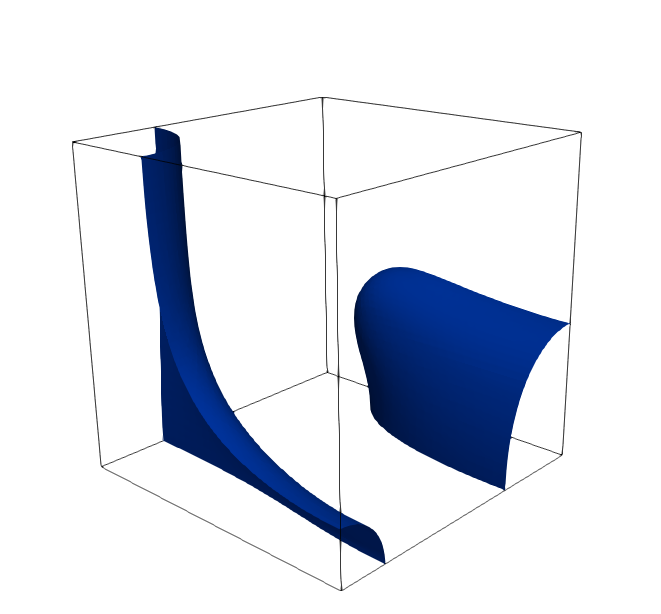}
\includegraphics[angle=-0,width=0.2\textwidth]{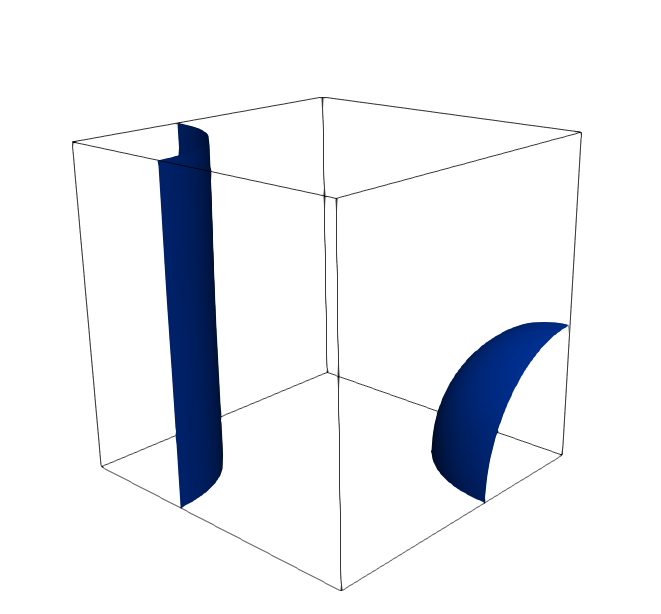}}
\caption{($\eps=\frac1{8\pi}$, $\Omega=(0,1)^3$) 
Evolution for $\beta=0.1$, $S_-=1$, $S_+ = -4$. 
We show the solution at times $t=0.02,0.1,0.5,1,5$.}
\label{fig:3d_8pisd_beta001_1-4}
% ~/hpc_cluster/data/alberta/ool/3d.8pi.squaresd.beta1e-1_1-4
\end{figure}%

For a three-dimensional analogue of Figure~\ref{fig:3balls2d}, we
use the parameters $\beta = 0.1$, $S_- = 0.8$, $S_+= -10$, $m_-=0.2$,
$m_+=0.5$, $\rho_\pm = 0.6$, on the cube $\Omega=(0,4)^3$. 
We start from three spherical initial blobs with radii 
$0.34$, $0.35$ and $0.36$. As can be seen from the
evolution shown in Figure~\ref{fig:3balls3d}, the three blobs hardly
change their size, and soon settle on an
arrangement that in our numerical computations is a steady state.
\mod{As further coarsening is eventually suppresses these two and three dimensional computations show that the active Cahn--Hilliard model can suppress Ostwald ripening. This is in agreement with  \cite{zwickerostwald} which studied the
suppression of Ostwald ripening in so-called active emulsions. In particular, multiple  droplets can be stable. Due to the Neumann boundary conditions,  we can extend the solution obtained by reflections in space  and we then observe also the occurrence of cylinders and toroidal shapes.}
\begin{figure}
% ../plotool; mv energy.png aug16pi3balls_002_025-4_e.png
% paraview --data=uw_h..vtk
%          choose white, blue, black colour palette, save animation
% cp 3djan8pi3balls_01_08-10b_00[015]0.png ~/tex/glns/ool/figures && scpp 3djan8pi3balls_01_08-10b_00[015]0.png e23:tex/glns/ool/figures
\center
\newcommand\localwidth{0.25\textwidth}
\mbox{
\includegraphics[angle=-0,width=\localwidth]{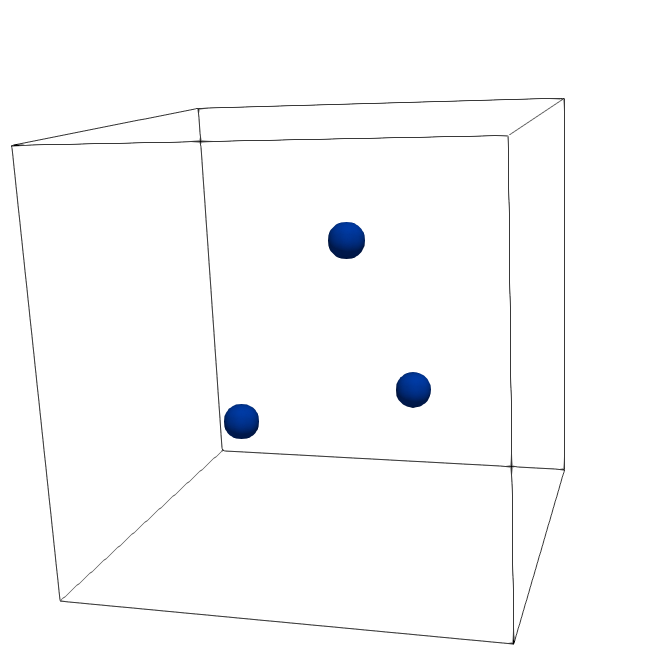}
\includegraphics[angle=-0,width=\localwidth]{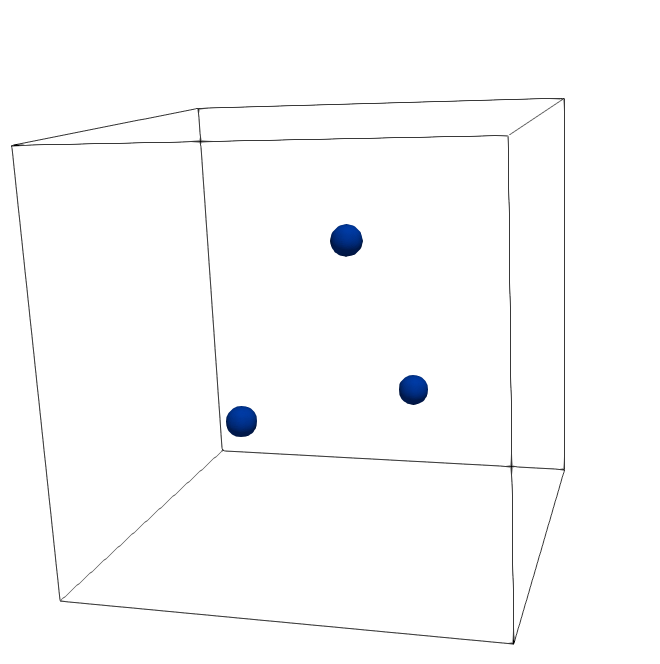}
\includegraphics[angle=-0,width=\localwidth]{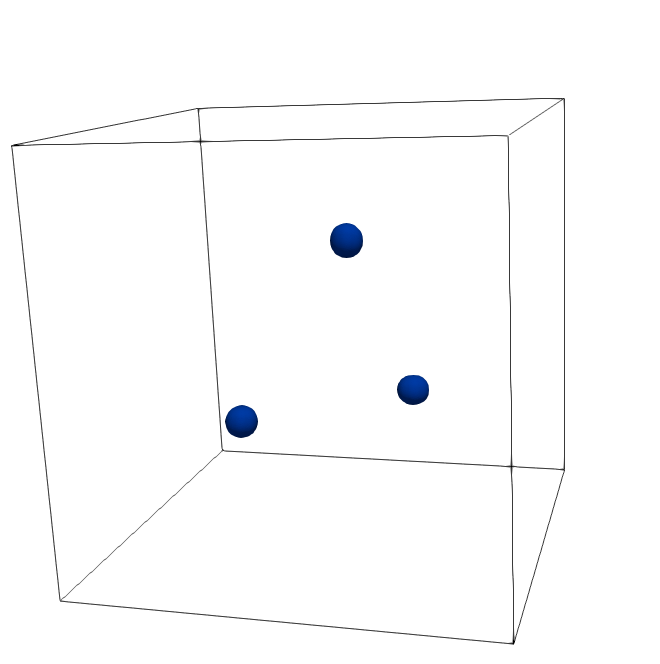}
}
\caption{($\eps=\frac1{8\pi}$, $\Omega=(0,4)^3$) 
Evolution for $\beta=0.1$, $S_- = 0.8$, $S_+= -10$, $m_-=0.2$,
$m_+=0.5$, $\rho_\pm = 0.6$. 
We show the solution at times $t=0,1,5$.
{\anold{In our numerical computations the solution at the final time appears to be a steady state.}}
} 
\label{fig:3balls3d}
% ~/hpc_cluster/data/alberta/ool/3d.8pi.rc05.square8_3balls.beta1e-1_08-10_0606_m0205
\end{figure}%

\subsection{Numerical computations: Active droplets in 2d}

Our first simulation for active droplets shows the possible creation of shells
in 2d. We let $\Omega=(0,2)^2$ and choose $\beta=0.002$, $S_-=0.25$, $S_+ = -4$
for the physical parameters. The initial droplet is a disk of radius 
$r_0=0.4$. For the phase field parameter we choose
$\eps=\frac1{32\pi}$. %, and set $N_f = 8N_c = 64$ with $\tau=10^{-3}$.
See Figure~\ref{fig:square1_32piR2_beta0002}, where we observe
the development of a stable shell.
\begin{figure}
% ../plotool; mv energy.png aug32pisquare1b0002_025-4_e.png
% paraview --data=uw_h..vtk
%          choose white, blue, black colour palette, save animation
% cp aug32pisquare1b0002_025-4_e.png aug32pisquare1b0002_025-4_00[0125][02].png ~/tex/glns/ool/figures && scpp aug32pisquare1b0002_025-4_e.png aug32pisquare1b0002_025-4_00[0125][02].png e23:tex/glns/ool/figures
\center
\mbox{
\includegraphics[angle=-0,width=0.2\textwidth]{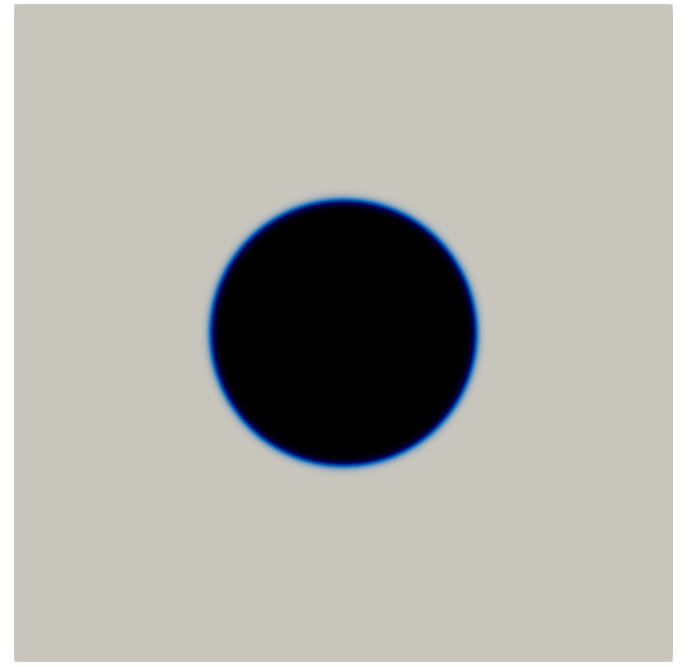}
\includegraphics[angle=-0,width=0.2\textwidth]{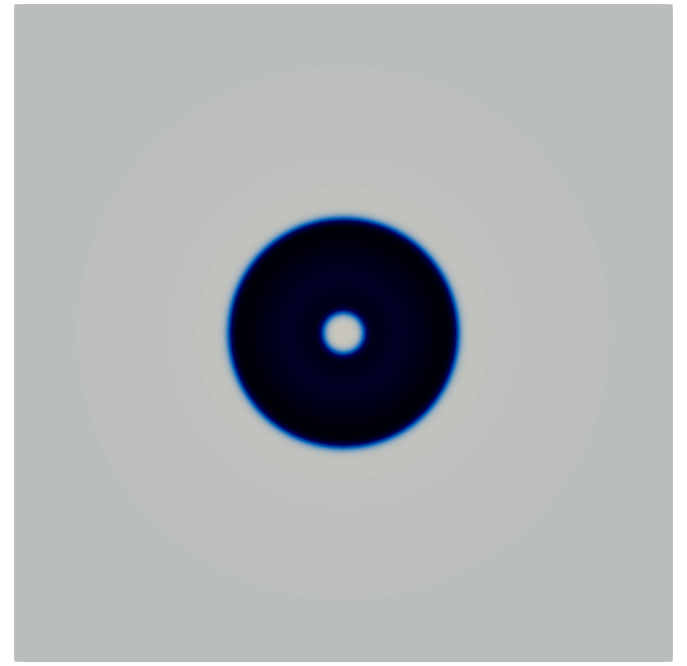}
\includegraphics[angle=-0,width=0.2\textwidth]{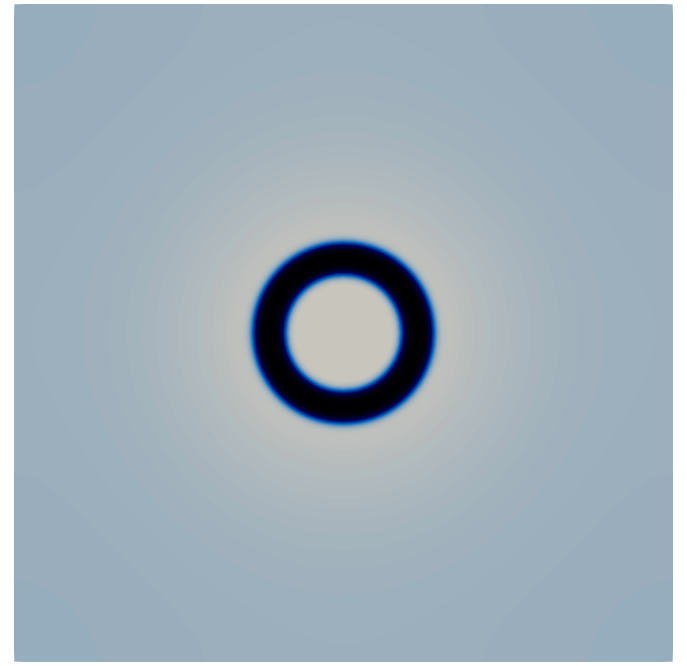}
\includegraphics[angle=-0,width=0.2\textwidth]{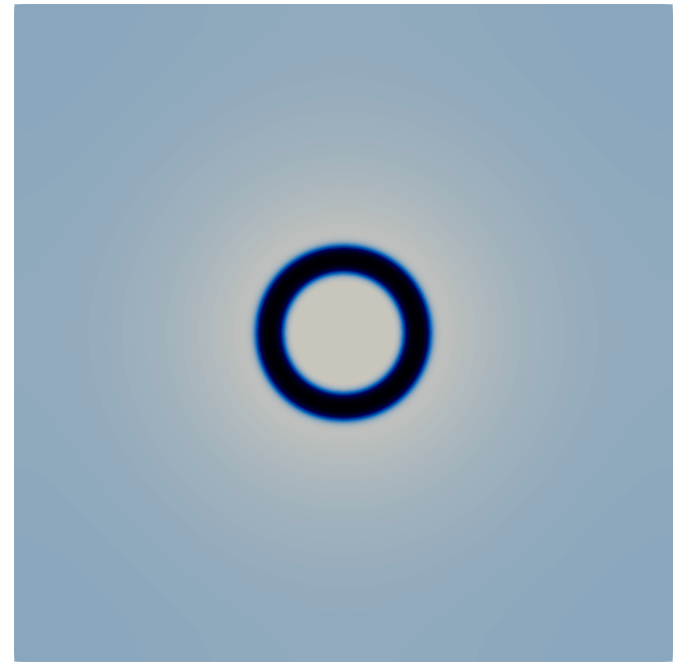}
\includegraphics[angle=-0,width=0.2\textwidth]{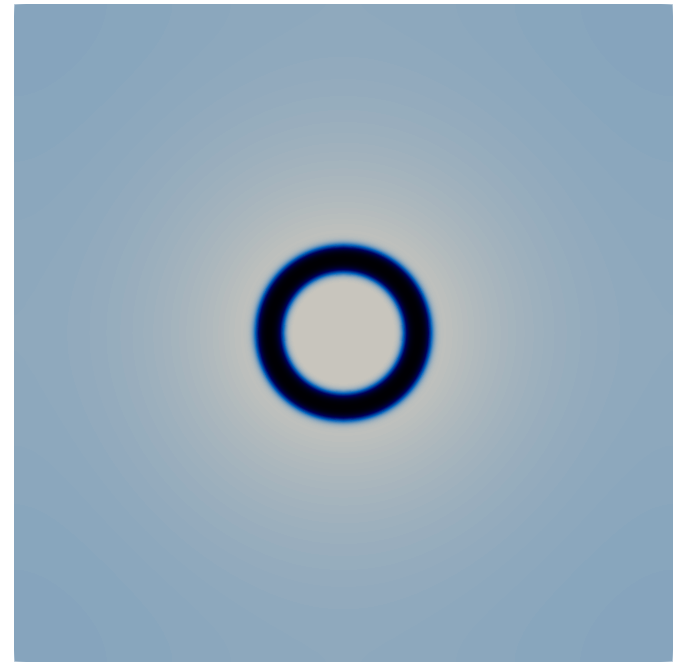}}
\caption{($\eps=\frac1{32\pi}$, $\Omega=(0,2)^2$) 
Evolution for $\beta=0.002$, $S_-=0.25$, $S_+ = -4$. 
We show the solution at times $t=0,0.2,1,2,5$.
%Below a plot of the discrete energy.
}
\label{fig:square1_32piR2_beta0002}
% ~/hpc_cluster/data/alberta/ool/2d.aug32pi.square1.beta2e-3_025-4
\end{figure}%
However, if we reduce the d\anold{i}mension of the domain to the unit square,
$\Omega=(0,1)^2$, then the shell shape is only transient. In fact, the
temporary shell evolves to a smaller, stable disk. See
Figure~\ref{fig:square_32piR2_beta0002} for the results.
%Here we adjusted the discretization parameters to $N_f = 8N_c = 32$.
\begin{figure}
% ../plotool; mv energy.png aug32pisquareb0002_025-4_e.png
% paraview --data=uw_h..vtk
%          choose white, blue, black colour palette, save animation
% cp aug32pisquareb0002_025-4_e.png aug32pisquareb0002_025-4_00[0125][02].png ~/tex/glns/ool/figures && scpp aug32pisquareb0002_025-4_e.png aug32pisquareb0002_025-4_00[0125][02].png e23:tex/glns/ool/figures
\center
\mbox{
\includegraphics[angle=-0,width=0.2\textwidth]{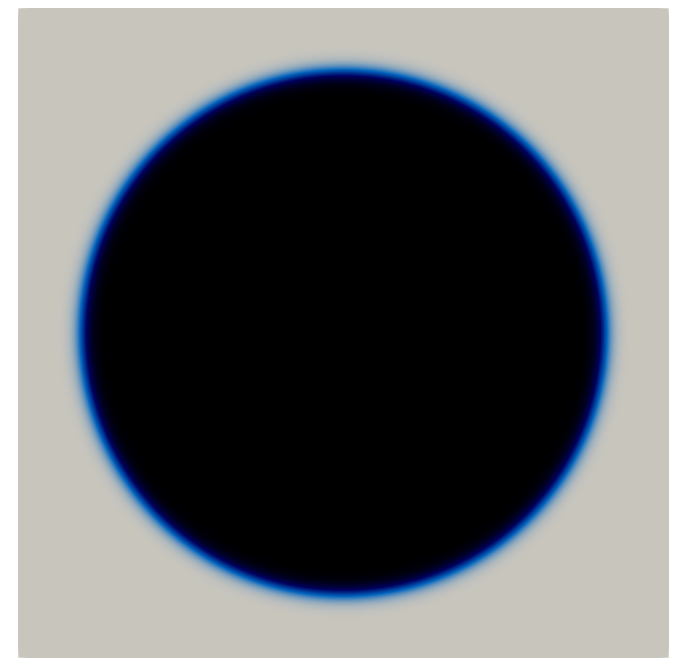}
\includegraphics[angle=-0,width=0.2\textwidth]{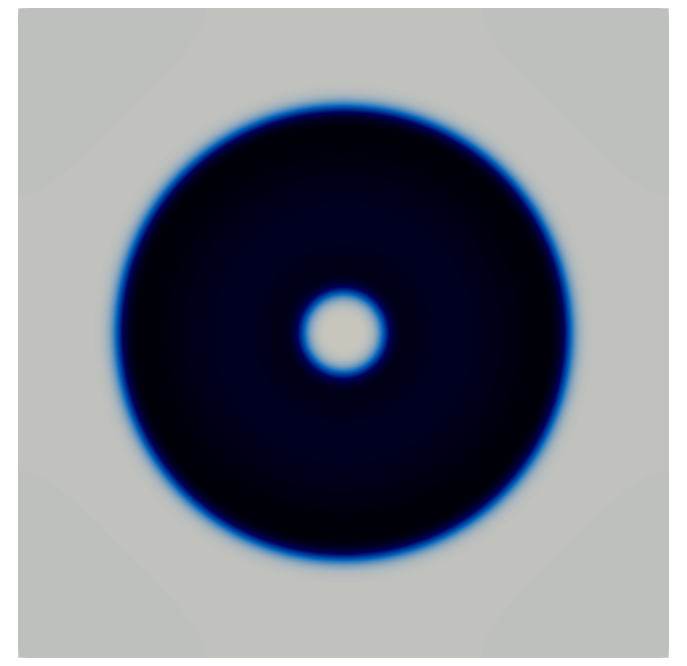}
\includegraphics[angle=-0,width=0.2\textwidth]{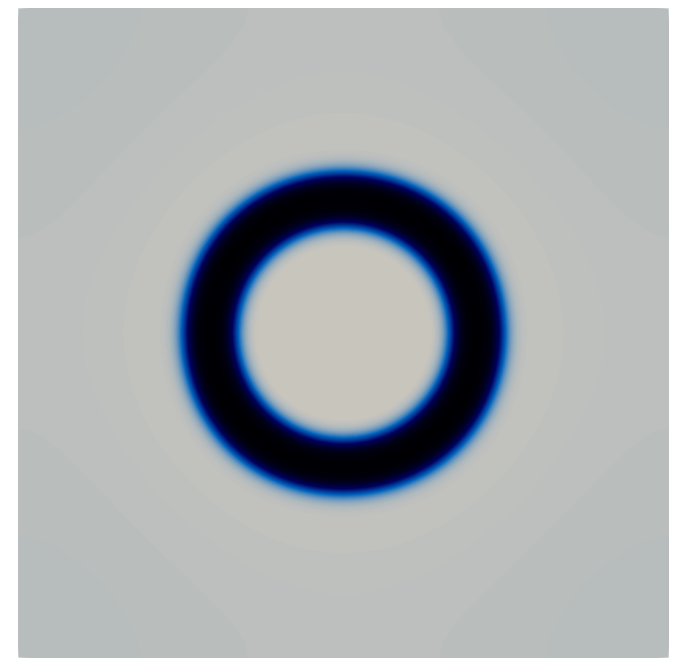}
\includegraphics[angle=-0,width=0.2\textwidth]{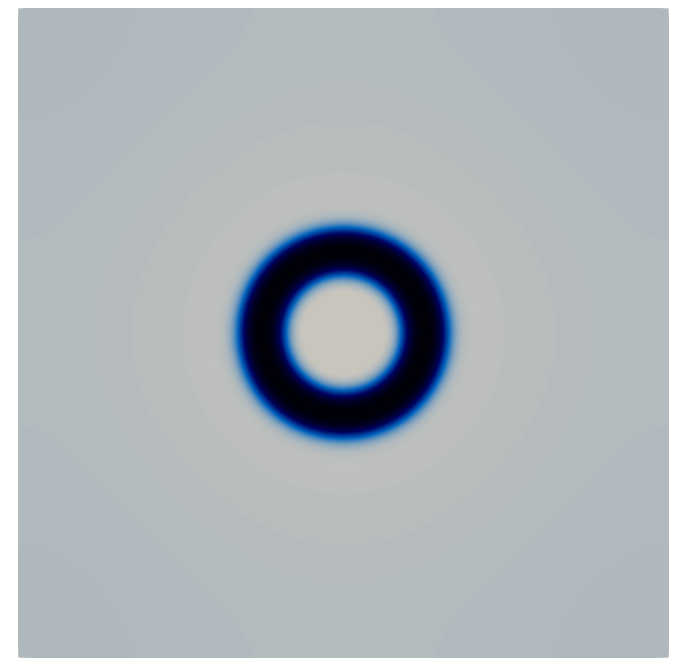}
\includegraphics[angle=-0,width=0.2\textwidth]{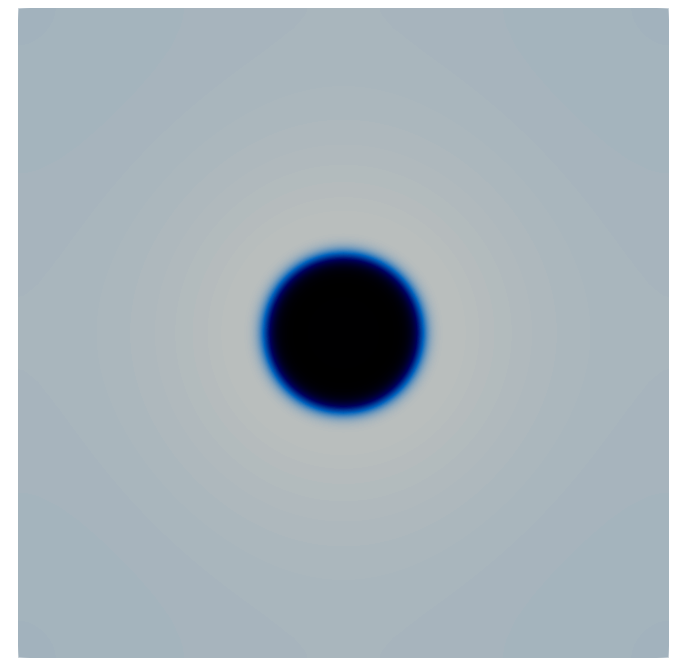}}
\caption{($\eps=\frac1{32\pi}$, $\Omega=(0,1)^2$) 
Evolution for $\beta=0.002$, $S_-=0.25$, $S_+ = -4$. 
We show the solution at times $t=0,0.2,1,2,5$.
%Below a plot of the discrete energy.
}
\label{fig:square_32piR2_beta0002}
% ~/hpc_cluster/data/alberta/ool/2d.aug32pi.square.beta2e-3_025-4
\end{figure}%

The next simulation shows pinch-off for a slightly perturbed initially circular
droplet. In fact, for the ``radius'' of the initial droplet we choose
\begin{equation} \label{eq:r0perturb}
r_0(\theta) = 0.25 + 0.02 \cos(2\theta - \tfrac{2\pi}9), % 20^o
\end{equation}
with $\theta \in [-\pi,\pi]$ denoting the angle in polar coordinates. Hence the
initial blob has its widest dimension along an axis that is tilted by
$20^\circ$ compared to the $x$-axis, with the smallest dimension along an axis
perpendicular to it. Letting this initial droplet evolve in the domain
$\Omega = (0,2)^2$ leads to a thinning of the droplet at the centre of the
domain, and eventually to a pinch-off into two separate droplets.
See Figure~\ref{fig:square1_16pi_r025_beta002} for our numerical results for
$\eps=\frac1{16\pi}$.
\begin{figure}
% ../plotool; mv energy.png aug16pisquare1r025_002_1-4_e.png
% paraview --data=uw_h..vtk
%          choose white, blue, black colour palette, save animation
% cp aug16pisquare1r025_002_1-4_e.png aug16pisquare1r025_002_1-4_0[0,1][0,2,4,6,8]0.png ~/tex/glns/ool/figures && scpp aug16pisquare1r025_002_1-4_e.png aug16pisquare1r025_002_1-4_0[0,1][0,2,4,6,8]0.png e23:tex/glns/ool/figures
\center
\mbox{
\includegraphics[angle=-0,width=0.15\textwidth]{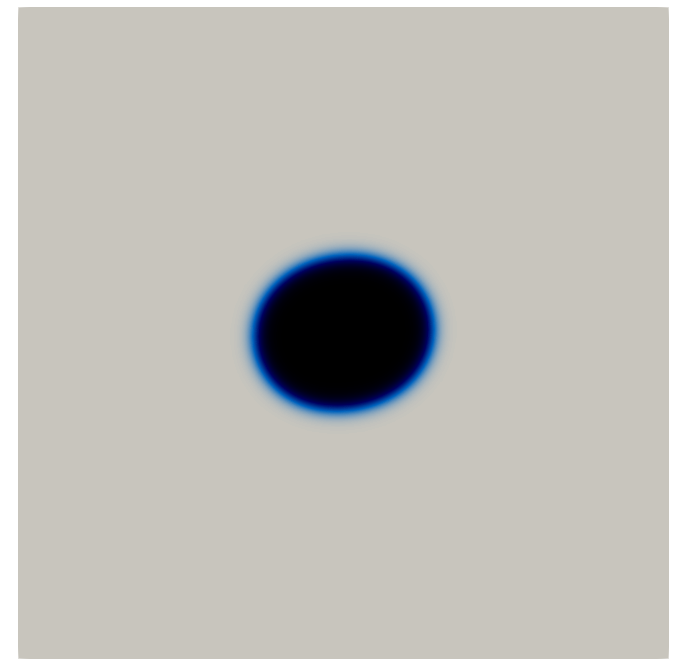}
\includegraphics[angle=-0,width=0.15\textwidth]{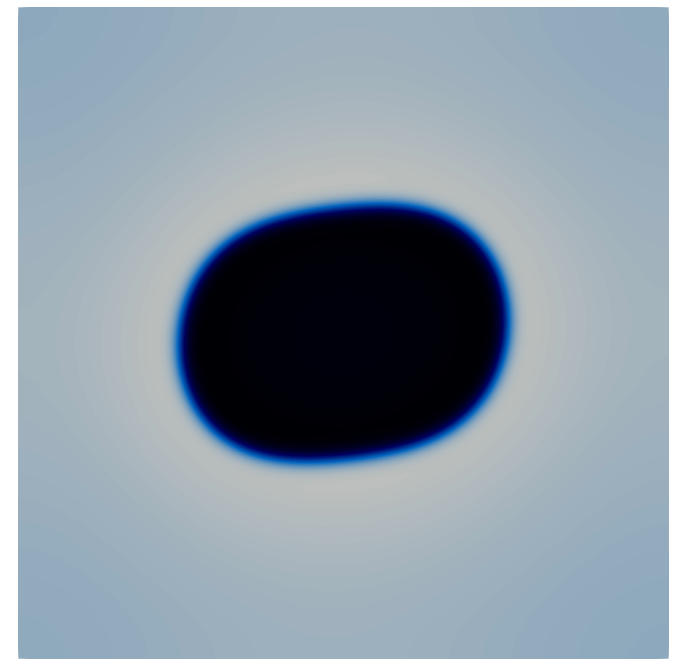}
\includegraphics[angle=-0,width=0.15\textwidth]{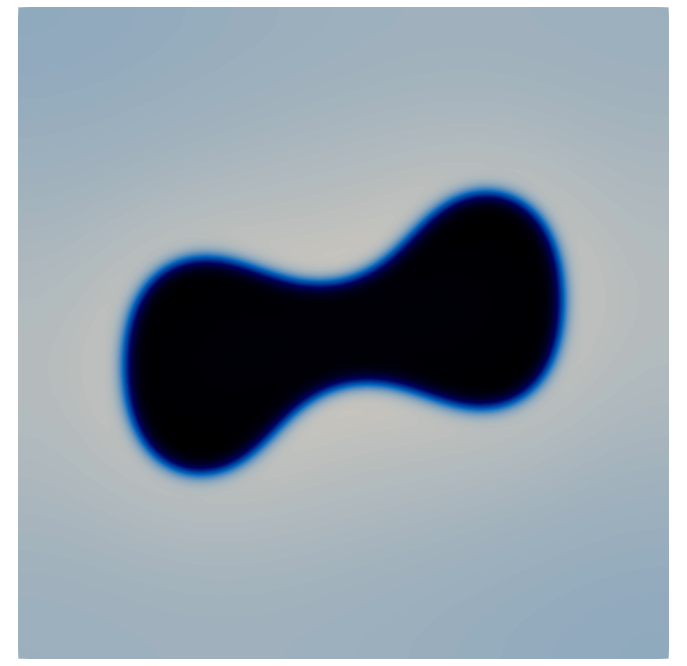}
\includegraphics[angle=-0,width=0.15\textwidth]{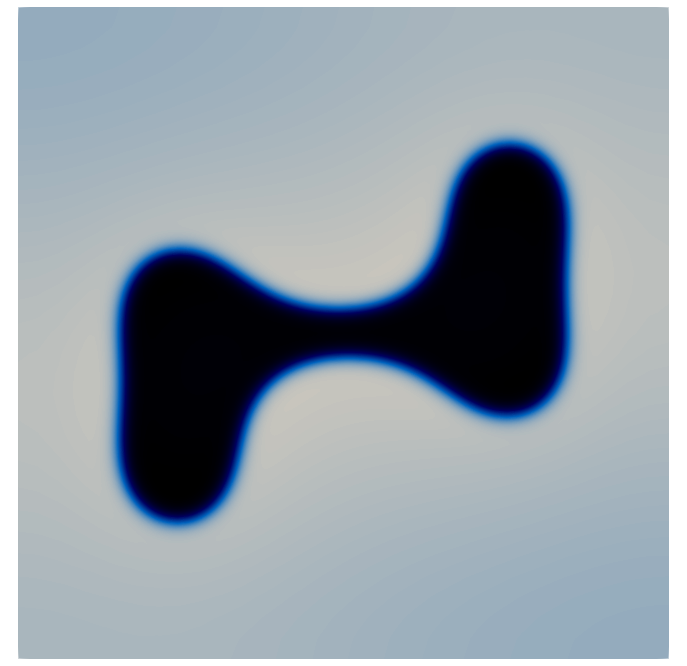}
\includegraphics[angle=-0,width=0.15\textwidth]{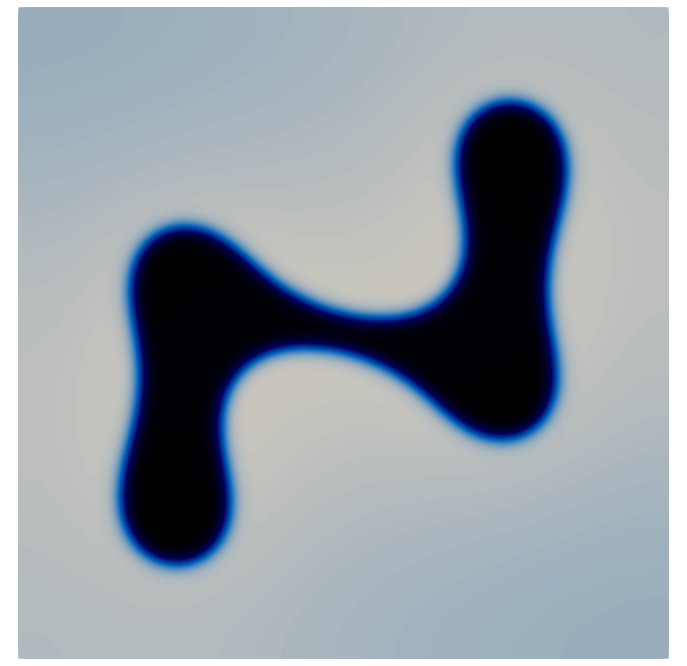}
\includegraphics[angle=-0,width=0.15\textwidth]{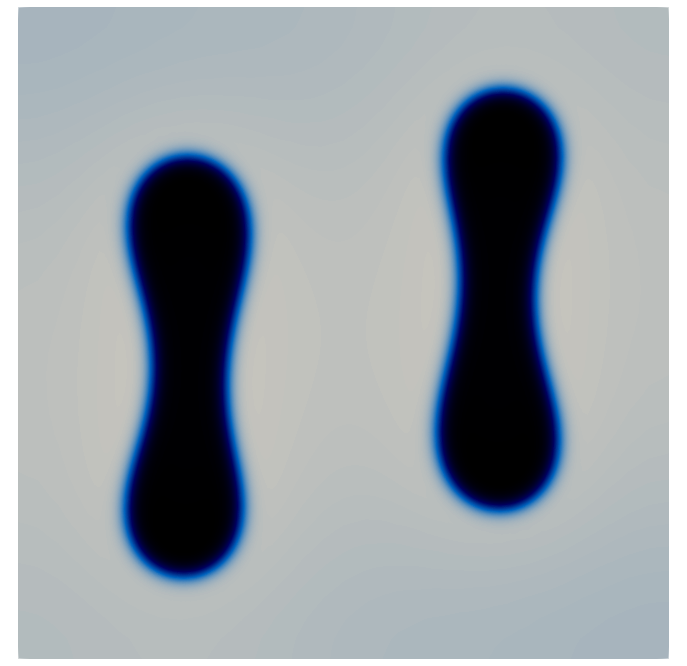}
}
\caption{($\eps=\frac1{16\pi}$, $\Omega=(0,2)^2$) 
Evolution for $\beta=0.02$, $S_-=1$, $S_+ = -4$. 
We show the solution at times $t=0,2,4,6,8,10$. 
%Below a plot of the discrete energy.
}
\label{fig:square1_16pi_r025_beta002}
% ~/hpc_cluster/data/alberta/ool/2d.aug16pi.square1.r025beta2e-2_1-4
\end{figure}%
%Interestingly, when we repeat the same simulation on the larger domain
%$\Omega = (0,4)^2$, we no longer observe pinch-off and instead see some
%nucleation at the boundary, see Figure~\ref{fig:square2_16pi_r025_beta002}.
%\begin{figure}
%% ../plotool; mv energy.png aug16pisquare2r025_002_1-4_e.png
%% paraview --data=uw_h..vtk
%%          choose white, blue, black colour palette, save animation
%% cp aug16pisquare2r025_002_1-4_e.png aug16pisquare2r025_002_1-4_0[0,1][0,2,4,6,8]0.png ~/tex/glns/ool/figures && scpp aug16pisquare2r025_002_1-4_e.png aug16pisquare2r025_002_1-4_0[0,1][0,2,4,6,8]0.png e23:tex/glns/ool/figures
%\center
%\mbox{
%\includegraphics[angle=-0,width=0.15\textwidth]{figures/aug16pisquare2r025_002_1-4_0000}
%\includegraphics[angle=-0,width=0.15\textwidth]{figures/aug16pisquare2r025_002_1-4_0020}
%\includegraphics[angle=-0,width=0.15\textwidth]{figures/aug16pisquare2r025_002_1-4_0040}
%\includegraphics[angle=-0,width=0.15\textwidth]{figures/aug16pisquare2r025_002_1-4_0060}
%\includegraphics[angle=-0,width=0.15\textwidth]{figures/aug16pisquare2r025_002_1-4_0080}
%\includegraphics[angle=-0,width=0.15\textwidth]{figures/aug16pisquare2r025_002_1-4_0100}
%}
%\caption{($\eps=\frac1{16\pi}$, $\Omega=(0,4)^2$) 
%Evolution for $\beta=0.02$, $S_-=1$, $S_+ = -4$. 
%We show the solution at times $t=0,2,4,6,8,10$. 
%%Below a plot of the discrete energy.
%}
%\label{fig:square2_16pi_r025_beta002}
%% ~/hpc_cluster/data/alberta/ool/2d.aug16pi.square2.r025beta2e-2_1-4
%\end{figure}%
Running the simulation on the %even 
larger domain $\Omega = (0,8)^2$
once again shows very delicate fingering patterns appear, but no pinch-off
occurs. See Figure~\ref{fig:square4_16pi_r025_beta002}.
\begin{figure}
% ../plotool; mv energy.png aug16pisquare4r025_002_1-4_e.png
% paraview --data=uw_h..vtk
%          choose white, blue, black colour palette, save animation
% cp aug16pisquare4r025_002_1-4_e.png aug16pisquare4r025_002_1-4_0[0,1][0,2,4,6,8]0.png ~/tex/glns/ool/figures && scpp aug16pisquare4r025_002_1-4_e.png aug16pisquare4r025_002_1-4_0[0,1][0,2,4,6,8]0.png e23:tex/glns/ool/figures
\center
\mbox{
\includegraphics[angle=-0,width=0.15\textwidth]{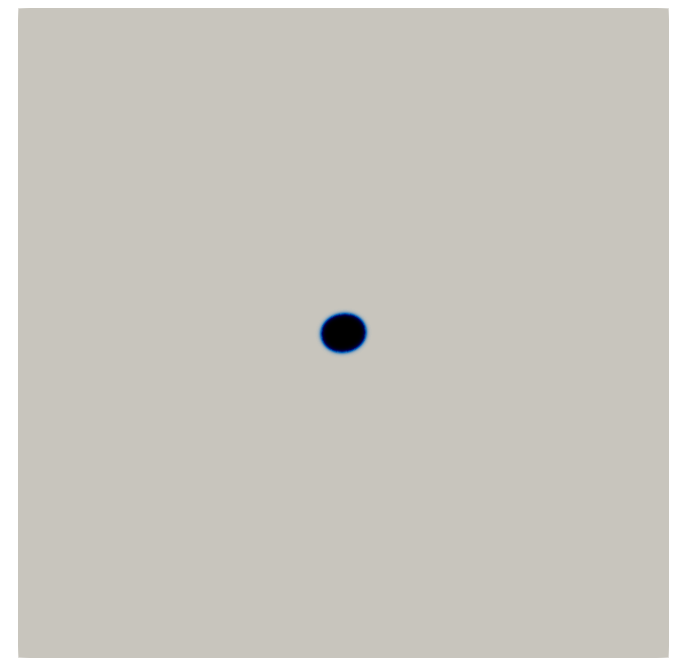}
\includegraphics[angle=-0,width=0.15\textwidth]{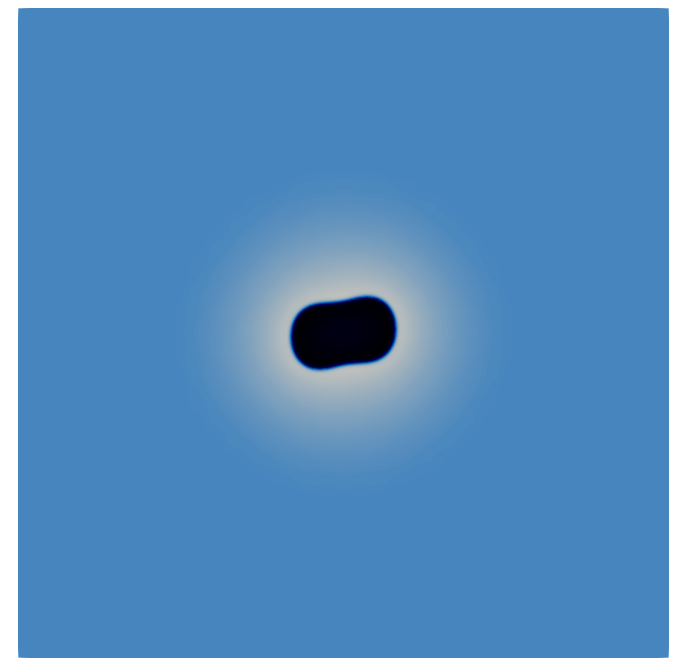}
\includegraphics[angle=-0,width=0.15\textwidth]{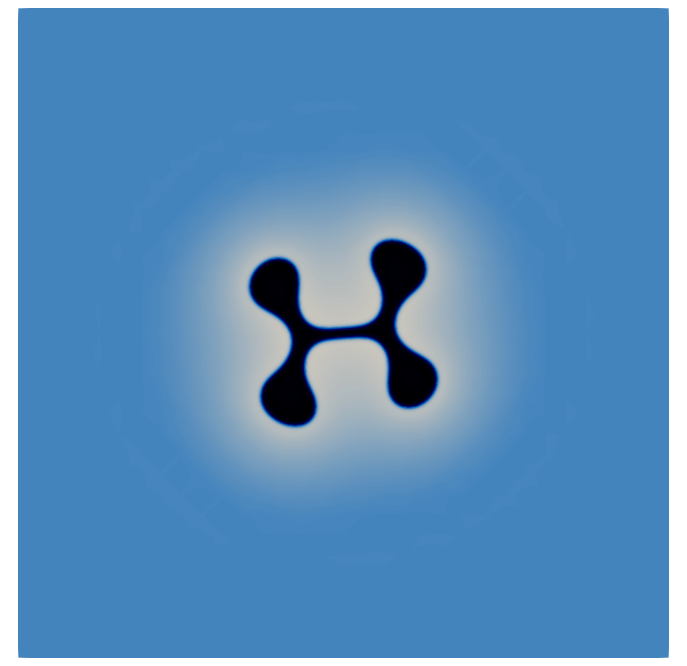}
\includegraphics[angle=-0,width=0.15\textwidth]{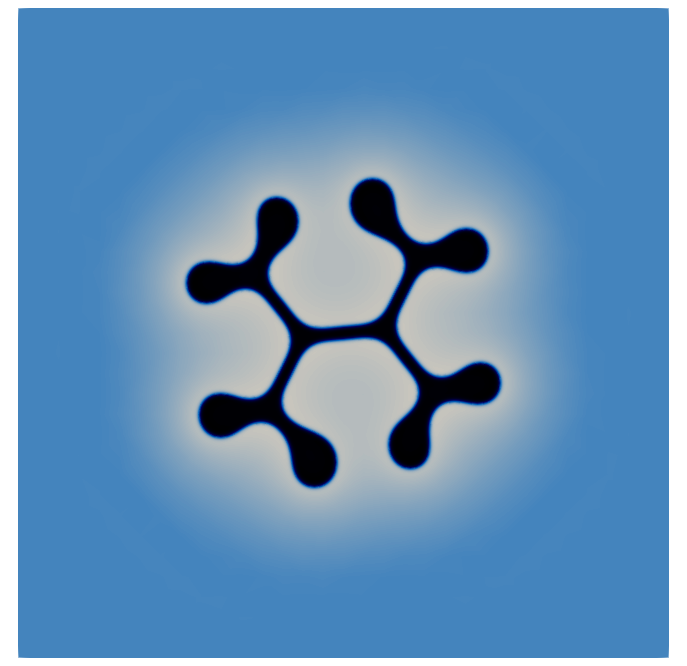}
\includegraphics[angle=-0,width=0.15\textwidth]{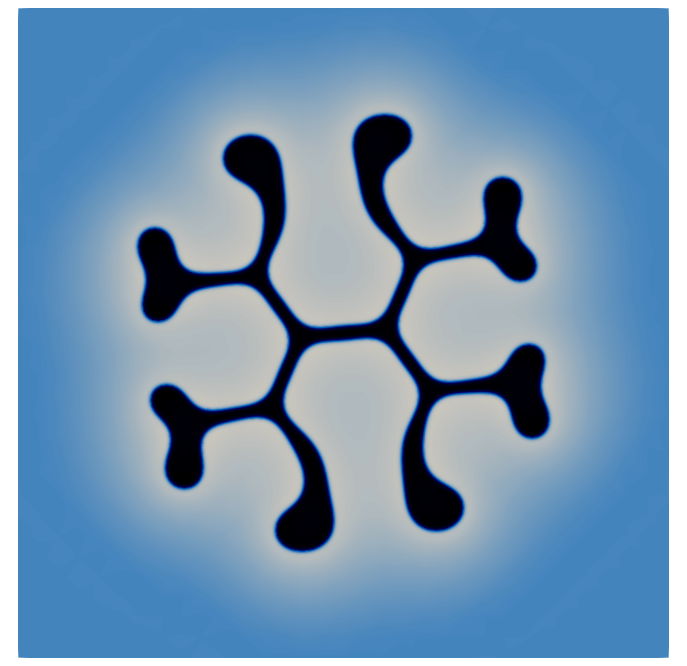}
\includegraphics[angle=-0,width=0.15\textwidth]{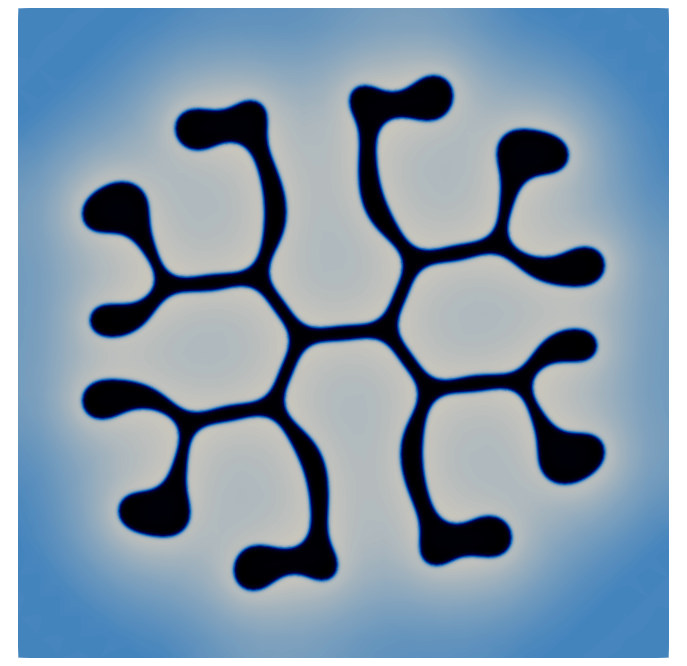}
}
\caption{($\eps=\frac1{16\pi}$, $\Omega=(0,8)^2$) 
Evolution for $\beta=0.02$, $S_-=1$, $S_+ = -4$. 
We show the solution at times $t=0,2,4,6,8,10$. 
%Below a plot of the discrete energy.
}
\label{fig:square4_16pi_r025_beta002}
% ~/hpc_cluster/data/alberta/ool/2d.aug16pi.square4.r025beta2e-2_1-4
\end{figure}%

\last{\subsection{Numerical computations: Active droplets in 3d}
In our first simulations for active droplets in 3d, we attempt
to create analogous shell structures as in 
Figure~\ref{fig:square1_32piR2_beta0002}.
To this end, we let $\Omega=(0,2)^3$ and choose the same physical parameters
$\beta=0.002$, $S_-=0.25$, and $S_+ = -4$. The initial droplet is a ball 
of radius $r_0=0.4$. For the phase field parameter we choose 
$\eps=\frac1{16\pi}$.
See Figure~\ref{fig:cube2_16pi_beta0002_r04}, where we observe
the creation of a shell, which eventually changes into a single ball again.
%\begin{figure}
%% paraview --data=uw_h..vtk
%%          choose white, blue, black colour palette, save animation
%% slice with 0.5 opacity, clip  with 1.0 opacity
%% cp pf16picube2b0002_00[0125][03].png ~/tex/glns/ool/figures && scpp pf16picube2b0002_00[0125][03].png e23:tex/glns/ool/figures
%\center
%\mbox{
%\includegraphics[angle=-0,width=0.2\textwidth]{figures/pf16picube2b0002_0000}
%\includegraphics[angle=-0,width=0.2\textwidth]{figures/pf16picube2b0002_0003}
%\includegraphics[angle=-0,width=0.2\textwidth]{figures/pf16picube2b0002_0010}
%\includegraphics[angle=-0,width=0.2\textwidth]{figures/pf16picube2b0002_0020}
%\includegraphics[angle=-0,width=0.2\textwidth]{figures/pf16picube2b0002_0050}
%}
%\caption{($\eps=\frac1{16\pi}$, $\Omega=(0,2)^3$) 
%Evolution for $\beta=0.002$, $S_-=0.25$, $S_+ = -4$. 
%We show the solution at times $t=0,0.3,1,2,5$.
%}
%\label{fig:cube2_16pi_beta0002_r04}
%% ~/hpc_cluster/data/alberta/ool/3d.16pi.square2.beta2e-3_025-4 
%\end{figure}%
\begin{figure}
% paraview --data=uw_h..vtk
%          choose white, blue, black colour palette, save animation
% slice with 0.5 opacity, clip  with 1.0 opacity
% cp pf16picube2b0002a_00[0125][03].png ~/tex/glns/ool/figures && scpp pf16picube2b0002a_00[0125][03].png e23:tex/glns/ool/figures
\center
\mbox{
\includegraphics[angle=-0,width=0.2\textwidth]{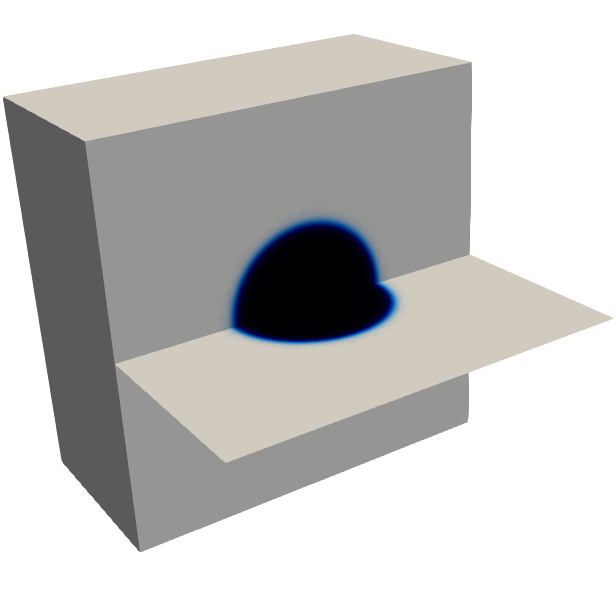}
\includegraphics[angle=-0,width=0.2\textwidth]{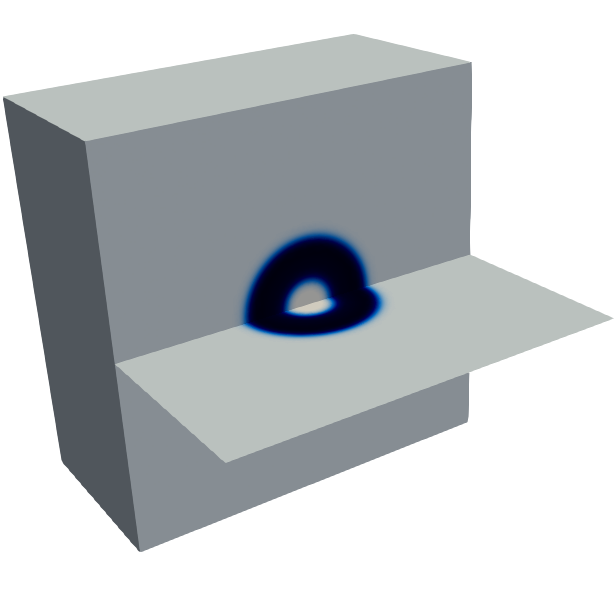}
\includegraphics[angle=-0,width=0.2\textwidth]{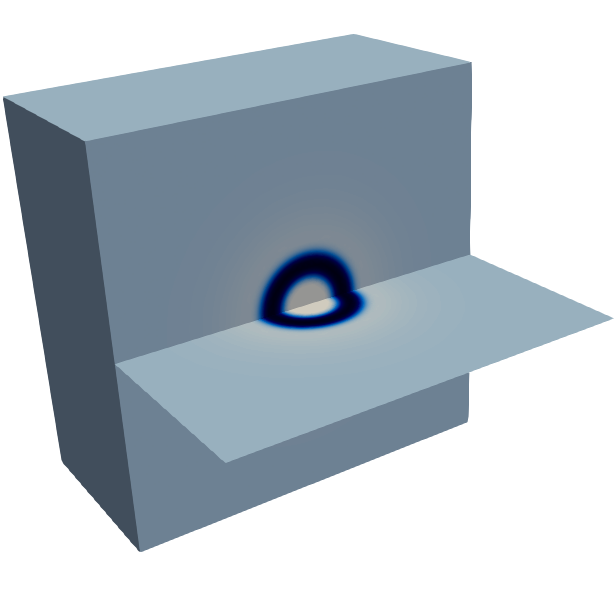}
\includegraphics[angle=-0,width=0.2\textwidth]{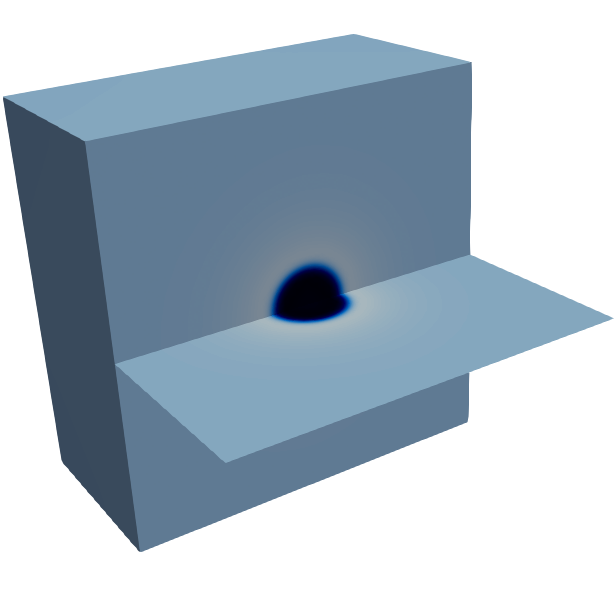}
\includegraphics[angle=-0,width=0.2\textwidth]{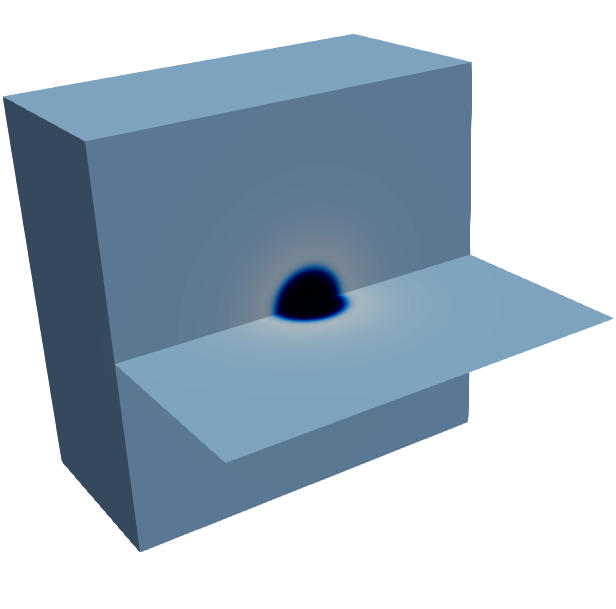}
}
\caption{($\eps=\frac1{16\pi}$, $\Omega=(0,2)^3$) 
Evolution for $\beta=0.002$, $S_-=0.25$, $S_+ = -4$. 
We show the solution at times $t=0,0.3,1,2,5$.
}
\label{fig:cube2_16pi_beta0002_r04}
% ~/hpc_cluster/data/alberta/ool/3d.16pi.square2.beta2e-3_025-4 
\end{figure}%
When we increase the radius of the initial ball to $r_0=0.6$, then the ensuing
evolution is more intricate, see Figure~\ref{fig:cube2_16pi_beta0002_r06}. 
At first two concentric shells appear, with the
inner shell merging into a ball after some time. Then the inner ball
disappears, leaving just a thin outer shell, which starts to become thinner 
and thinner, and which eventually
fragments into several much smaller blobs. These blobs become spherical and
then continue to move away from each other, slowly increasing in size.
\begin{figure}
% paraview --data=uw_h..vtk
%          choose white, blue, black colour palette, save animation
% choose contour with 0.5 opacity, clip and slice with 0.9 opacity
% cp pf16picube2b0002r06_00[0125][0356].png ~/tex/glns/ool/figures && scpp pf16picube2b0002r06_00[0125][0356].png e23:tex/glns/ool/figures
\center
%\includegraphics[angle=-0,width=0.24\textwidth]{figures/pf16picube2b0002r06_0000}
%\includegraphics[angle=-0,width=0.24\textwidth]{figures/pf16picube2b0002r06_0003}
%\includegraphics[angle=-0,width=0.24\textwidth]{figures/pf16picube2b0002r06_0005}
%\includegraphics[angle=-0,width=0.24\textwidth]{figures/pf16picube2b0002r06_0010}
%\caption{($\eps=\frac1{16\pi}$, $\Omega=(0,2)^3$) 
%Evolution for $\beta=0.002$, $S_-=0.25$, $S_+ = -4$. 
%We show the solution at times $t=0,0.3,0.5,1,1.5,1.6,2,5$.
%}
\label{fig:cube2_16pi_beta0002_r06}
% ~/hpc_cluster/data/alberta/ool/3d.16pi.square2r06.beta2e-3_025-4 splits into 4
%\end{figure}%
%\begin{figure}
% paraview --data=uw_h..vtk
%          choose white, blue, black colour palette, save animation
% choose contour with 0.5 opacity, clip and slice with 0.9 opacity
% cp pf16picube2b0002r06a_00[0125][0356].png ~/tex/glns/ool/figures && scpp pf16picube2b0002r06a_00[0125][0356].png e23:tex/glns/ool/figures
%\center
\includegraphics[angle=-0,width=0.24\textwidth]{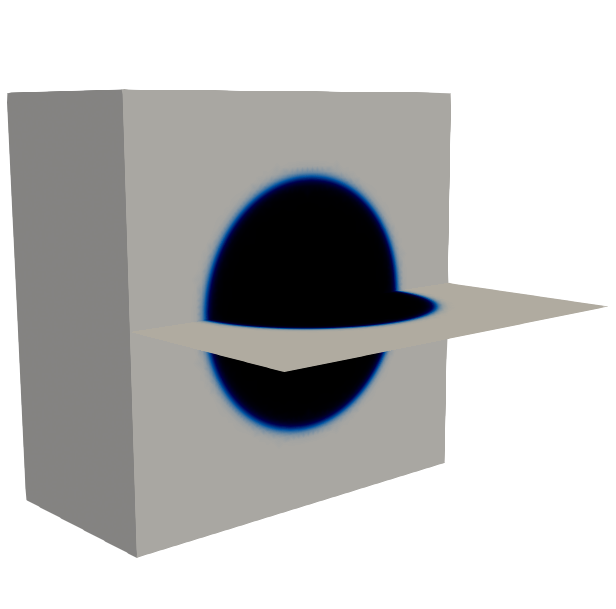}
\includegraphics[angle=-0,width=0.24\textwidth]{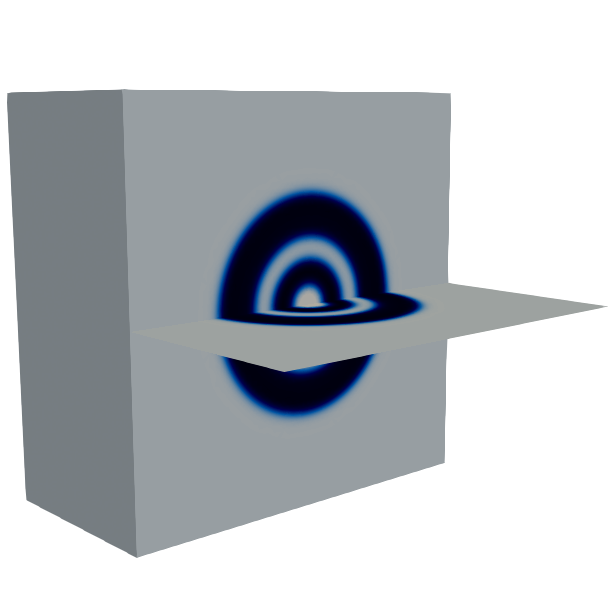}
\includegraphics[angle=-0,width=0.24\textwidth]{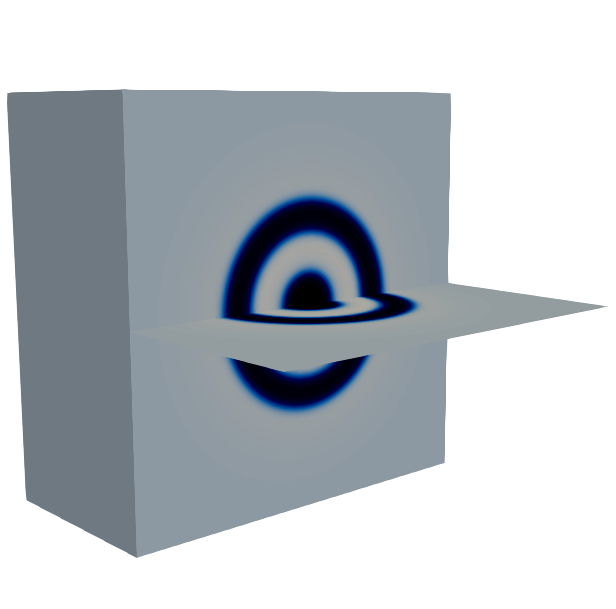}
\includegraphics[angle=-0,width=0.24\textwidth]{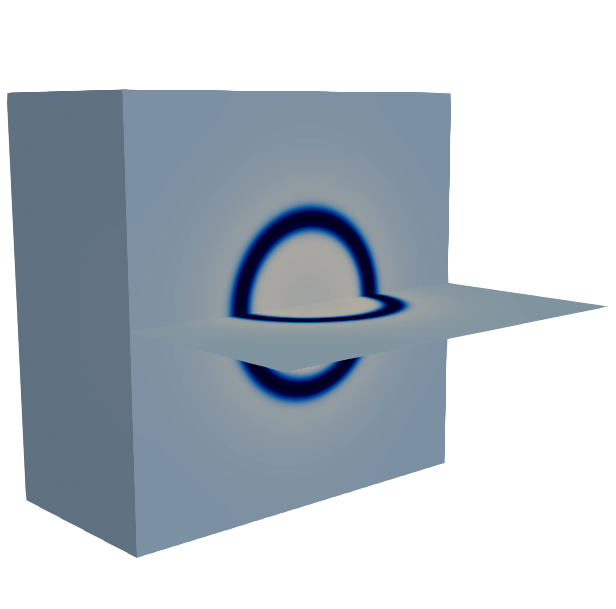}
\includegraphics[angle=-0,width=0.24\textwidth]{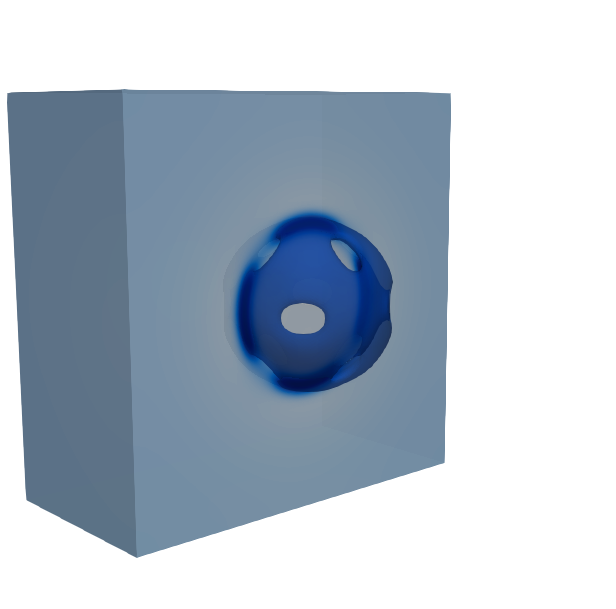}
\includegraphics[angle=-0,width=0.24\textwidth]{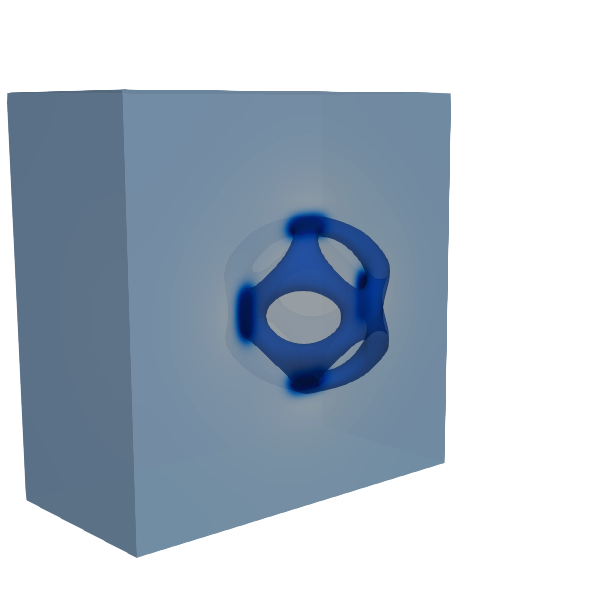}
\includegraphics[angle=-0,width=0.24\textwidth]{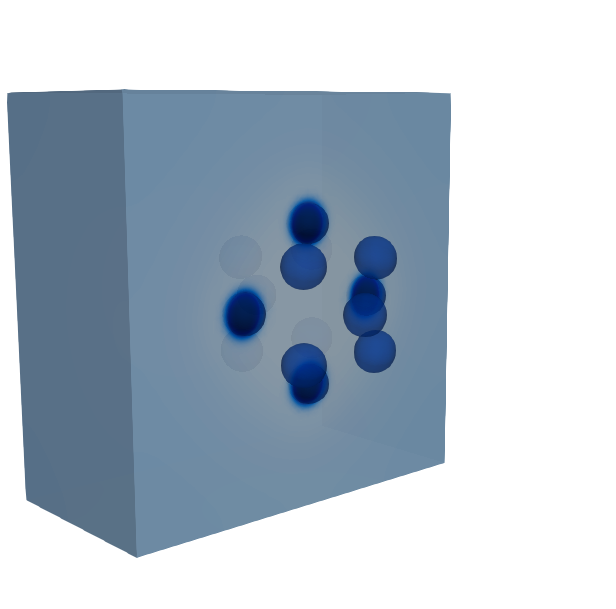}
\includegraphics[angle=-0,width=0.24\textwidth]{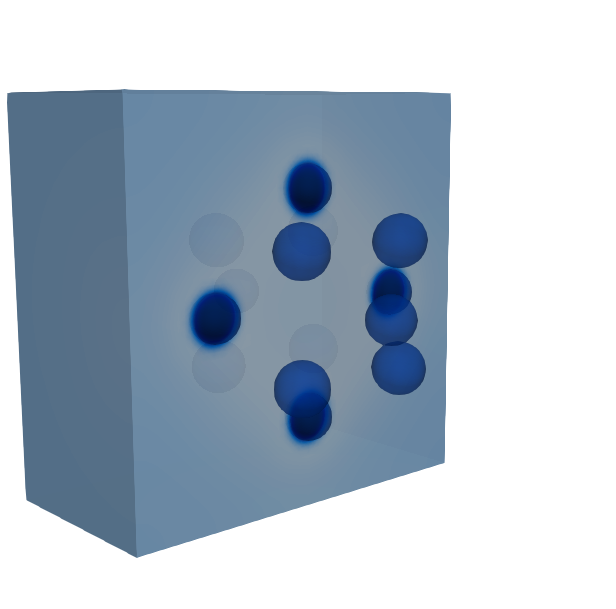}
\caption{($\eps=\frac1{16\pi}$, $\Omega=(0,2)^3$) 
Evolution for $\beta=0.002$, $S_-=0.25$, $S_+ = -4$. 
%We show the solution at times $t=0,0.3,0.5,1,1.5,1.6,2,5$.
\last{We show
the solution at times t = 0, 0.3, 0.5, 1, 1.5, 1.6, 2, 5
where the visualizations in the two rows differ.}
}
\label{fig:cube2_16pi_beta0002_r06}
% ~/hpc_cluster/data/alberta/ool/3d.16pi.square2r06.beta2e-3_025-4 splits into 4
\end{figure}%
}

\last{
For the remaining experiments in this subsection we use $r_c = 0.5$, recall}
\eqref{def:source1}, and consider the domain $\Omega = (0,8)^3$.
%\comm{[This is perhaps more of a comment for us. We derived $S_I$ only in the case $r_c = 1$, whereas here $r_c = 0.5$ is used. On the other hand, using the PF model, $S_I$ is not observed...]} 
In our first experiment, we choose for the initial droplet we choose a rounded
cylinder of total dimension $0.8\times0.4\times0.4$. In particular, it is made
up of two half-balls of radius $0.2$, which are connected with a cylinder
of radius $0.2$ and height $0.4$. For the physical parameters we choose
$\beta=0.1$, $S_-=0.8$, $S_+ = -10$, $m_-=0.2$, $m_+=0.5$,
$\rho_-=0.2$, $\rho_+=0.1$, $L=-1$, and we let $\eps = \frac1{16\pi}$.
The results of our numerical simulation can be seen in
Figure~\ref{fig:c080404}, where we notice a primary pinch-off that then
leads to several secondary pinch-offs.
\begin{figure}
% cp pf16pi_*finalpush9*_0[0-9][0-9]0.png ~/tex/glns/ool/figures && scpp pf16pi_*finalpush9*_0[0-9][0-9]0.png e23:tex/glns/ool/figures
\center
\includegraphics[angle=-0,width=0.3\textwidth]{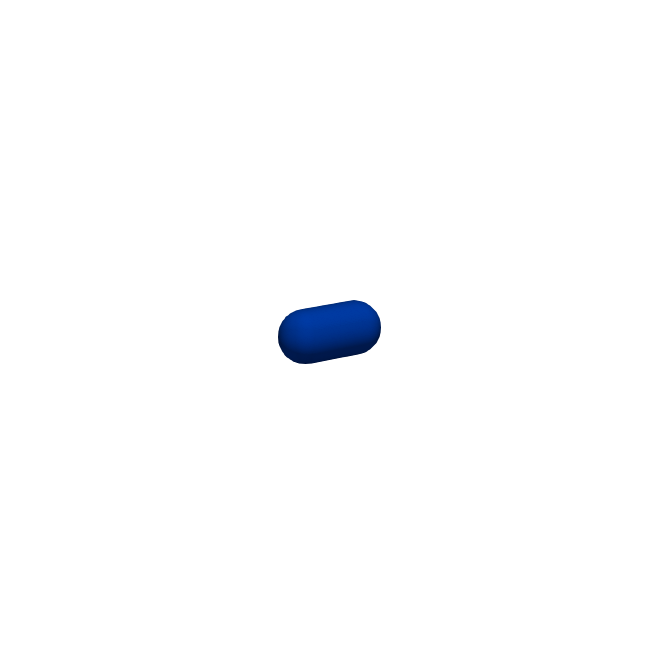}
\includegraphics[angle=-0,width=0.3\textwidth]{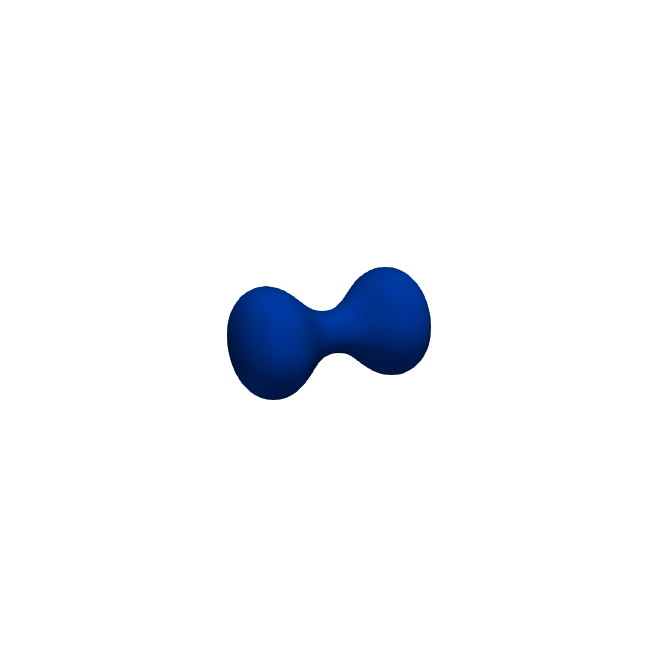}
\includegraphics[angle=-0,width=0.3\textwidth]{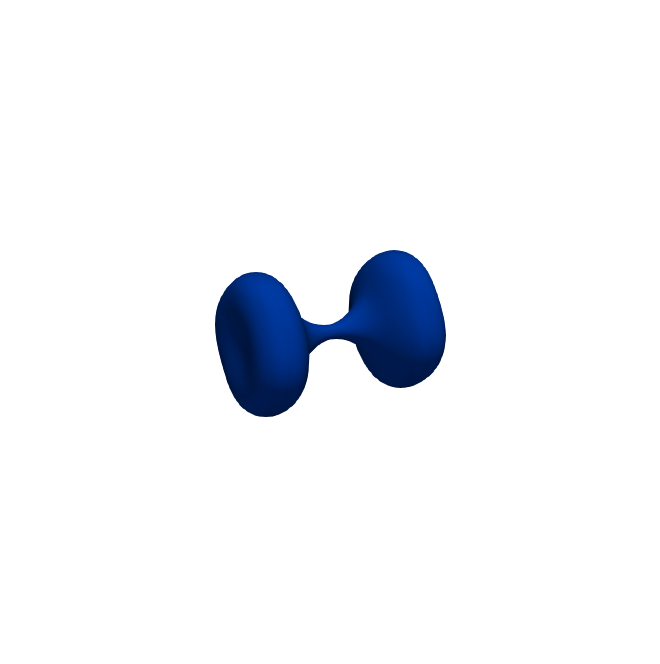}
\includegraphics[angle=-0,width=0.3\textwidth]{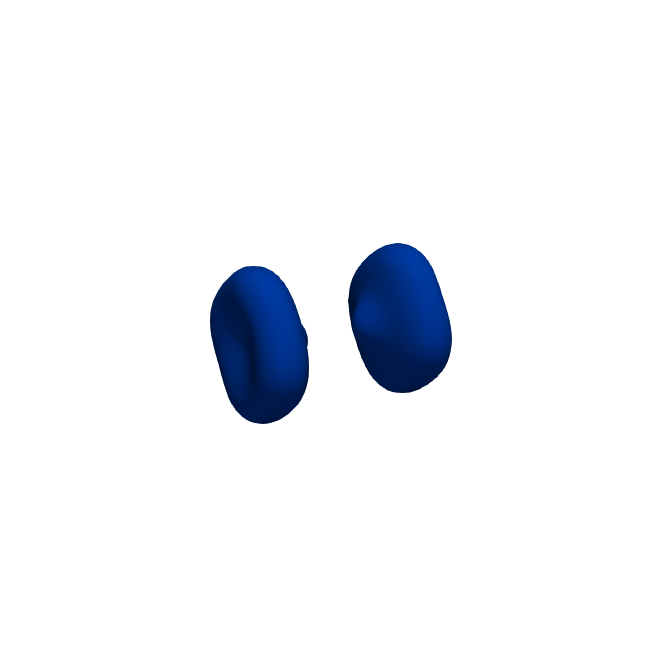}
\includegraphics[angle=-0,width=0.3\textwidth]{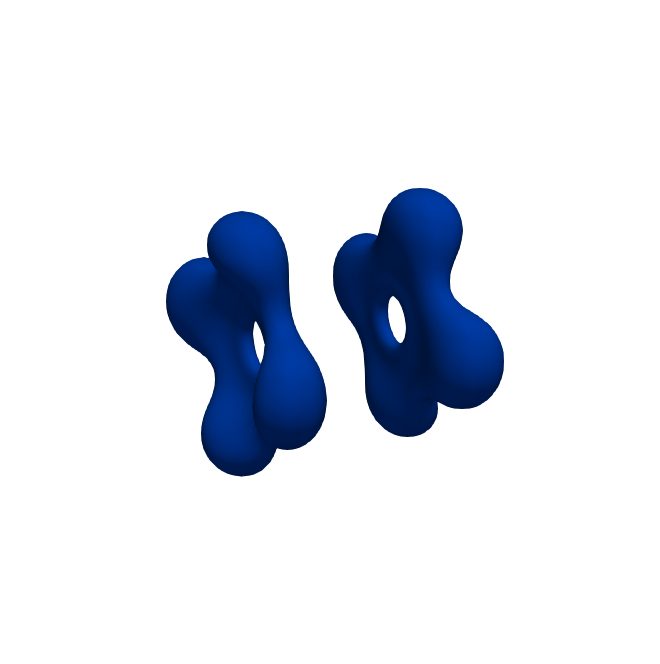}
\includegraphics[angle=-0,width=0.3\textwidth]{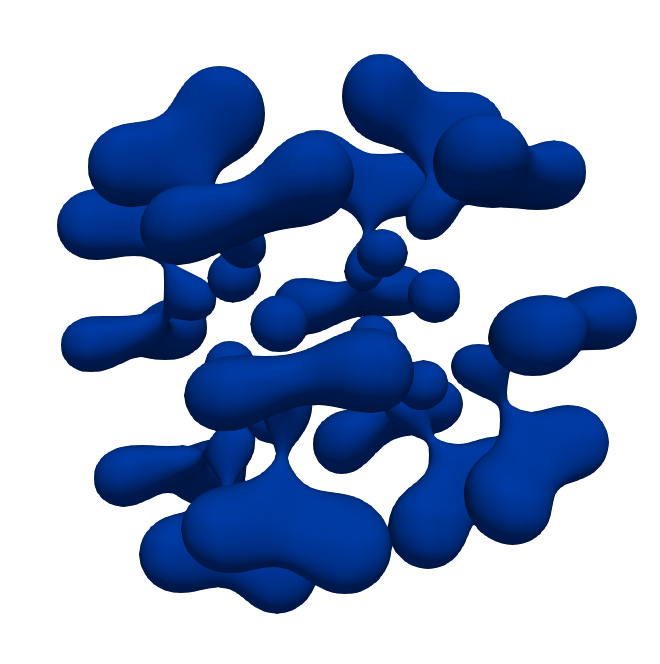}
\caption{($\eps=\frac1{16\pi}$, $\Omega = (0,8)^3$) 
Evolution for $\beta=0.1$, $S_-=0.8$, $S_+ = -10$, $m_-=0.2$, $m_+=0.5$,
$\rho_-=0.2$, $\rho_+=0.1$, $L=-1$.
We show the solution at times $t=0,1,1.3,1.4,2,3.5$. 
}
\label{fig:c080404}
% ~/hpc_cluster/data/alberta/ool/3d.nov16pi.rc05.squared4.r028noise_beta01_08-10_m0205_0201_L-1_cigar080404
\end{figure}%

For the next experiment we only change the aspect ratio
of the initial droplet. In particular, the initial droplet now is a 
a rounded cylinder of total dimension $0.6\times0.4\times0.4$.
%, which is either
%placed at the centre of the domain, or at the point $(2,2,2)^T\in\mathbb R^3$.
The results of this numerical simulation can be seen in
Figure~\ref{fig:c060404}. In Figure~\ref{fig:c060404a} we show the interfaces
at time $t=4.4$ within $\Omega$ from three different points of view. \mod{
We observe that the evolution can be very complex with many splittings and other topological changes. 
Figure~\ref{fig:c060404a} shows that droplets have been formed in the center  and away from the center the evolution is still very complex. We expect that eventually more and more round droplets will form.}
\begin{figure}
% rsync -avz --include="*.png"  --exclude="*" e23:hpc_local/data/alberta/ool/3d.nov16pi.rc05.squared4.r028noise_beta01_08-10_m0205_0201_L-1_cigar060404/ .
% cp pf16pi_c060404*_0[0-9][0-9]0.png ~/tex/glns/ool/figures && scpp pf16pi_c060404*_0[0-9][0-9]0.png e23:tex/glns/ool/figures
\center
\includegraphics[angle=-0,width=0.3\textwidth]{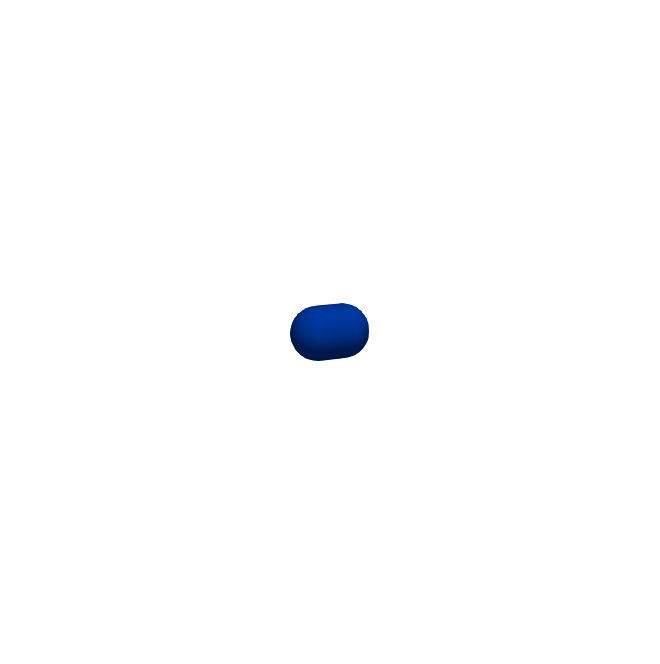}
\includegraphics[angle=-0,width=0.3\textwidth]{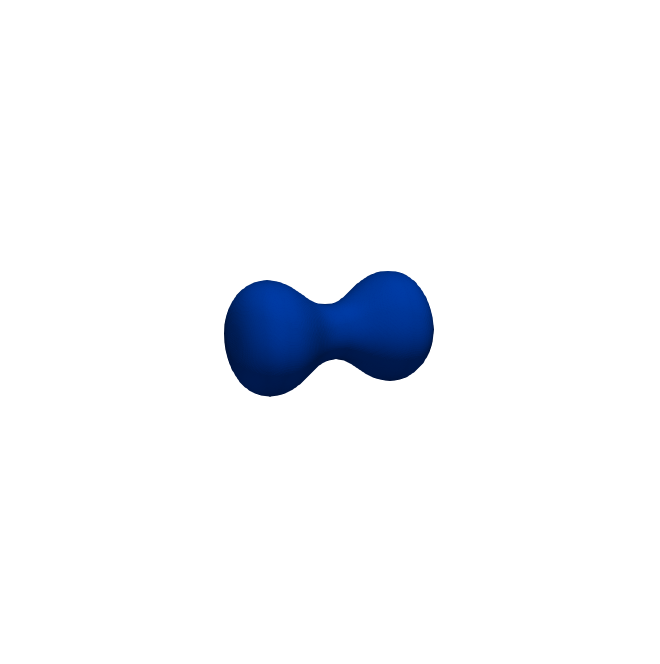}
\includegraphics[angle=-0,width=0.3\textwidth]{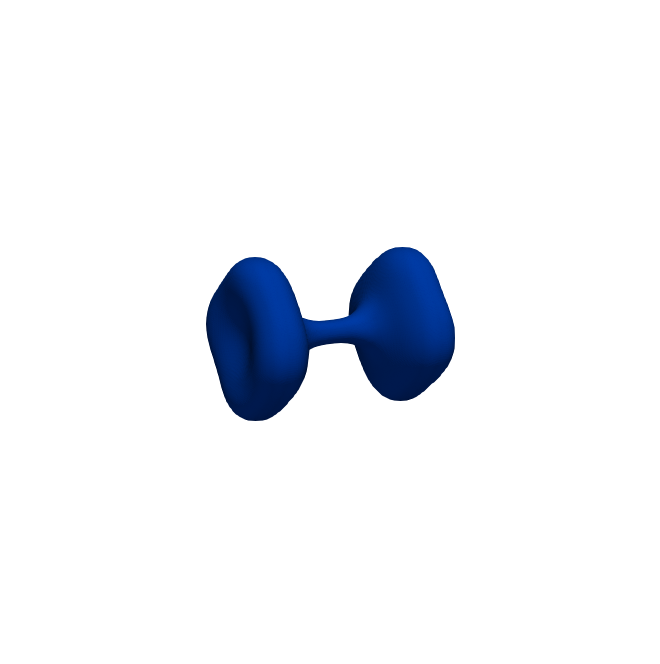}
\includegraphics[angle=-0,width=0.3\textwidth]{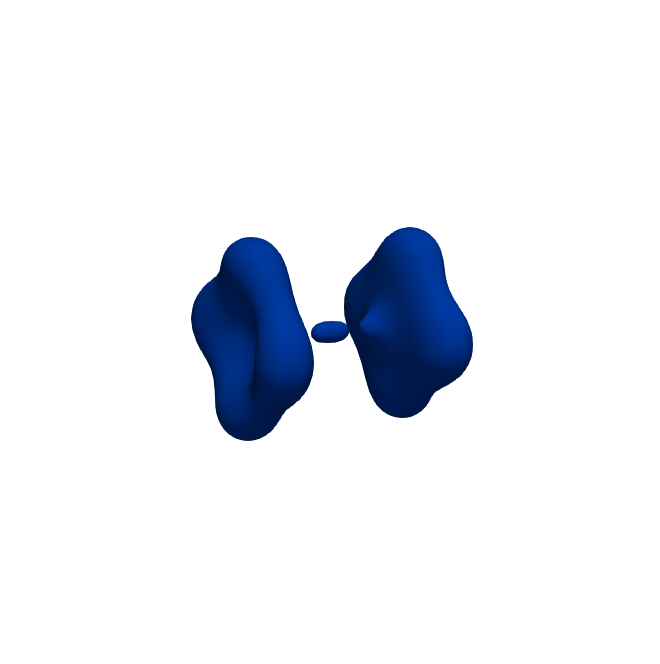}
\includegraphics[angle=-0,width=0.3\textwidth]{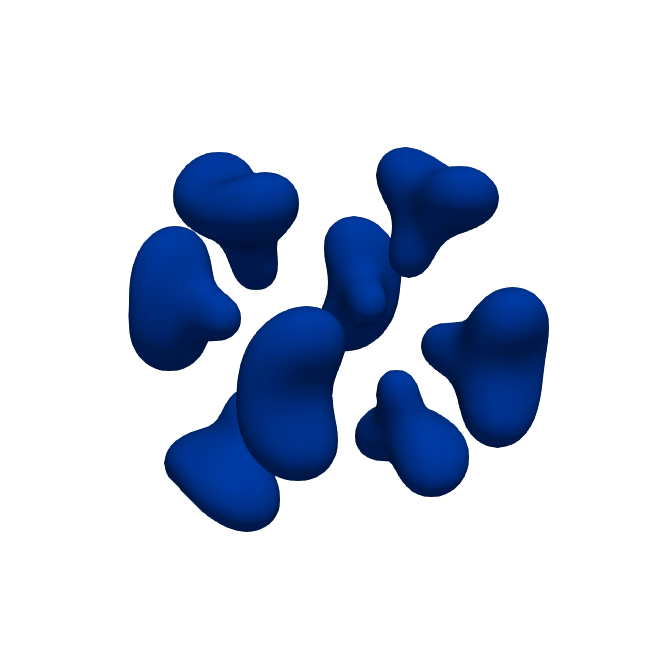}
\includegraphics[angle=-0,width=0.3\textwidth]{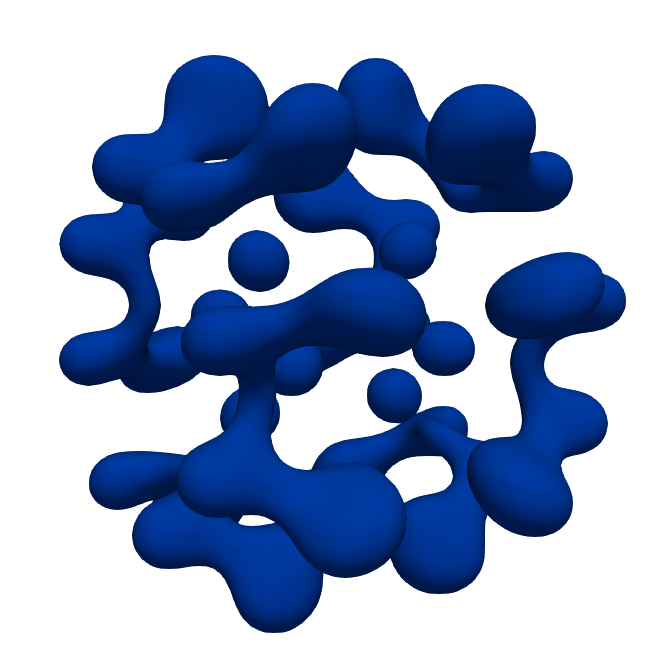}
\caption{($\eps=\frac1{16\pi}$, $\Omega = (0,8)^3$) 
Evolution for $\beta=0.1$, $S_-=0.8$, $S_+ = -10$, $m_-=0.2$, $m_+=0.5$,
$\rho_-=0.2$, $\rho_+=0.1$, $L=-1$.
We show the solution at times $t=0,1.4,1.8,2,3,4$. 
}
\label{fig:c060404}
% ~/hpc_cluster/data/alberta/ool/3d.nov16pi.rc05.squared4.r028noise_beta01_08-10_m0205_0201_L-1_cigar060404
\end{figure}%
\begin{figure}
% cp pf16pi_c060404*_*[xyz].png ~/tex/glns/ool/figures && scpp pf16pi_c060404*_*[xyz].png e23:tex/glns/ool/figures
\center
\includegraphics[angle=-0,width=0.3\textwidth]{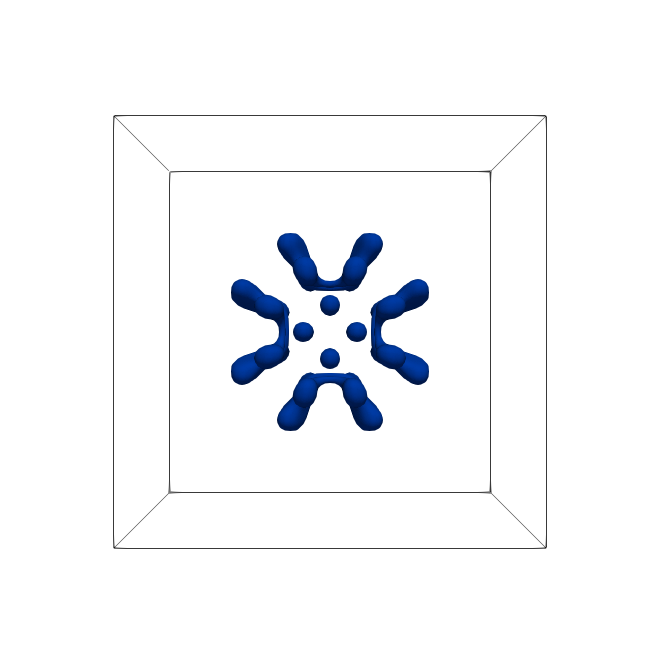}
\includegraphics[angle=-0,width=0.3\textwidth]{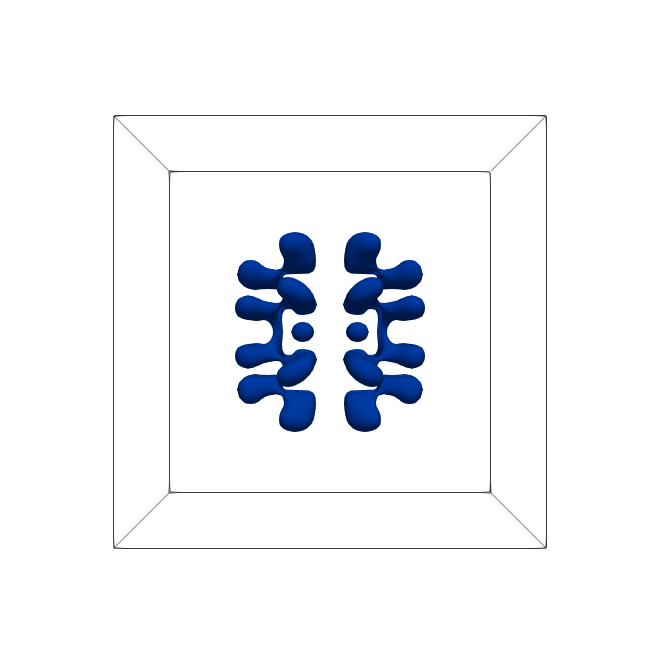}
\includegraphics[angle=-0,width=0.3\textwidth]{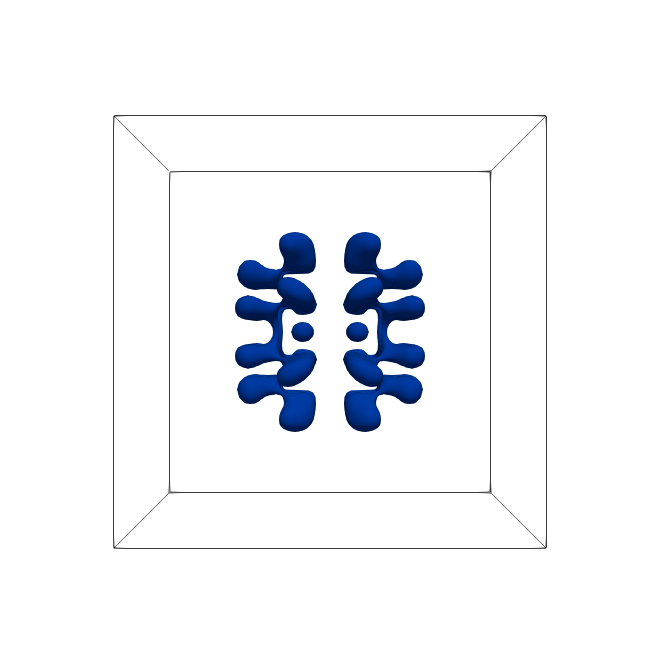}
\caption{The evolution from Figure~\ref{fig:c060404} at time $t=4.4$ viewed
from the front, from the side and from above.
}
\label{fig:c060404a}
% ~/hpc_cluster/data/alberta/ool/3d.nov16pi.rc05.squared4.r028noise_beta01_08-10_m0205_0201_L-1_cigar060404
\end{figure}%
\begin{comment}
\begin{figure}
% cp pf16pi_*finalpush9*_0[0-9][0-9]0.png ~/tex/glns/ool/figures && scpp pf16pi_*finalpush9*_0[0-9][0-9]0.png e23:tex/glns/ool/figures
\center
\includegraphics[angle=-0,width=0.3\textwidth]{figures/pf16pi_finalpush9off_c060404_0000}
\includegraphics[angle=-0,width=0.3\textwidth]{figures/pf16pi_finalpush9off_c060404_0140}
\includegraphics[angle=-0,width=0.3\textwidth]{figures/pf16pi_finalpush9off_c060404_0180}
\includegraphics[angle=-0,width=0.3\textwidth]{figures/pf16pi_finalpush9off_c060404_0200}
\includegraphics[angle=-0,width=0.3\textwidth]{figures/pf16pi_finalpush9off_c060404_0300}
\includegraphics[angle=-0,width=0.3\textwidth]{figures/pf16pi_finalpush9off_c060404_0340}
\caption{($\eps=\frac1{16\pi}$, $\Omega = (0,8)^3$) 
Same as Figure~\ref{fig:c060404}, but the initial blob is centred at 
$(2,2,2)^T$.
We show the solution at times $t=0,1.4,1.8,2,3,3.4$. 
}
\label{fig:c060404off}
% ~/hpc_cluster/data/alberta/ool/3d.nov16pi.rc05.squared4.beta01_08-10_m0205_0201_L-1_cigar060404_offset
\end{figure}%
\end{comment}
%\comm{[@Robert: Here, we should explain why we are unable to produce 3D shells...]}

\section*{\bf Acknowledgements}
\noindent
AS gratefully acknowledge some support
from the MIUR-PRIN Grant 2020F3NCPX ``Mathematics for industry 4.0 (Math4I4)'', from ``MUR GRANT Dipartimento di Eccellenza'' 2023-2027  and from the Alexander von Humboldt Foundation.
Additionally, AS appreciates affiliation with GNAMPA (Gruppo Nazionale per l'Analisi Matematica, la Probabilità e le loro Applicazioni) of INdAM (Istituto Nazionale di Alta Matematica).

%%%%%%%%%%%%%%%%%%%%%%%%%%%%%%%%%%%%%%%%%%%%%%%%%%%%%%%%%%%%%%%%%%%%%%%%
\footnotesize
%\bibliographystyle{abbrv} % We choose the "plain" reference style
%\bibliography{refs} % Entries are in the refs.bib file

\end{document}